\documentclass[9pt,twocolumn,twoside]{article}
\usepackage[paperwidth=8.5in,paperheight=11in,left=0.65in,right=0.65in,top=0.75in,bottom=1in]{geometry}
\usepackage[numbers]{natbib}
\usepackage{amssymb}
\usepackage{hyperref}
\usepackage{xcolor}
\usepackage{authblk}
\usepackage{placeins}

\usepackage{amsthm}
\usepackage{amsmath}             
\usepackage{mathrsfs}            
\usepackage{bbm}                 
\usepackage{mathtools}           
\usepackage{algorithm,algpseudocode}  
\usepackage{nicefrac}            
\usepackage{graphicx}            
\usepackage{verbatim}            
\usepackage{subcaption}          
\usepackage{diagbox}             
\usepackage{booktabs,colortbl}   
\usepackage{cleveref}            
\usepackage{tikz-cd}             
\usepackage{enumitem}            
\usepackage{xparse}              
\usepackage{todonotes}
\usepackage{pgffor}
\usepackage{adjustbox}
\usepackage{longtable}
\usepackage{array}
\graphicspath{{figures/}}

\numberwithin{equation}{section}

\setlist[itemize]{label=\tiny\textbullet}
\setlist[enumerate]{label=\normalfont(\roman*)}

\definecolor{FTChampagnePink}{HTML}{F2DFCE}

\crefname{equation}{equation}{equations}
\crefname{figure}{figure}{figures}
\crefname{section}{section}{sections}
\crefname{subsection}{section}{sections}
\crefname{lem}{lemma}{lemmas}
\crefname{ex}{example}{examples}

\tikzset{
  symbol/.style={
    draw=none,
    every to/.append style={
      edge node={node [sloped, allow upside down, auto=false]{$#1$}}
    }
  }
}

\newtheorem{theorem}{Theorem}            
\newtheorem{lemma}{Lemma}            

\newtheorem{definition}{Definition}  
\newtheorem{example}{Example}        
\newtheorem{remark}{Remark}          

\NewDocumentCommand{\CF}{O{\R} O{}}{%
  \operatorname{CF}_{#2}(#1)%
}

\newcommand{\deff}[1]{\textbf{#1}}
\def\Z{{\mathbb Z}}    
\def\R{{\mathbb R}}    
\def\Q{{\mathbb Q}}    
\def\d{\,\mathrm{d}}   

\def\1{\mathbf{1}}      
\renewcommand{\phi}{\varphi}

\def\cf{\text{CF}}

\def\Int{\text{Int}}

\def\ect{\text{ECT}}
\def\samp{\text{SampEuler}}
\def\im{\text{Im}}
\def\id{\text{Id}}

\NewDocumentCommand{\transform}{s O{\phi} d<> O{\kappa}}{%
  \IfBooleanTF{#1}{
    \operatorname{T}_{#4}%
  }{
    \operatorname{T}_{#4}\left[#2\right]%
    \IfNoValueTF{#3}{}{\left(#3\right)}%
  }%
}

\newcolumntype{L}[1]{>{\raggedright\arraybackslash}p{#1}}

\newcommand{\beginsupplement}{%
  \setcounter{table}{0}
  \renewcommand{\thetable}{S\arabic{table}}%
  \setcounter{figure}{0}
  \renewcommand{\thefigure}{S\arabic{figure}}%
  \setcounter{equation}{0}
  \renewcommand{\theequation}{S\arabic{equation}}%
  \setcounter{theorem}{0}
  \renewcommand{\thetheorem}{S\arabic{theorem}}%
  \setcounter{definition}{0}
  \renewcommand{\thedefinition}{S\arabic{definition}}%
  \setcounter{lemma}{0}
  \renewcommand{\thelemma}{S\arabic{lemma}}%
}

\title{Topological shape transform for thymus structures}

\author[1,2,3,4]{Haochen Yang}
\author[6]{Vadim Lebovici}
\author[7]{Andreas Tarcevski}
\author[7,9]{Liliana Tchernev}
\author[8]{Saulius Zuklys}
\author[7,8,9,10]{Georg A. Holl\"{a}nder\thanks{To whom correspondence should be addressed. E-mail: georg.hollander@paediatrics.ox.ac.uk}}
\author[1,2]{Helen M. Byrne\thanks{Corresponding author. E-mail: helen.byrne@maths.ox.ac.uk}}
\author[1,3,4,5]{Heather A. Harrington\thanks{Corresponding author. E-mail: Harrington@mpi-cbg.de}}

\affil[1]{Mathematical Institute, University of Oxford, Oxford OX2 6GG, United Kingdom}
\affil[2]{Ludwig Institute for Cancer Research, Nuffield Department of Medicine, Oxford OX3 7DQ, United Kingdom}
\affil[3]{Centre for Systems Biology Dresden, Dresden 01307, Germany}
\affil[4]{Max Planck Institute of Molecular Cell Biology and Genetics, Dresden 01307, Germany}
\affil[5]{Technische Universit\"{a}t Dresden, Dresden 01062, Germany}
\affil[6]{Sorbonne University, Institut de Math\'{e}matiques de Jussieu-Paris Rive Gauche, Paris 75005, France}
\affil[7]{Department of Paediatrics, University of Oxford, Oxford OX3 9DU, United Kingdom}
\affil[8]{Paediatric Immunology, Department of Biomedicine, University of Basel, Basel 4058, Switzerland}
\affil[9]{Botnar Institute of Immune Engineering, Basel 4052, Switzerland}
\affil[10]{Department of Biosystems Science and Engineering, ETH Zurich, Basel 4056, Switzerland}

\date{}

\begin{document}

\maketitle

\begin{abstract}
The Euler characteristic transform (ECT) is an emerging and powerful framework within topological data analysis for quantifying the geometry of shape. The applicability of ECT has been limited due to its sensitivity to noisy data. Here, we introduce SampEuler, a novel ECT-based shape descriptor designed to achieve enhanced robustness to perturbations. We provide a theoretical analysis establishing the stability of SampEuler and validate these properties empirically through pairwise similarity analyses on a benchmark dataset and showcase it on a thymus dataset. The thymus is a primary lymphoid organ that is essential for the maturation and selection of self-tolerant T cells, and within the thymus, thymic epithelial cells are organized in complex three-dimensional architectures, yet the principles governing their formation, functional organization, and remodeling during age-related involution remain poorly understood. Addressing these questions requires robust and informative shape descriptors capable of capturing architectural changes across developmental stages. We introduce and apply SampEuler to a newly generated two-dimensional imaging dataset of mouse thymi spanning young adulthood and advanced age groups, where SampEuler outperforms both persistent homology–based methods and deep learning models in detecting localized morphological differences associated with aging. To facilitate interpretation, we develop a vectorization and visualization framework for SampEuler, which preserves rich morphological information and enables identification of structural features that distinguish thymi across age groups. Collectively, our results demonstrate that SampEuler provides a robust and interpretable approach for quantifying thymic architecture, revealing age-dependent structural changes and providing new insights into thymic involution.
\end{abstract}

\noindent\textbf{Keywords:} topological data analysis $|$ Euler characteristic transform $|$ thymus $|$ involution $|$ geometric morphometrics


\section*{Introduction}

Morphological information is fundamental to understanding how biological structure relates to function across tissues, organs, and organisms~\cite{betge2022drug, rezzani2014thymus, freemont2007morphology}. Therefore, quantifying geometric variation can provide valuable insights into developmental processes, disease states, and biological function~\cite{hale2024cellular}. The thymus is an essential organ of the adaptive immune system responsible for the maturation, selection, and differentiation of T cells \cite{immungen}. Within the confines of the thymic microenvironment, T cells are instructed by haematopoietic antigen presenting cells (APCs), thymic epithelial cells (TECs) and other stromal cells to establish a diverse, yet self-tolerant T cell receptor (TCR) repertoire - a process known as thymopoiesis \cite{gillThymicGenerationRegeneration2003}.  With age, the thymus undergoes involution characterized by reduced total cellularity, disruption of the epithelial network architecture, and diminished naive T cell output \cite{goronzyImmuneAgingAutoimmunity2012, elife}. Paralleling thymic involution and reduced thymopoietic activity, the TCR repertoire becomes constricted with age \cite{nikolich2018twilight}. Together, this leads to an increased susceptibility to viral infections, due to the immune system’s diminishing capacity to mount effective responses to de-novo antigens, heightened autoimmunity, and diminished surveillance of malignant cells~\cite{liang2022age, dengSinglecellAnalysisHuman2025, elyahu2021thymus}. Simultaneously, with age, the thymus undergoes significant architectural remodeling, impacting both cortical TECs (cTECs) and medullary TECs (mTECs) exhibiting structural changes \cite{rezzani2014thymus, venablesDynamicChangesEpithelial2019a}. To better understand these structural changes, precise mathematical descriptors are needed to quantify morphological alterations at the cellular level. Previous image-based studies have quantified thymic architecture through separate measurements, such as the cortico-medullary ratio and medullary volume~\cite{irla2013three}, the perivascular-space fraction~\cite{flores1999analysis}, interfacial area and branching~\cite{lagou2026morphometric}, and single-cell shape~\cite{venablesDynamicChangesEpithelial2019a}. Measured independently, these indices each quantify only one aspect of morphology and may vary independently across samples, motivating a single descriptor that captures overall shape change. 
%

Recent advances in topological data analysis (TDA), a field dedicated to extracting geometric and topological insights from spatial data, have led to the introduction of persistent homology \cite{edelsbrunner2002topological, cohen2005stability} (tracking how holes and loops in a shape appear and vanish across scales) and to the development of the Euler Characteristic Transform (ECT) \cite{turner2014persistent} (a descriptor that scans a shape from many directions). Extensions to ECT and refined transforms have been devised to capture aspects of shape analysis \cite{crawford2020predicting, kirveslahti2024digital, roell2023differentiable}, and theoretical studies have established important properties of the ECT, including its invertibility \cite{schapira1995tomography,curry2022many, ghrist2018persistent} and the existence of stability bounds \cite{skraba2020wasserstein, marsh2023stability} (small changes in a shape cause only small changes in its descriptor). The invertibility of ECT means that it captures all geometric and topological information of the input shape and is highly effective at capturing morphological differences between shapes. Despite the stability results, the ECT remains extremely sensitive to small perturbations and misalignments in the input shape, undermining its practical robustness. Although several alignment procedures have been proposed to address these issues \cite{meng2024randomness, wang2021statistical, kirveslahti2024digital}, they often come at a high computational cost. Recent theoretical advances \cite{curry2022many} have effectively resolved the sensitivity to misalignments, but the isometry invariant (unchanged by rotation, reflection, or translation) algorithms they have inspired \cite{marsh2022detecting} suffer from unsatisfactory levels of information loss.

Motivated by the limitations of existing isometry invariant algorithms, we introduce SampEuler, a discrete construction based on the ECT, which comes with theoretical guarantees about the information it captures and its stability. Building on SampEuler, we introduce a vectorization and visualization algorithm for the ECT. We first apply SampEuler and its vectorization to synthetic tree structures and show that they outperform current topological shape transforms. We then benchmark against conventional shape analysis methods and TDA methods using the MPEG7 dataset.

Finally, we apply SampEuler to a segmentation masks of the thymic epithelial scaffold covering different age groups, to characterize age-related structural changes. To the best of our knowledge, this is the first complete, quantitative characterisation of how the geometry and topology of the thymic epithelial network change between young and old thymi; because one descriptor captures the full morphological change at once, rather than separate indices tracked one by one, it also lets us study how this change co-varies with shifts in cellular composition across the life trajectory. In our classification study, we benchmark SampEuler against state-of-the-art TDA methods, including various vectorizations of persistent homology, as well as deep learning methods, demonstrating that SampEuler achieves higher classification accuracy in significantly less computation time. We interpret the classification results based on the vectorization of SampEuler in relation to underlying thymic structural differences. In particular, we relate these shape changes to the spatial distribution of T cells.

\section*{Description of Datasets}
We evaluate SampEuler on three datasets: synthetic structures for the method 
validation, established shape benchmarks for comparison with existing 
approaches, and mouse thymus images for biological applications.

\subsection*{Synthetic Networks}
 We produce simple geometric simplicial complex examples as networks consisting of three edges of fixed lengths emanating from a common origin that are uniformly discretized into points connected by edges in order. We perturb the points and the edges that connect them by adding Gaussian noise. This creates a controlled setting within which to evaluate the performance of ECT-based methods under perfect alignment versus rotational perturbations. Two representations of two groups of simplices are shown in \Cref{fig: toy example}\textit{B}.  Details of the dataset can be found in the \textit{SI Appendix}.

\subsection*{MPEG-7 Shape Library}
We use the published image dataset from previous shape descriptor studies \cite{reininghaus2015stable, kusano2016persistence, carriere2017sliced, anirudh2016riemannian, le2018persistence, van2025discrete, turner2014persistent, marsh2022detecting, bai2009integrating, wang2014bag, adams2017persistence, hofer2019learning}. The dataset consists of 70 groups of 2D binary images, each group containing silhouettes of objects from the same category. For example, one group contains 20 images of horse silhouettes. Although all the underlying objects are horses and share similar geometry, they are not identical; in many cases, a 2D isometry—such as a slight rotation—is applied to the object before the image is captured so that the same side of the horse is photographed from different angles.

We preprocess the dataset by first padding each image with black pixels to achieve a consistent size of $1126$ pixels $\times1126$ pixels. Next, we denoise the images by retaining only the largest connected component of white pixels—treating any additional white pixels as false positives and converting them to black. The processed images are then used for further analysis. The preprocessing details are included in the \textit{SI Appendix}.
\subsection*{Mouse Thymus Dataset}

Mouse thymi (n=7) were processed as described in 'Materials and Methods' (see below). The boundaries of the thymic epithelial cell (TEC) network were segmented using MATLAB's Image Processing Toolbox: multilevel thresholding 
(\texttt{multithresh}) and skeletonization (\texttt{bwskel}) were applied 
to the cytokeratin 8 (K8) and cytokeratin 14 (K14) TEC markers to generate 2D binary images for all 7 thymi and both markers (values of 1 where the K8/K14 staining is present, 0 where it is absent). A 200 x 200 pixel lattice was superimposed on the 2D images to generate subimages, for subsequent TDA.

\section*{Results}

We introduce SampEuler to capture shape information in a way that is robust to noise and overcomes the sensitivity of ECT (\Cref{fig: toy example}). We show that the proposed approach achieves higher accuracy with less 
computational cost than existing methods, including persistent homology 
vectorizations, ECT-based approaches, and deep learning models (\Cref{fig:main_2}). Employing SampEuler on the images of mouse thymi, we achieve age classification based on the shape of local thymic structures (\Cref{fig:main_3}). Comparing with the distribution of T cells, we identify relationships between thymic structure and cellular composition at different ages (\Cref{fig:main_4}).

\subsection*{From Data to Topological Transforms} 

Within TDA, the ECT is a computable shape descriptor that has been extended in numerous ways and successfully applied to biological datasets \cite{crawford2020predicting, roell2023differentiable, wang2021statistical, wang2022gpu, kirveslahti2024digital}. Such data comes in the form of $n$-dimensional images or discretized boundaries of segmented structures; these can be represented as simplicial complexes comprising points, edges, faces, and higher-dimensional simplices. One of the simplest computable topological invariants, the Euler characteristic, is given by the alternating sum of the number of $n$-dimensional simplices in the simplicial complex $K$: $\chi(K) = \sum_{i=0}^n (-1)^i \#\text{ of $i$-dimensional simplices in }K,$ where $n$ is the maximal dimension of all simplices in $K$. To capture geometric information of the complex, Turner et al. proposed sweeping a (hyper)plane through the shape progressively in time ($t$) in a given direction ($v$) and computing the Euler characteristic ($\chi$), which outputs an integer-valued curve for the particular direction (\Cref{ECC_def}). By sweeping through all directions, the Euler characteristic transform (ECT) is the correspondence between directions and curves (\Cref{ECT_def}). The ECT is invertible, meaning one can reconstruct the original shape and thereby derive any geometric information captured by other shape descriptors \cite{schapira1995tomography, curry2022many, ghrist2018persistent}.
\begin{equation}
\label{ECC_def}
    \text{ECC}_K^v: \mathbb{R}\rightarrow\mathbb{Z}; t\mapsto \chi(K\cap\{x\in\mathbb{R}^n: x\cdot v\leq t\}).
\end{equation}
\begin{equation}
\label{ECT_def}
    \text{ECT}(K): S^{n-1}\times\mathbb{R}\rightarrow \mathbb{Z}; (v,t)\mapsto \text{ECC}_K^v(t).
\end{equation}
Since the ECT is sensitive to the orientation of the complex, DETECT (\cite{marsh2022detecting}) was proposed to smooth, normalize, and average across all $\text{ECC}_K^v$: 
\begin{align}
    &\text{DETECT}(K): [-a, a] \rightarrow \mathbb{R};\\
    & c \mapsto \int_{S^{n-1}}\left(\int_{-a}^c \text{ECC}_K^v(t) - (\frac{1}{2a}\int_{-a}^a \text{ECC}_K^v(s) \text{d}s) \text{d}t \right) \text{d}v
\end{align}

We apply ECT and DETECT to the synthetic networks. To visualize the discriminative power of each method, we embed the resulting shape descriptors in two dimensions using multidimensional scaling (MDS) \cite{borg2005modern, kruskal1964multidimensional} (which places shapes as points so that more similar shapes are closer together). The standard ECT distinguishes two classes of perfectly aligned shapes, whereas DETECT fails (\Cref{fig: toy example}\textit{C-D}, left), illustrating the information loss introduced during the averaging step. Indeed, both complexes yield identical DETECT transforms.

When the shapes are perturbed randomly to simulate the misalignment in applications, both methods fail to distinguish the two classes of complexes (\Cref{fig: toy example}\textit{C–D}, right).

\subsection*{Introduction to SampEuler}
To overcome the sensitivity of ECT to perturbations on input shapes and the information loss of DETECT, we propose to record only the output curves and not the direction $v$ in ECC. This shape transform is isometry-invariant. To achieve this mathematically, we consider the ECT pushforward measure of $K$, which can be thought of as the probability law of ECCs of $K$ for a given direction chosen uniformly at random. The ECT pushforward measure is proven to be injective up to rotations and reflections (different shapes always give different measures)~\cite{curry2022many}.

By convention, the shape of a given complex is the equivalence class of the complex up to rotations, reflections, and translations (that is, two structures have the same shape if one is a rotated, reflected, or shifted copy of the other). We centre all complexes by translating their centres of mass to the origin to filter the translation information. On these centred complexes, the ECT pushforward measure serves as a complete shape descriptor: two complexes have the same ECT pushforward measure if and only if they have the same shape. The ECT pushforward measure is a probability law on the function space (a probability distribution over the resulting curves), and is therefore hard to compute in practice. Based on the ECT pushforward measure, we propose an isometry-invariant shape descriptor, \textbf{SampEuler}. As shown in \Cref{fig: toy example}\textit{A}, for an input shape $K$, SampEuler samples curves from the ECT pushforward measure by randomly sampling directions and computing their corresponding ECCs, and outputs an empirical measure (the sampled approximation of this distribution) formed by these sampled curves. We show that the SampEuler converges to the ECT pushforward measure as the discretization is refined. Equipped with the Wasserstein distance~\cite{Kantorovich1960MMOPP} (a standard metric of how different two distributions are), it induces a continuous map from the input complex to its SampEuler representation (this stability means that small perturbations of a shape produce only small changes in its descriptor). Together, this convergence and stability give confidence that the descriptor is accurate: because SampEuler converges to the ECT pushforward measure, which is invariant under rotations and reflections, it inherits this invariance and so depends only on the shape; and because it is stable, noise or small imperfections in the data alter the descriptor only slightly, so that two shapes can still be compared reliably. We also introduce a vectorization and visualization method for SampEuler. For small datasets, the vectorized representation yields a low-dimensional feature space that supports effective training of standard machine-learning models and, when paired with interpretation techniques, allows direct geometric interpretation. For larger or geometrically complex datasets, using the full SampEuler representation preserves complete geometric information with theoretical guarantees, avoiding information loss from dimensionality reduction. Details and proofs of SampEuler convergence and continuity are included in the \textit{SI Appendix}.

We apply SampEuler and its vectorization to the synthetic networks, 
visualizing the results using MDS (\Cref{fig: toy example}\textit{E-F}). In both cases, SampEuler and its vectorization correctly distinguish the two 
classes (see \textit{SI Appendix} for more details), demonstrating that they are both isometry invariant and informative.

\begin{figure*}[p]
\centering
\includegraphics[width=14cm]{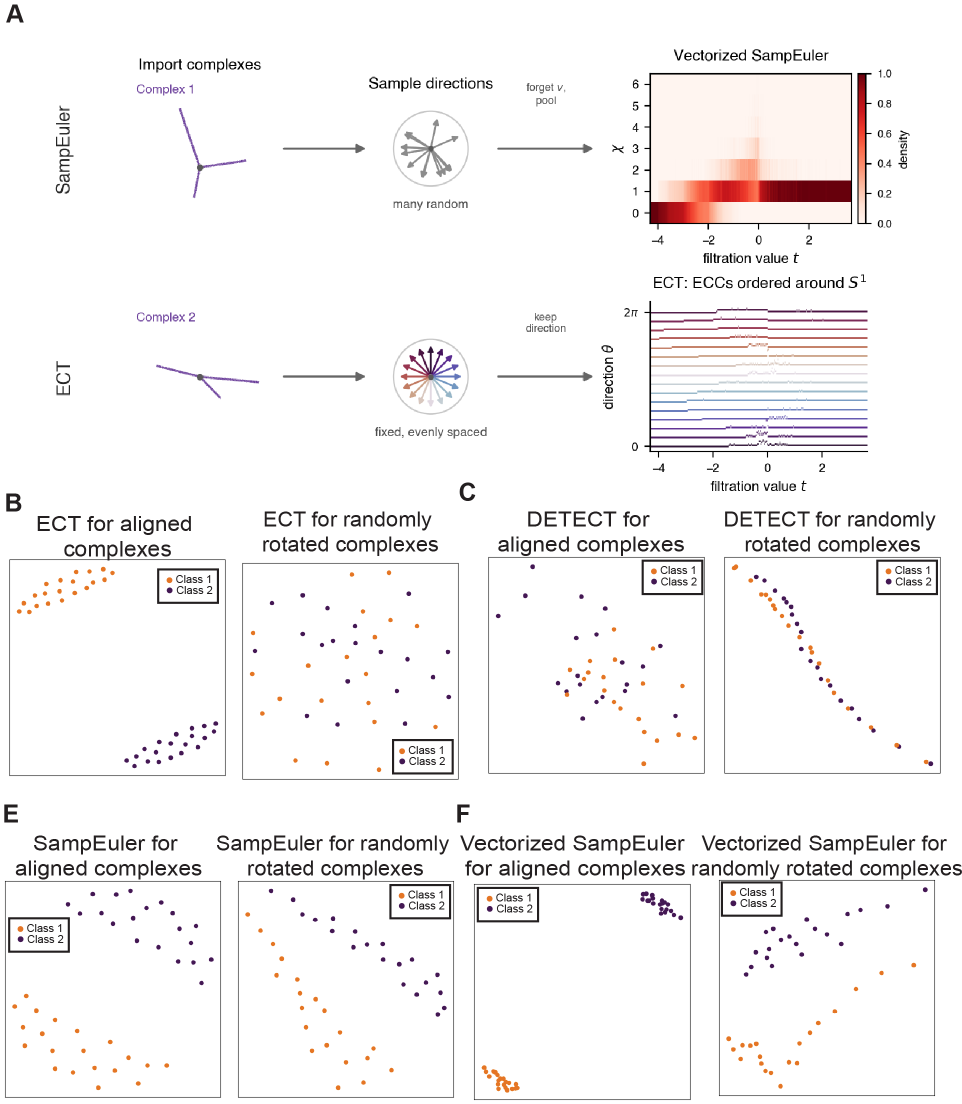}
\caption{MOTIVATION FOR NEW PUSHFORWARD MEASURE on ECT. We apply ECT, DETECT, SampEuler, and vectorization of SampEuler to aligned toy simplicial complexes (left) and randomly rotated simplicial complexes (right). (\textit{A}) Illustration and comparison of the ECT and SampEuler computation pipelines on two example complexes (each consisting of three edges of fixed lengths sharing a common origin, randomly rotated with Gaussian noise added along the edges). These two complexes are representative members of the two shape classes that are classified in the subsequent panels' experiments (\textit{B}--\textit{E}). For SampEuler (top row), we sample a large set of random directions, compute one integer-valued Euler characteristic curve (ECC) per direction, then discard the direction labels and pool the curves; its vectorization is obtained by discretizing the $(\chi, t)$ plane into a grid and computing, for each cell, the density of curves passing through it, giving an approximation of the distribution of the sampled ECCs. Because the directions are forgotten, the resulting descriptor is invariant to rotation. For ECT (bottom row), the directions are instead fixed and evenly spaced on $S^1$, and the resulting ECCs are retained in their angular order; keeping this direction information makes the descriptor depend on the orientation of the complex, so randomly rotating the input scrambles the ordered curves and degrades class separation (panel \textit{B}).  (\textit{B}) Two-dimensional MDS embedding (a projection that preserves pairwise distances between shapes) of ECT shape descriptors. The results of two distinct classes can be separated when all complexes are aligned. However, in the rotated case, ECT results are heavily affected by the rotations and fail to distinguish two classes. (\textit{C}) MDS embedding of DETECT shape descriptors. DETECT fails to distinguish two classes of shapes in both experiments due to information loss. (\textit{D}) MDS embedding of SampEuler shape descriptors. In both experiments, the two classes are separated correctly. (\textit{E}) MDS embedding of vectorized SampEuler shape descriptors. In both experiments, the two classes are separated correctly.  }
\label{fig: toy example}
\end{figure*}

\subsection*{SampEuler Outperforms Topological Statistics}
We apply SampEuler and its vectorization with Support Vector Machine (SVM) to classify the MPEG-7 dataset and benchmark against other methods. Details of image preprocessing and the experiments can be found in Materials and Methods.

We benchmark SampEuler against standard TDA methods that use the 10-class subset of MPEG-7: Persistence Scale Space Kernel \cite{reininghaus2015stable}, Persistence Weighted Gaussian Kernel \cite{kusano2016persistence}, Sliced Wasserstein Kernel \cite{carriere2017sliced}, Tangent Vector Representation with Gaussian Kernel \cite{anirudh2016riemannian}, Persistence Fisher Kernel \cite{le2018persistence}, and QUPID \cite{van2025discrete}. As shown in \Cref{fig:main_2}\textit{B}, SampEuler outperforms conventional TDA methods. 

We run the classification task over the whole MPEG7 dataset to compare with other methods (\Cref{fig:main_2}): Skeleton Paths \cite{bai2009integrating}, Bag of Contours \cite{wang2014bag}, Persistence Image + Neural Network (PI+NN) \cite{adams2017persistence}, Persistence Codebook + Neural Network (PC+NN) \cite{hofer2019learning}. SampEuler and its vectorization exhibit consistent performance; however, they are outperformed by dedicated 2D image recognition algorithms and TDA with deep learning approaches. We remark that topological methods offer more interpretability than deep learning, whose decision processes can be comparatively opaque.

\begin{figure}[htbp]
  \centering
  \includegraphics[width=\linewidth]{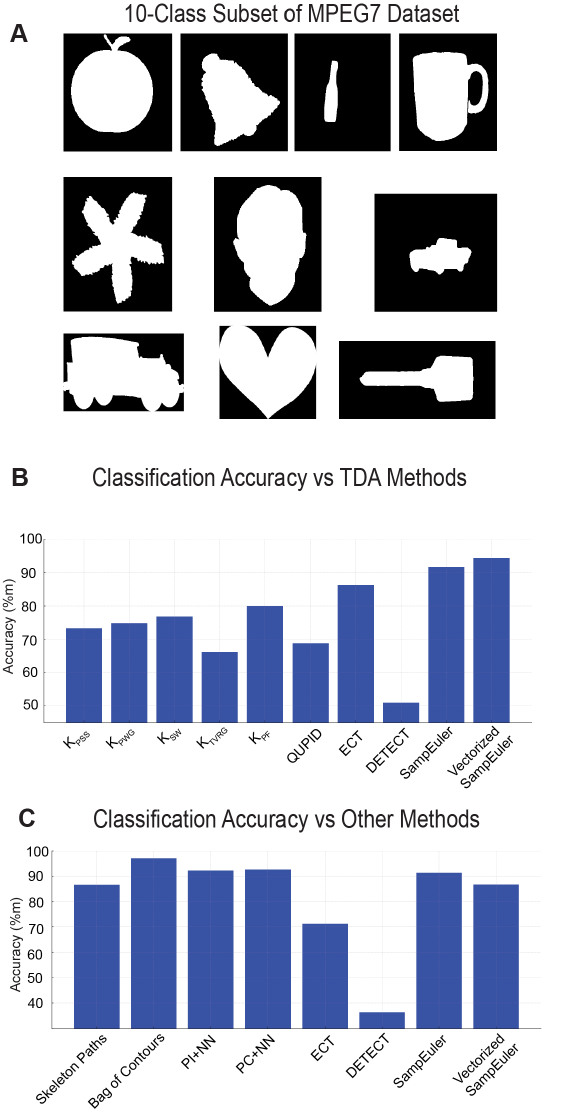}
  \caption{CLASSIFICATION STUDY on MPEG7 DATASET. We compare SampEuler with conventional shape analysis methods using the MPEG7 dataset \cite{ralph1999mpeg7CEShape1}. (\textit{A}) Examples of the MPEG7 dataset images, one sample from each of the 10-class subset used in previous studies \cite{le2018persistence, le2019tree}. (\textit{B}) The barplot of the accuracy of each method in the classification task of the 10-class subset of MPEG7. We denote Persistence Scale Space Kernel \cite{reininghaus2015stable} by $\mathcal{K}_{PSS}$, Persistence Weighted Gaussian Kernel \cite{kusano2016persistence} by $\mathcal{K}_{PWG}$, Sliced Wasserstein Kernel \cite{carriere2017sliced} by $\mathcal{K}_{SW}$, Tangent Vector Representation with Gaussian Kernel \cite{anirudh2016riemannian} by $\mathcal{K}_{TVRG}$, Persistence Fisher Kernel \cite{le2018persistence} by $\mathcal{K}_{PF}$. (\textit{C}) The barplot of the accuracy of each method in the classification task of the whole MPEG7 dataset (70 classes). We use methods: Skeleton Paths \cite{bai2009integrating}, Bag of Contours \cite{wang2014bag}, Persistence Image + Neural Network (PI+NN) \cite{adams2017persistence}, Persistence Codebook + Neural Network (PC+NN) \cite{hofer2019learning}, ECT, DETECT, SampEuler and vectorization of SampEuler.}
  \label{fig:main_2}
\end{figure}

To understand the effect of discretizing ECT and ECT pushforward in practice, we use the MPEG-7 dataset and vary the number of filtration values (the levels at which the shape is progressively scanned) sampled along each curve and the total number of curves sampled, and perform classification tasks with SVM. Experimental details and full results can be found in \textit{SI Appendix}. We report that if we sample more than 100 directions and 1000 points along each direction, we obtain high accuracy, and the additional information obtained for the MPEG7 dataset is limited. Moreover, leveraging the optimised implementation in Eucalc \cite{lebovici2024efficient} enables efficient computation of these descriptors, making the above parameter choices practical in routine use.  

\subsection*{SampEuler Quantifies Ageing Impact on Thymic Structure}
To assess how ageing affects local architecture of TEC, we sub-sampled  quadrants of size $200\times 200$ pixels in representative areas of the cortex (i.e. K8 segmentation mask) and medulla (K14 segmentation mask). We sampled 20 quadrants per thymic lobe, resulting in a total of 40 quadrants for young and 100 for old. They were used as input to a supervised, binary classification of young versus old (\Cref{fig:main_3}\textit{A}). For this resolution, we compute $360$ directions and $300$ points along each direction for all ECT-based methods.

We benchmark our proposed methods against the Persistence image \cite{adams2017persistence} with the signed distance function, AlexNet \cite{krizhevsky2012imagenet}, Visual Transformer \cite{dosovitskiy2020image}, MobileNet V2 \cite{sandler2018mobilenetv2}, EfficientNetB0  \cite{tan2019efficientnet}, ECT and DETECT to verify that they accurately quantify local thymic structure. Details of the experiments can be found in \textit{SI Appendix}. The results are shown in \Cref{fig:main_3}\textit{B-C}. SampEuler and its vectorization give better accuracy with shorter run times, demonstrating their ability to efficiently detect age-related structural differences within the cortex or medulla.

To interpret geometric differences between age groups, we use the vectorization of SampEuler to visualize Euler characteristic curves across all directions. We apply the SHAP algorithm (which attributes a classifier's decision to each input feature) to identify which regions of this image are most important for classifying each age group. We repeat training 50 times, retaining trials with accuracy above 80\% as valid. For each age group, we average SHAP values across test samples from all valid trials to obtain a final feature importance map (\Cref{fig:main_3}\textit{D}). These maps reveal the distribution of Euler characteristic curves that the classifier recognizes as typical for each age group: regions of high importance indicate where curves consistently lie, effectively summarizing each group's typical shape signature.

In the cortical region (K8) of the thymus, the importance map (\Cref{fig:main_3}\textit{D}) indicates, roughly, how the classifier reads the typical Euler characteristic curve for each age group: for old quadrants it points to a strong, steep decrease in the Euler characteristic to a much lower value, whereas for young quadrants the decrease is gentler and ends at a higher value. Since the Euler characteristic for the images is computed from their corresponding 2D cubical complexes (the image represented as a grid of cells), we have that

    $$\text{Euler characteristic } = \#(\text{connected components}) - \#(\text{loops}).$$
Here, a \emph{connected component} is a separate piece of the epithelial network, and a \emph{loop} is a ring of tissue enclosing a gap; formally, their counts are the zeroth and first Betti numbers $\beta_0$ and $\beta_1$ of the cubical complex (the ranks of its homology groups). An Euler characteristic curve (ECC) is built by sweeping across the image in a fixed direction and recording, at each step, the Euler characteristic of the part scanned so far. Its value therefore tracks the running balance between connected pieces and loops as the sweep proceeds.
Because the Euler characteristic equals the number of connected components minus the number of loops, a curve that falls to a lower value reflects a greater number of loops over separate pieces. The strongly decreasing curve that the classifier reads as typical of old cortex (\Cref{fig:main_3}\textit{D}) therefore points to more loops there than in young, and the steepness of that decrease suggests these loops are more densely distributed.

In the medulla (K14) of the thymus, the two age groups show no major difference in their overall Euler characteristic trends; the main distinction is that old quadrants are far more variable, with more diverse curve shapes than young. We confirm these interpretations with independent methods (see \textit{SI Appendix}).

We also analyzed how morphological and cellular composition differences between age groups vary with cortical depth (\Cref{fig:main_3}\textit{E}). For each thymus, we defined a depth kernel ranging from 0 at the outer capsule to 1 at the cortico-medullary junction (CMJ). At each depth, we used energy distance \cite{szekely2013energy} (a statistical metric of how different two distributions are) to quantify differences between age groups based on pairwise SampEuler distances (capturing structural differences) and T-cell enrichment ratios (capturing cellular composition across whole tiles; see \textit{SI Appendix} for definitions). The results show that both morphological and cellular differences between age groups occur near the capsule and diminish towards the CMJ, highlighting coincident changes between cTEC morphology and cellular composition with age. This is consistent with the SHAP analysis above, in which the cortex showed the most distinct age-related morphology while the medulla was relatively similar: the age difference is greatest in the outer, sub-capsular cortex and fades towards the cortico-medullary junction.

These cortical changes are statistically significant: across all K8 cortical quadrants, a two-sample permutation test on the pairwise SampEuler distances rejects the null hypothesis of a common shape distribution ($p<2\times10^{-5}$, $50{,}000$ permutations). Because SampEuler is a complete shape descriptor (equal representations if and only if identical shape~\cite{curry2022many}), this single test is sensitive to any change in shape, not to a pre-selected feature.

\begin{figure*}
    \centering
    \includegraphics[width = 17.5cm]{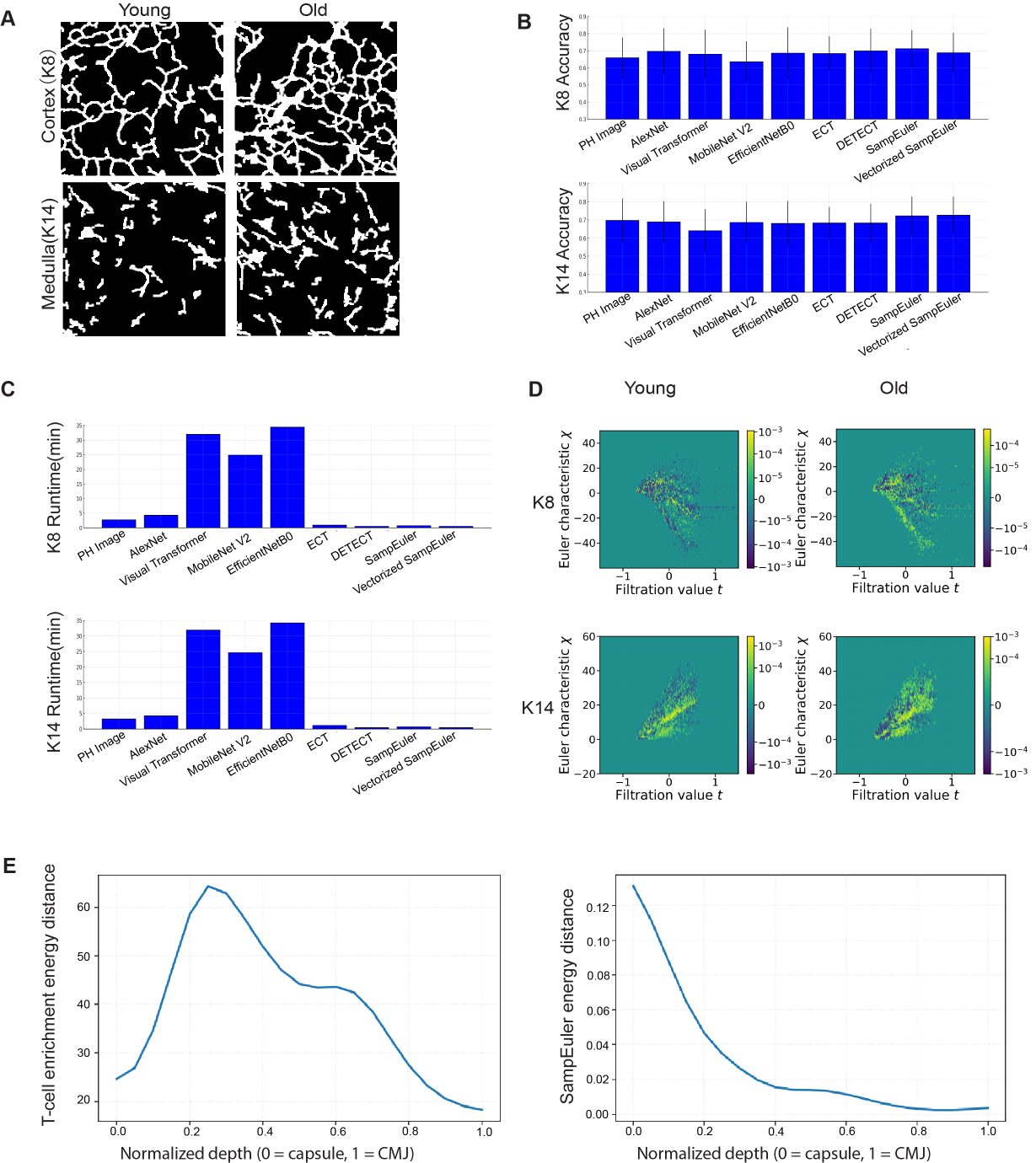}
    \caption{CLASSIFICATION STUDY ON SELECTED THYMIC QUADRANTS. We segment masked images of thymic epithelium into quadrants of size $200\times 200$ pixels. Then, we classify them into old and young groups separately for K8 and K14 markers. (\textit{A}) Example images of thymic quadrants of each age group and each marker. (\textit{B}) Barplots of the classification accuracies of each method. The top plot contains accuracies for classifying K8 quadrants, and the bottom plot contains accuracies for classifying K14 quadrants. The error bars show the standard deviations for repeating the classification 50 times. (\textit{C}) Barplots of the classification runtimes of each method. The reported runtime is the total time for each method, including feature generation, model training, and test inference. The top plot contains runtimes in minutes for classifying K8 quadrants, and the bottom plot contains runtimes in minutes for classifying K14 quadrants. (\textit{D}) We use the SampEuler vectorization as input features for classification training. We repeat the train-test process 50 times and keep trials with over $80\%$ accuracy. For trials with accuracy over $80\%$, we compute the SHAP value for each feature. For each marker and age group, we average across all test samples of the age group and then across all valid trials. The plot shows average SHAP values for each feature as feature importances towards classifying the sample as a given age group. The horizontal axis is the filtration value $t$ and the vertical axis is the Euler characteristic $\chi$. The colour gives the mean SHAP value on each feature: a positive value (yellow) indicates that ECCs passing through that $(t,\chi)$ cell push the classifier towards the given age group, while a negative value (purple) pushes it away. (\textit{E}) Depth-dependent comparisons between thymic quadrants from different age groups. For both plots, the x-axis shows the normalized depth of the thymus, where 0 corresponds to quadrants at the capsule and 1 to quadrants at the cortex–medulla junction (CMJ). In the left plot, the y-axis shows the energy distance \cite{szekely2013energy} between age groups based on T-cell enrichment ratios. In the right plot, the y-axis shows the energy distance between age groups based on pairwise SampEuler distances.}
    \label{fig:main_3}
\end{figure*}

\subsection*{Comparing Cell Distribution with Shape}
In addition to TEC segmentation, we annotated the haematopoietic and stromal cell types for this dataset. To compare the impact of ageing on both cell distributions and TEC architecture, we perform k-medoids clustering~\cite{kaufman2009finding} separately based on cell enrichment ratios (see \textit{SI Appendix} for definitions) and SampEulers. 

Age-related thymic involution impacts medullary and cortical TEC types differentially. While both cell populations decrease in absolute terms, the relative representation of cTEC to mTEC increases with age~\cite{elife}. cTECs are known to mediate key biological processes like lineage commitment and positive T cell selection~\cite{white2023diversity}. These processes are essential for haematopoietic T cell precursors to (a) adopt a T cell lineage, and later in development to (b) probe their ability to recognize self-peptide MHC complexes. Therefore, we decided to focus on the cortical region for the subsequent analyses and investigate morphological changes in the cortical thymic epithelial cells along with compositional changes of nearby T cell subtypes. For the composition definition we limited this analysis to the most abundant and biologically informative T cell subtypes in the cortical compartment: double-negative type 3 (DN3) T cells, pre (positive) selection double-positive T cells (Presel DP), post (positive) selection double positive T cells (CD69-) (Postsel DP, CD69-), and post (positive) selection double positive T cells (CD69+) (Postsel DP, CD69+). Full phenotype marker definitions are provided in \textit{SI Appendix}. Furthermore, we restricted the compositional analysis to those thymocyte centroids that lie within $5\,\mu$m of the K8 segmentation mask. We clustered all cortical quadrants from both age groups together so that cluster labels are comparable across age.

To allow faster computation of the pairwise Wasserstein distance matrix of all cortical quadrants, we use the sliced Wasserstein distance algorithm \cite{bonneel2015sliced} with $50$ slices to approximate the true Wasserstein distance. 

We first used the silhouette analysis \cite{Rousseeuw1987Silhouette} (a standard way to choose the number of clusters) to find the best $k$ for each clustering. The results suggest the highest silhouette score at 2 clusters for both the shape and the cell composition based clustering. In practice, 2-medoid results do not provide sufficient biological granularity so we fixed $k=3$ in both the SampEuler and the cell enrichment ratio-based clustering (\Cref{fig:main_4}\textit{A}). To complement the cell enrichment based clustering we calculated the average cell type enrichment ratios for key cortical cell types (\Cref{fig:main_4}\textit{B}). We observe that the overall spatial arrangement of the shape-clustering is largely preserved with age (\Cref{fig:main_4}\textit{A} i and iii), but this does not hold for the cell enrichment-based clustering. To quantify this observation, we computed the proportion of shape clusters within each cell enrichment ratio cluster separately for each age group (\Cref{fig:main_4}\textit{C}). The confusion matrices reveal that for the young tissue, the shape and the cell enrichment ratio-based clustering show overlap, with higher intensity concentrated along the diagonal (\Cref{fig:main_4}\textit{C} i). For the old tissue, however, the cell enrichment clusters representing the sub-capsular, deep cortical, and cortico-medullary areas (\Cref{fig:main_4}\textit{A} iv) all show preferential association with shape cluster 2 (\Cref{fig:main_4}\textit{A} iii). To confirm this observation, we selected CD4 and CD8 double-negative, type 3 thymocytes (DN3s; a cell type that in young thymi is typically associated with the sub-capsular localization; \textit{SI Appendix}, Fig. S13\textit{A}) and investigated whether there is a significant shift of that population into the deeper cortex with age. Qualitatively, the DN3 population in old thymi appears to be more diffuse and less confined to the sub-capsular area than in young (\textit{SI Appendix}, Fig. S13\textit{B}). To see if this qualitative difference is significant, we computed the distances between the capsule and all DN3 centroids within a $500\,\mu$m (sub-capsular) range for all samples (examples shown in \textit{SI Appendix}, Fig. S13\textit{C-D}). Then, we compared the resulting distributions of cell distances using a two-sample t-test ($p = 6.3\times10^{-11}$) and a Wilcoxon test ($p = 3.1\times10^{-18}$). These significantly different distributions indicate that DN3s in older thymi are less tightly localized to the sub-capsular region, suggesting that location-specific cellular organization in the sub-capsular region is disrupted with age.
\begin{figure*}
    \centering
    \includegraphics[width=15.5cm]{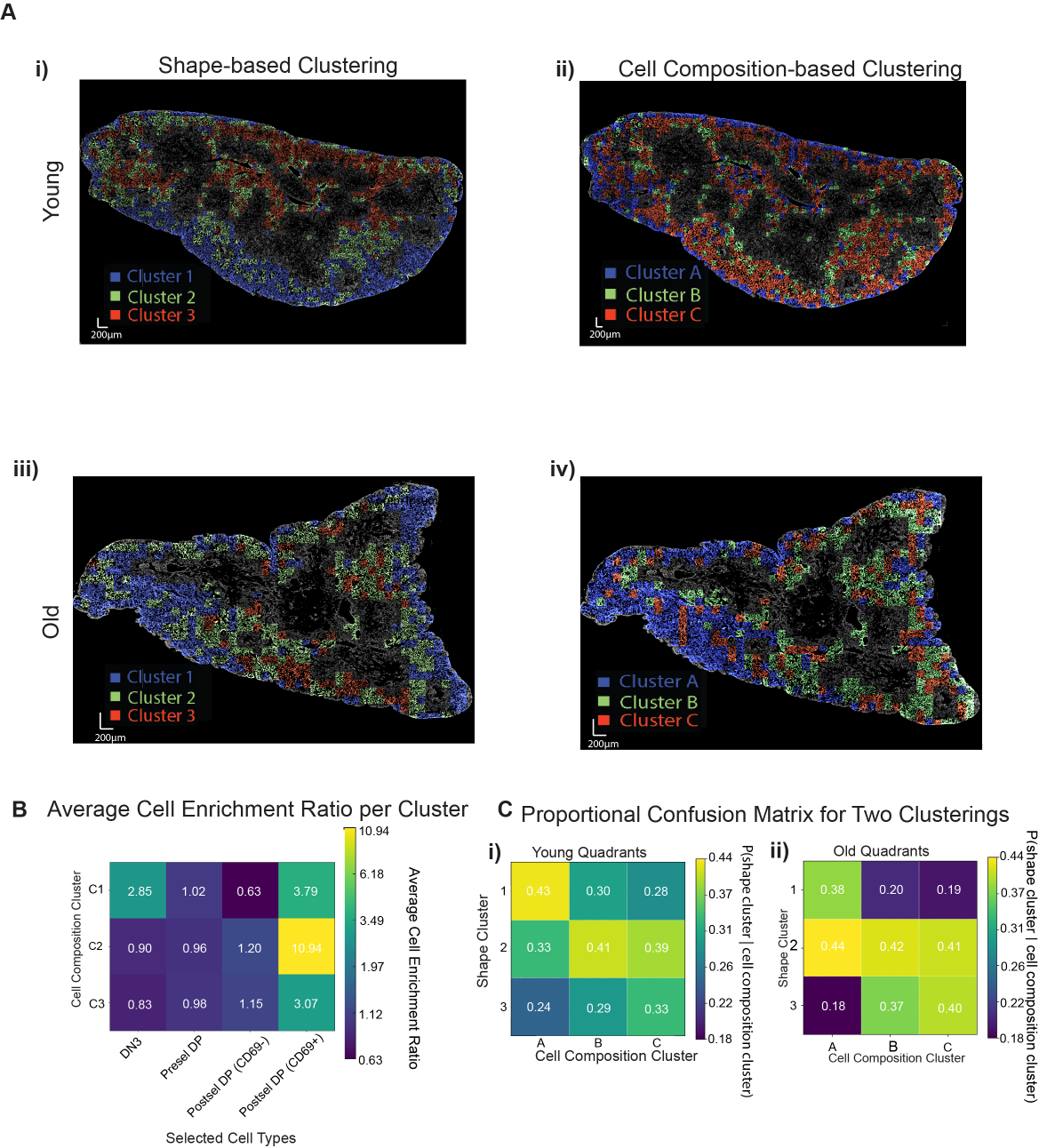}
    \caption{K-MEDOIDS CLUSTERING RESULTS. We cluster all cortex quadrants from both age groups using SampEuler distances and selected cell enrichment ratio distances with k-medoids clustering. We match the clustering results by finding the label permutation that gives the largest number of matched patches. (\textit{A}) Example overlay images of the thymic quadrants. The top two samples are young thymi, and the bottom two samples are old thymi. \textit{A} i) and iii) use SampEuler distances to perform shape-based clustering, and \textit{A} ii) and iv) use selected cell type enrichment ratios. We permute the labels to obtain the best match between two results. (\textit{B}) Average cell enrichment ratios by cluster and cell type. Rows represent the three cell-composition-based clusters, columns represent selected cell types, and values show the mean enrichment ratio for each cell type within each cluster. (\textit{C}) Proportional confusion matrices matching the two clustering results. The $(i,j)$-th entry of the matrix is given by the formula $\frac{\text{No. of quadrants being classified as shape cluster } i \text{ and cell composition cluster } j}{\text{No. of quadrants being classified as cell composition cluster } j}$} 
    \label{fig:main_4}
\end{figure*}

\section*{Discussion}
In this paper, we introduced and implemented SampEuler and its vectorization as two novel topological shape descriptors. We applied them to a benchmarking shape image dataset and a mouse thymus dataset. The results show that both novel shape descriptors are robust to rigid transformations, resolving the high sensitivity issue of ECT \cite{turner2014persistent}. Moreover, SampEuler can extract shape information effectively to resolve shape differences with high accuracy. 

We demonstrated the effectiveness of SampEuler and its vectorization on synthetic three-edge tree simplices, highlighting their capacity to capture and emphasize purely shape-based information. Benchmarking on the MPEG7 dataset confirmed that SampEuler achieves higher classification accuracy than conventional TDA methods based on persistent homology. While dedicated 2D image recognition algorithms and TDA with deep learning approaches achieve higher accuracy, SampEuler provides a favorable balance between performance and interpretability.

Their application to the mouse thymus dataset underscored the strength of SampEuler in the study of complex biological structures. To the best of our knowledge, this is the first complete, quantitative characterisation of the age-related change in the whole shape of the thymic epithelial network, rather than through the separate indices used previously. The random forest classifiers trained on SampEuler and the vectorization of SampEuler achieved competitive accuracies ($>70\%$) for age classification relying on a limited set of training data (i.e. cell shapes) at notably lower computational cost than deep learning models. This shows that a) a quantifiable change in cell morphology occurs with age for both cTEC and mTEC and b) this shape change can inform the age-category of the underlying sample. These age-related changes are most pronounced at the outer, sub-capsular boundary of the cortex, where the epithelial network forms more loops with age, while the medulla remains relatively similar between age groups. While this holds true for the two ages considered, a more granular dataset is required to confirm its validity across a broader life trajectory.

Applying the SHAP algorithm to random forest classifiers trained on SampEuler vectorizations enabled a mathematical characterization of cell shape. Although classification reveals detectable shape differences between age groups, these age-related changes appear to affect different morphological subtypes separately: each shape cluster undergoes its own characteristic shift with age, but cluster-specific features are nonetheless preserved. Consequently, the overall spatial organization of shape clusters remains largely intact (\Cref{fig:main_4}\textit{A} i and iii). However, when considering cell composition clustering, a spatial shift becomes apparent with age. In young thymi, Cluster A is confined to the sub-capsular region, whereas in old thymi, this same compositional pattern extends deep into the cortical compartment (\Cref{fig:main_4}\textit{A} ii vs. iv).

Relating morphological changes to cellular composition changes relies on a descriptor that captures the full morphological change in a single quantity. Previous quantifications of thymic architecture are either global, summarising whole-tissue compartments or network-level indices~\cite{irla2013three, flores1999analysis, lagou2026morphometric}, or local, describing individual cells~\cite{venablesDynamicChangesEpithelial2019a}, and each captures only a single aspect of the geometry or topology; they therefore cannot be applied here to capture all the local shape changes consistently across samples or to correlate them with cellular composition as we do. By reducing the entire shape change to one descriptor, SampEuler makes such correlations possible.

This compositional reorganization could arise from several age-related factors, such as loss of intra-thymic cytokine gradients that normally maintain youthful regional organization, or failure in the provision of lympho-stromal cross-talk necessary for double-negative T cell progression. However, additional experimental validation is required to confirm these hypotheses.

Building on these positive findings, we propose that future work investigate the theoretical properties of the ECT pushforward measure in general simplicial complex settings (pinning down exactly which shapes the descriptor can always tell apart; the open cases are symmetric shapes such as an exact square or circle, which do not occur in real tissue), as well as develop formal results on the information captured by the vectorization of SampEuler. While the present study focuses on 2D image masks, the pipeline extends naturally to simplicial complexes in both 2D and 3D. In particular, applying the framework to 3D images would better reflect the underlying biological structures and further enhance the quality of the analysis pipeline introduced in this work.

\section*{Materials and Methods}

\subsection*{MPEG-7 Dataset}
 The MPEG-7 shape dataset was used for benchmark shape analysis methods. We applied algorithms to fill in the holes within the objects and retained only the largest connected component inside the image. We then normalized the image sizes to be equal across all classes (See \textit{SI Appendix} for details).

\subsection*{Thymus Experimental Procedure}
Individual lobes from the thymi of two young wild type (C57BL/6) female mice (4 weeks of age) and five old C57BL/6 female mice (75-77 weeks of age) were collected, excess fat removed and frozen in optimal cutting temperature (OCT, TissueTek) polymer. Tissue sections were cut at 7 $\mu$m thickness and processed according to the Phenocycler-Fusion user guide (v2.1.0, Akoya Biosciences) using the Sample kit for Phenocycler-Fusion (Akoya Biosciences). Briefly, thymic tissue was fixed in acetone (Merck) for 10 minutes at room temperature, rehydrated in a hydration buffer (Akoya Biosciences) for 10 minutes and fixed in paraformaldehyde (PFA, 1.6\%, Thermo Fisher Scientific) for 10 minutes at room temperature. It was then equilibrated in a staining buffer (Akoya Biosciences) for 20 minutes and stained in a staining buffer containing four blockers (Akoya Biosciences) and antibodies barcoded with DNA oligonucleotides overnight at $4 ^\circ\mathrm{C}$ . The tissue was washed with the staining buffer 10 minutes, fixed in 1.6\% PFA for 10 minutes, washed in PBS and exposed to ice-cold methanol (Merck) for 5 minutes. After rinsing in PBS, the tissue was fixed in a fixation solution (Akoya Biosciences) for 20 minutes at room temperature, washed in PBS and stored in a storage solution (Akoya Biosciences) until loading into the Phenocycler for cyclic hybridisation with fluorescently labelled complementary oligonucleotides \cite{CodexGrandma}. The barcoded antibodies listed in Table S4 were either from Akoya Biosciences or purchased from indicated providers and custom conjugated using a conjugation kit (Akoya Biosciences). Images were acquired using the Fusion slide scanner (Akoya Biosciences). The iterative imaging resulted in image stacks with 40 channels and a resolution of 0.5 $\mu$m/pixel.

\subsection*{Thymus Image Processing and Cell Annotation}
Image analysis was performed in MATLAB (R2024b). Preprocessing involved restricting the regions of interest to the individual thymus sections, thereby removing artefacts. The background fluorescence was subtracted, and the fluorescent intensities of the markers were normalized across tissue samples. Cell segmentation was performed in two ways depending on the cell morphology. Round cell boundaries (e.g., immune cells) were outlined using Cellpose \cite{CellposeGrandpa}. Stromal cell boundaries (e.g., TEC, fibroblasts, endothelial cells) have a more complex morphology and were therefore customly segmented using MATLAB's multithresh() and bwskel() functions (Image Processing Toolbox), followed by pseudo-cellularization into 25x25 pixel sized tiles. For all cells and pseudo cells, centroid coordinates, area and average fluorescence intensities were exported. Cells with outlier fluorescence intensities were removed, and the fluorescence intensity of each marker was then scaled between 0 and 1. This scaling enabled the selection of marker intensity thresholds which were used to phenotype cells (see Supplementary Table S5). For the positional analysis of the DN3 T cells, first the concave hulls of all samples were identified using the R(V4.5.1) package concaveman (V1.2.0). Next, sf (V1.0-22) was used to compute the shortest distance within the first 500 $\mu$m of all DN3 centroids to the capsule. Those distances were aggregated across all mice, regardless of age, and compared using base R's implementation of the \texttt{t.test} and the \texttt{Wilcox.test}.

\subsection*{Euler Characteristic Transform (ECT)}
ECT computations for grayscale images were performed using the Eucalc package (\url{https://github.com/HugoPasse/Eucalc}). The algorithm normalizes images to a size of 1 unit by 1 unit. We chose the range for the parameter of sublevel sets as $(-1.5, 1.5)$ to allow full inclusion of the image along all possible directions. The number of directions and sublevel set parameters were tuned for each dataset; an ablation study examining parameter sensitivity is provided in \textit{SI Appendix}.

\subsection*{SampEuler and Its Vectorization}
SampEuler is a discretization of the ECT pushforward of a uniform measure on the sphere of directions. For each complex, we draw a fixed number of directions independently and uniformly at random from the unit circle, compute the corresponding Euler characteristic curve, and retain the resulting curves as an \emph{unordered} collection, that is, as an empirical measure on the space of curves. We prove that it converges to the ECT pushforward measure in the Wasserstein metric as the number of directions grows, so that for generic complexes it determines the shape class uniquely. Distances between two SampEulers are Wasserstein-1 distances between the corresponding empirical measures, approximated by the sliced Wasserstein distance~\cite{bonneel2015sliced} with $50$ slices when large pairwise distance matrices are required. The vectorization of SampEuler evaluates this empirical measure on sets of curves that are constant on prescribed intervals: we partition the plane of filtration value against Euler characteristic into a uniform grid and record, for each cell, the proportion of sampled curves that are constant on that cell's filtration interval and equal to the corresponding Euler characteristic value. This yields a low-dimensional, isometry-invariant feature vector that also serves as a visualization of the distribution of Euler characteristic curves, and from which DETECT~\cite{marsh2022detecting} can be recovered exactly. All parameters were held fixed across every sample within a given task. Formal definitions, the convergence and stability theorems with their proofs, and ablation studies over the number of directions and filtration values are given in \textit{SI Appendix}.

\subsection*{Classification, Clustering, and Statistical Analysis}
For the MPEG-7 benchmark, we built radial basis function kernels from pairwise Wasserstein distances between SampEulers and from pairwise $L_2$ distances between vectorized SampEulers, and classified with a support vector machine~\cite{cortes1995support} using a $70/30$ train-test split averaged over $100$ repeats. For the thymus quadrants, we trained random forest classifiers~\cite{breiman2001random} on vectorized SampEuler features with grid-searched hyperparameters, repeating the train-test split $50$ times and reporting the mean accuracy. The same 50-repeat train-test protocol was used for every baseline (ECT, DETECT, persistence images, and the deep learning models), so that accuracies and runtimes are directly comparable. Feature importances were obtained by computing SHAP values~\cite{lundberg2017unified} for every trial exceeding $80\%$ accuracy and averaging them over test samples and trials, yielding the importance maps in \Cref{fig:main_3}\textit{D}. To test for age-related shape change, we applied a two-sample permutation test with $50{,}000$ permutations to the pairwise SampEuler distances of all cortical (K8) quadrants. Depth-resolved comparisons used the energy distance~\cite{szekely2013energy} between age groups, evaluated at each value of a normalized depth coordinate running from 0 at the capsule to 1 at the cortico-medullary junction, and applied both to pairwise SampEuler distances and to T-cell enrichment ratios. Quadrants were clustered with k-medoids~\cite{kaufman2009finding} separately on SampEuler distances and on cell enrichment ratios, pooling both age groups so that cluster labels are comparable across age, with $k=3$ chosen as described in \textit{Results}. Distances from DN3 centroids to the capsule were compared between age groups using a two-sample t-test and a Wilcoxon rank-sum test. Full parameter settings, definitions of the enrichment ratio, and additional analyses are provided in \textit{SI Appendix}.

\subsection*{Ethics Statement}
Both the caring for the mice and the experiments were performed in compliance with permissions of the Cantonal Veterinary Office of Basel-Stadt, Switzerland, and in accordance with Swiss federal regulations.

\section*{Data Availability}
Code for generating synthetic networks, preprocessing MPEG-7 images, and computing ECT-based methods for simplicial complexes is available at \url{https://github.com/reddevil0623/Shape-transform-descriptor-for-thymus-structures} (DOI: \url{https://doi.org/10.5281/zenodo.18614535}). Thymus segmentation mask quadrants, total masks, and cell composition data for each quadrant are also available at this repository. The original MPEG-7 dataset is publicly available from \url{https://dabi.temple.edu/external/shape/MPEG7/dataset.html}. Raw cell-level thymus data are available at \url{https://doi.org/10.5281/zenodo.21813640}.

\section*{Acknowledgments}
HAH, HMB, VL are grateful for the support provided by the UK Centre for Topological Data Analysis Engineering and Physical Sciences Research Council (EPSRC) grant EP/R018472/1 and EPSRC EP/Z531224/1. HAH gratefully acknowledges funding from the Royal Society RGF/EA/201074, UF150238 and EPSRC EP/Y028872/1 and EP/Z531224/1. HAH and HY acknowledge funding by the Leverhulme Trust Prize PLP-2020-252. HY is grateful for the support of the Leathersellers’ Foundation. HMB acknowledges support provided by the Mark Foundation for Cancer Research. AT and GAH are supported by the Wellcome Trust (Wellcome Collaborative Award SynThy, 211944/Z/18/Z), the Swiss National Science Foundation (310030B\_138655 and 310030\_215113), and the DiGeorge Syndrome Research Fund at Stanford University, generously established by an anonymous donor. LT gratefully acknowledges funding from Foundation Immune Engineering for Global Child and Adolescent Health.

\section*{Author Contributions}
H.A.H., G.A.H. and H.M.B. designed research; H.Y., A.T., L.T. and V.L. performed research; H.Y., V.L., H.A.H., and H.M.B. developed new analytic tools; S.Z. performed the experiments; H.Y., A.T. and L.T. analyzed data; and H.Y., V.L., A.T., L.T., G.A.H., H.A.H., and H.M.B. wrote the paper.

\section*{Competing Interests}
The authors declare no competing interests.

\clearpage
\onecolumn
\appendix
\beginsupplement

\section*{Supporting Information Appendix}

\setcounter{section}{0}
\renewcommand{\thesection}{S\arabic{section}}

\section{Data}

We use three sources of data: synthetic network data, published MPEG-7 Core Experiment Shape-Matching dataset (MPEG-7) \cite{Latecki2000ShapeCE}, and experimental image data of mouse thymi of different age groups. We explain the details of each dataset in turn: first the synthetic network dataset, then the preprocessing performed on the MPEG-7 dataset. Finally, we describe the thymus experiment and image processing details.

\subsection{Synthetic Networks}\label{section: toy sample}
We generate a class of synthetic networks as follows. We construct each sample by first generating three edges of lengths $2$, $3$, and $4$ with their common endpoint fixed at the origin $(0, 0)$. Each edge is discretized into $n=100$ equally‐spaced vertices (including the origin). For each edge, we select a rotation angle $\phi\sim\mathrm{Uniform}(0,2\pi)$ and rotate all vertices about $(0,0)$ by $\phi$.  For each non-origin vertex, we independently sample Gaussian noise 
    $$
    (\varepsilon_x,\varepsilon_y)\sim\mathcal N\bigl(0,\,\sigma^2\bigr),\quad\sigma=0.02,
    $$
and obtain the perturbed vertex position by adding this noise to the rotated coordinates. Connectivity (line segments between consecutive edge points and from the origin to the first point of each edge) is preserved.  We repeat this process $k=20$ times to generate $20$ structures for each shape class. We generate two different shape classes for comparison by choosing different angles of the three edges at the start.

\subsection{MPEG-7 Core Experiment Shape-Matching dataset}\label{section: mpeg7}
We use the standard MPEG-7 Core Experiment Shape-Matching (MPEG-7) dataset \cite{Latecki2000ShapeCE}, which contains 1400 binary images of silhouettes. The images are split into 70 categories, each containing 20 images based on the same types of objects (e.g., apples, horses, etc.).

We preprocess the images by first expanding the white connected component by 1 pixel in all directions to close small gaps and filter to keep only the largest connected component. We then fill in the holes of the filtered images using \cite{2020SciPy-NMeth}. To normalize across the dataset, each remaining white component is then rescaled to a fixed target area, and the image is padded with black pixels until it matches a common output size. Finally, we translate the white component so that the centroid of its white pixels coincides with the image centre. The preprocessing code is available at: \url{https://github.com/reddevil0623/Shape-transform-descriptor-for-thymus-structures} 

\begin{figure}

\centering
\includegraphics[scale = 0.8]{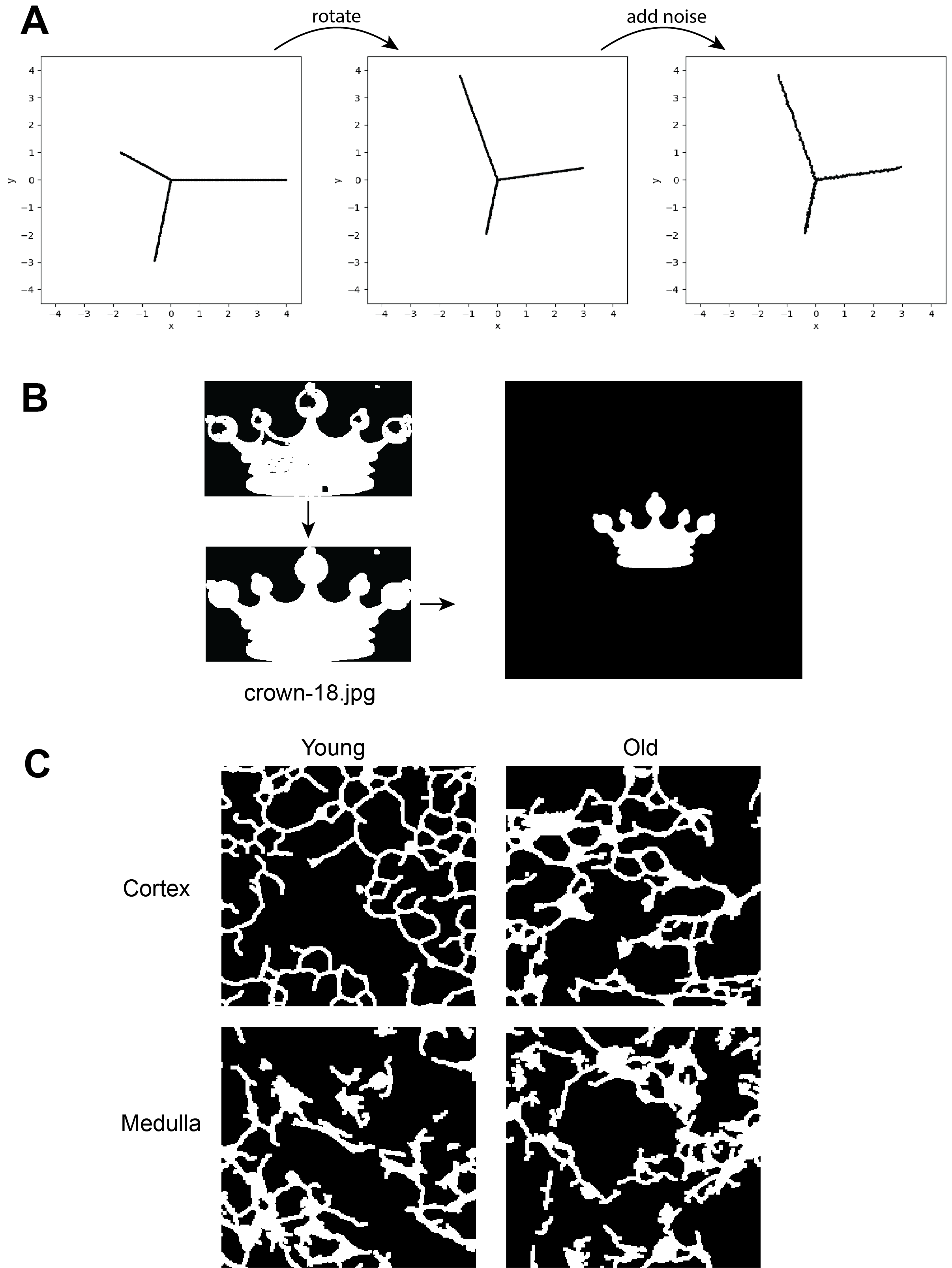}
\caption{\textit{A}: Pipeline for generating a class of tree samples. We first generate three edges of lengths 2, 3, and 4 from the origin, randomly oriented. We discretize each edge by sampling 100 equally spaced points along it and connecting consecutive points with edges. For each sample in the class, we randomly rotate the tree around the origin and then add coordinate-wise Gaussian noise of $\sigma = 0.02$ to each vertex. We keep the original connecting order of the edges. \textit{B}: Preprocessing pipeline of MPEG-7 Core Experiment Shape-Matching dataset. For each image, we first expand all white connected components by 1 pixel in all directions to close small gaps and filter to keep the largest connected component. We then fill in all the holes for the filtered image using \cite{2020SciPy-NMeth}. We pad black pixels around the original image to normalize all images to the same size, with the centroid of the connected component at the center of each padded image. \textit{C}: Four example image quadrants of the mouse thymus dataset, each of size 200 $\times$ 200 pixels. The pixels are of side length $0.5 \text{ }\mu \text{ m}$. From left to right, top to bottom, they are from young-cortex, old-cortex, young-medulla, and old-medulla.}
\label{figure: data_processing}
\end{figure}

\subsection{Thymus Experimental Procedure, Image Processing, and Cell Annotation}
Individual lobes from the thymi of two young wild type (C57BL/6) female mice (4 weeks of age) and five old C57BL/6 female mice (75-77 weeks of age) were collected, excess fat removed and frozen in optimal cutting temperature (OCT, TissueTek) polymer. Tissue sections were cut at 7 $\mu$m thickness and processed according to the Phenocycler-Fusion user guide (v2.1.0, Akoya Biosciences) using the Sample kit for Phenocycler-Fusion (Akoya Biosciences). Briefly, thymic tissue was fixed in acetone (Merck) for 10 minutes at room temperature, rehydrated in a hydration buffer (Akoya Biosciences) for 10 minutes and fixed in paraformaldehyde (PFA, 1.6\%, Thermo Fisher Scientific) for 10 minutes at room temperature. It was then equilibrated in a staining buffer (Akoya Biosciences) for 20 minutes and stained in a staining buffer containing four blockers (Akoya Biosciences) and antibodies barcoded with DNA oligonucleotides overnight at $4 ^\circ\mathrm{C}$. The tissue was washed with the staining buffer 10 minutes, fixed in 1.6\% PFA for 10 minutes, washed in PBS and exposed to ice-cold methanol (Merck) for 5 minutes. After rinsing in PBS, the tissue was fixed in a fixation solution (Akoya Biosciences) for 20 minutes at room temperature, washed in PBS and stored in a storage solution (Akoya Biosciences) until loading into the Phenocycler for cyclic hybridisation with fluorescently labelled complementary oligonucleotides \cite{CodexGrandma}. The barcoded antibodies listed in Table S4 were either from Akoya Biosciences or purchased from indicated providers and custom conjugated using a conjugation kit (Akoya Biosciences). Images were acquired using the Fusion slide scanner (Akoya Biosciences). The iterative imaging resulted in image stacks with 40 channels and a resolution of 0.5 $\mu$m/pixel.

Image analysis was performed in MATLAB (R2024b). Preprocessing involved restricting the regions of interest to the individual thymus sections, thereby removing artefacts. The background fluorescence was subtracted, and the fluorescent intensities of the markers were normalized across tissue samples. Cell segmentation was performed in two ways depending on the cell morphology. Round cell boundaries (e.g., immune cells) were outlined using Cellpose \cite{CellposeGrandpa}. Stromal cell boundaries (e.g., TEC, fibroblasts, endothelial cells) have a more complex morphology and were therefore customly segmented using MATLAB's multithresh() and bwskel() functions (Image Processing Toolbox), followed by pseudo-cellularization into 25x25 pixel sized tiles. For all cells and pseudo cells, centroid coordinates, area and average fluorescence intensities were exported. Cells with outlier fluorescence intensities were removed, and the fluorescence intensity of each marker was then scaled between 0 and 1. This scaling enabled the selection of marker intensity thresholds which were used to phenotype cells (see \Cref{CODEXphenotypes}). For the positional analysis of the DN3 T cells, first the concave hulls of all samples were identified using the R(V4.5.1) package concaveman (V1.2.0). Next, sf (V1.0-22) was used to compute the shortest distance within the first 500 $\mu$m of all DN3 centroids to the capsule. Those distances were aggregated across all mice, regardless of age, and compared using base R's implementation of the \texttt{t.test} and the \texttt{Wilcox.test}.

\section{Analysis Methods}

In this section, we introduce the necessary background on topological transforms and theoretical extensions in the first section, and then in the second section, we describe other analysis methods used in this study. 

\subsection{Topological Transforms}
We recall the Euler characteristic transform (ECT) and then present a theoretical extension (\Cref{theorem: ECT continuity} and \Cref{theorem: ect_pushforward_measure_stability}). Next we introduce SampEuler and present desirable theoretical and practical properties of this descriptor for analysing shape datasets (\Cref{theorem: SampEuler convergence}).

\subsubsection*{The Euler Characteristic Transform}
The Euler characteristic transform (ECT) is one of the foundational tools within the Topological Data Analysis community \cite{wang2021statistical, marsh2022detecting, kirveslahti2024digital, roell2023differentiable}.

In this section, we introduce the basic concepts of ECT for geometric simplicial complexes in order to describe our analysis methods. See \cite{vdD} and \cite{curry2022many} for a more general setting of ECT.

\begin{definition}
    A \textbf{geometric $k$-simplex} $K_1$ is the convex hull of $k+1$ affinely independent points $v_0,v_1,\dots, v_k$ and denoted by $[v_0,\dots, v_k]$. For any $\{u_0,\dots,u_j\}\subset\{v_0,\dots,v_k\}$, we call $[u_0,\dots,u_j]$ a \textbf{face} of $v_0,v_1,\dots, v_k$. The \textbf{dimension} of $K_1$ is $\dim(K_1) = k$.

    A geometric simplicial complex $K$ is a finite collection of geometric simplices $\{K_1,\dots,K_n\}$ such that
    \begin{enumerate}
        \item For every $K_i \in K$, every face of $K_i$ is also an element of $K$;
        \item For any $K_i, K_j \in K$, either $K_i \cap K_j = \emptyset$ or $K_i \cap K_j$ is a face of both $K_i$ and $K_j$.
    \end{enumerate}
    The \textbf{dimension} of $K$ is defined as $\max\{\dim(K_i):i\in\{1\dots n\}\}$.
\end{definition}

\begin{definition}
    For a geometric $k$-simplex $K_1$, the Euler characteristic of $K_1$, denoted by $\chi(K_1)$, is defined to be $\chi(K_1):=(-1)^k$. For a geometric simplicial complex $K$ consisting of the finite collection $\{K_1,\dots,K_n\}$, the Euler characteristic of $K$ is defined as $\chi(K) = \sum_{i=1}^n \chi(K_i)$.
\end{definition}

\begin{example}
    For any discretization $K$ of the unit circle, there are $n$ vertices ($0$-simplices) and $n$ edges ($1$-simplices) forming a closed cycle for some integer $n$. Therefore, the Euler characteristic is $\chi(K) = n - n = 0$ for any such discretization $K$.
\end{example}

\begin{definition}
    For geometric simplicial complex $K$ consists of simplices $K_1,\dots, K_n$ of dimension $d$, the \textbf{Euler characteristic transform} of $K$, denoted by $\ect(K): S^{d-1}\times\R\rightarrow\Z$ is defined as: 
        $$\ect(K)(v,t) = \chi(K^{\langle v, \cdot\rangle\leq t}),$$
    where $K^{\langle v, \cdot\rangle\leq t} = \{K_i: \langle v, x\rangle\leq t, \forall x\in K_i\}$. Fixing direction $v$, the curve $\ect(K)(v,-):\R\rightarrow\Z$ is the \textbf{Euler characteristic curve (ECC)} of $K$ along direction $v$.
\end{definition}

We now briefly recall the notion of persistent homology and persistence diagrams; see \cite{turner2014persistent, curry2022many} for details in the context of topological transforms.

Given a topological space $X$, its $k$-th homology (with coefficients in a field) is a vector space whose dimension counts the number of independent $k$-dimensional ``holes'' of $X$. By $0$-dimensional holes we mean connected components, by $1$-dimensional holes we mean cycles or loops, and by $2$-dimensional holes we mean enclosed voids. These $k$-dimensional holes are often called \emph{homological features} of $X$.

The idea of persistent homology is to track how these homological features appear and disappear as the simplicial complex grows through a one-parameter filtration. Given a sequence of complexes $(K_t)_{t \in \R}$ such that $K_t\subset K_r$ for all $t\leq r$, each homological feature is born at some parameter value $a \in \R$ and dies at some $b > a$. The pair $(a,b)$ records the birth and death of that feature. The multiset of all such pairs $\{(a_i, b_i)\}$ is called the \emph{persistence diagram}. Points far from the diagonal $\{a = b\}$ correspond to long-lived, and hence typically more significant, topological features, while points close to the diagonal represent short-lived features that are often attributed to noise.

In the context of topological transforms, for a fixed direction $v \in S^{d-1}$, the height filtration $K^{\langle v, \cdot \rangle \leq t}$ gives rise to a one-parameter family of subcomplexes of $K$ from which we can compute a persistence diagram for each homological degree.
\begin{definition}
    For geometric simplicial complex $K$, the \textbf{persistent homology transform} of $K$ denoted by $\text{PHT}(K): S^{d-1}\rightarrow \mathrm{Dgm}^d$ is defined as:
        $$\text{PHT}(K)(v) = (\text{PH}_0(K,v),\dots, \text{PH}_{d-1}(K,v))$$
    where $\mathrm{Dgm}$ is the space of persistence diagrams, and $\text{PH}_i(K,v)$ is the $i^{\text{th}}$ persistent homology of $K$ induced by the height filtration along direction $v$.
\end{definition}

We use $\CF$ to denote the space of all piecewise constant functions from $\R$ to $\Z$ as an integer combination of half-closed and half-open intervals. The ECCs resulting from ECTs are within $\CF$. 

\begin{theorem}[Continuity and Hausdorffness \cite{curry2022many}]\label{continuity-and-t2}
    Given a geometric simplicial $K\subset\R^d$, if the ECT of $K$ is viewed as $\ect(K): S^{d-1}\rightarrow \CF$, a map from the space of directions to the space of ECCs, then $\ect(K)$ is continuous with normal metric on $S^{d-1}$ and $L_p$ metric on $\CF$ with $1\leq p < \infty$. Furthermore, the space $\CF$ is a Hausdorff space.
\end{theorem}

We fix $p=1$ and adopt the $L_1$ metric for $\CF$ below, but the results hold true for any $1\leq p <\infty$ unless stated explicitly. 

\begin{theorem}[Injectivity theorem \cite{curry2022many, ghrist2018persistent}]
    Consider two geometric simplicial complexes $K$ and $L$. We have $\ect(K) = \ect(L)$ if and only if $K = L$.
\end{theorem}

The injectivity result shows that ECT can capture all geometric and topological information of the input complex. However, some differences between various complexes may not be of interest in analysis. For instance, a nonsymmetric complex $K$ before and after applying a nontrivial rotation would generate two different ECT results.

As suggested by \cite{kendall1999shape}, the shape of a geometric simplicial complex $K$ is the equivalence class of $K$ up to rigid motions, i.e., translations, rotations, and reflections. Centring the complex $K$ such that the sum of all vertices of $K$ is at the origin enables rapid alignment of shapes up to translations. Throughout this paper, we assume the input complex is centred.  However, no computationally cheap alignments can be performed for isometries.

\begin{definition}
    Given two geometric simplicial complexes $K$ and $L$, we write $K\sim L$ if and only if there exists a rigid motion $\phi$ such that $\phi(K) = L$. The relation $\sim$ is an equivalence relation, and we define the \textbf{shape class} of $K$ as the equivalence class $[K]$.
\end{definition}

We use the following definitions from \cite{curry2022many}.

\begin{definition}
    A geometric simplicial complex $K$ in $\R^d$ with vertex set $X$ is \deff{strongly generic} if:
    \begin{enumerate}
        \item ECCs are distinct for all directions $v\in S^{d-1}$,
        \item the vertex set $X$ is in general position, i.e., no subset of $k$ points in $X$ lie on a $(k-2)$-dimension affine subspace, for $k = 2,3,\dots, d+1$.
    \end{enumerate}
\end{definition}

\begin{definition}
    For strongly generic geometric simplicial complexes $K\subset\R^d$, the \deff{ECT pushforward measure} of $K$ is $\ect(K)_*\mu: \mathcal{B}(\CF)\rightarrow \R$, such that $\ect(K)_*\mu(E) = \mu(\{v\in S^{d-1}: \ect(K)(v,-)\in E\})$, where $\mu$ is the uniform measure on the sphere and $\mathcal{B}(\CF)$ is the Borel $\sigma-$algebra on the space of ECCs.
\end{definition}

\begin{theorem}[Theorem 6.7 of \cite{curry2022many}]\label{pushforward-measure}
    Let $K$ and $L$ be strongly generic geometric simplicial complexes in $\R^d$. Let $\mu$ be the Lebesgue measure on $S^{d-1}$. If the pushforward measures $\ect(K)_*(\mu) = \ect(L)_*(\mu)$, then there is some $\phi\in O(d)$ such that $L = \phi(K)$, i.e.\ $L$ and $K$ are identical up to a combination of rotations and reflections.
\end{theorem}

This shows that the ECT pushforward measure characterizes the shape classes of the complexes completely. Based on the ECT pushforward measure, Marsh et al.\ defined the Detecting Temporal shape changes with the Euler Characteristic Transform (DETECT) algorithm and applied it successfully to biological datasets \cite{marsh2022detecting}. 

\begin{definition}
    Given a geometric simplicial complex $K\subset\R^d$, and suppose that the norms of all vertices are bounded within interval $[-a, a]$. The DETECT curve of $K$ is defined as: 
        \begin{align*}
            & \text{DETECT}(K): \R \rightarrow \R \\
            & t\rightarrow \int_{S^{d-1}}\int_{-a}^t \ect(K)(v, s) - \overline{\ect(K)(v,-)} \d s \d v
        \end{align*}
    where $\overline{\ect(K)(v,-)} = \int_{-a}^a \ect(K)(v, s) \d s$.
\end{definition}

Stability properties are important for data analysis. We first introduce some metrics and prove the stability property for the ECT pushforward measure in \Cref{theorem: ect_pushforward_measure_stability}. 

\begin{definition}
    Let $(M,d)$ be a metric space and $\mathcal{P}_p(M)$ be the collection of all probability measures $P$ such that $\int_M d(x, x_0)^p \text{d}P(x)<\infty$ for some $x_0$. The $p^{th}$ Wasserstein distance for $P, Q\in \mathcal{P}_p(M)$ is defined as: 
        $$W_p(P,Q) := \left(\inf_{\gamma\in\Gamma(P,Q)}\left\{\int_{M\times M} d(x,y)^p \text{d}\gamma(x,y)\right\}\right)^{1/p}$$
    where $\Gamma(P,Q)$ is the set of all measures $\mu$ on $M\times M$ such that $\mu(A\times M) = P(A)$ and $\mu(M\times B) = Q(B)$.
\end{definition}

Let $\Delta = \{(x, x) : x \in \mathbb{R}\}$ denote the diagonal in $\mathbb{R}^2$. For a point $(b, d) \in \mathbb{R}^2$, define its projection onto the diagonal as
$$ \pi_\Delta(b, d) = \left(\frac{b + d}{2}, \frac{b + d}{2}\right). $$

\begin{definition}[Matching]
Given two persistence diagrams $D$ and $D'$, a \emph{matching} between $D$ and $D'$ is a map
$$ \gamma : D \cup \Delta \to D' \cup \Delta $$
where the diagonal $\Delta$ is taken with infinite multiplicity, such that:
\begin{enumerate}
    \item Each point $x \in D$ is either matched to a point $\gamma(x) \in D'$, or matched to its projection $\pi_\Delta(x) \in \Delta$;
    \item Each point $y \in D'$ is either matched to some point in $D$, or matched to its projection $\pi_\Delta(y) \in \Delta$.
\end{enumerate}
We denote the set of all such matchings by $\Gamma(D, D')$.
\end{definition}

\begin{definition}\cite{cohen2010lipschitz}
For $p \in [1, \infty]$, the \emph{$p$-Wasserstein distance} between persistence diagrams $D$ and $D'$ is defined as
$$ W_{p}(D, D') = \left( \inf_{\gamma \in \Gamma(D, D')} \sum_{x \in D \cup \Delta} \|x - \gamma(x)\|_p^p \right)^{1/p}, $$
where $\|\cdot\|_p$ denotes the $\ell^p$-norm on $\mathbb{R}^2$.

If $p = \infty$, the Wasserstein distance is defined as:
    $$ W_{\infty}(D, D') = \inf_{\gamma \in \Gamma(D, D')} \max_{x \in D \cup \Delta} \|x - \gamma(x)\|_\infty. $$
\end{definition}

Recall the following definition taken from \cite{curry2022many}.

\begin{definition}
    Given two geometric simplicial complexes $K$ and $L$, the ECT metric between $K$ and $L$ is defined as:
        $$d_{\ect}^1(K, L) = \int_{S^{d-1}} \|ECT(K,v) - ECT(L,v)\|_1 \d v.$$
    The $\infty-$PHT distance in degree $i>0$ between $K$ and $L$ is defined as:
        $$d_{\infty}^{\text{PHT}^i}(K, L) = \max_{v \in S^{d-1}} W_{\infty}(\text{Dgm}_i(K,v), \text{Dgm}_i(L,v)).$$
    where $\text{Dgm}_i(K,v)$ and $\text{Dgm}_i(L,v)$ are the $i^{\text{th}}$ persistence diagrams of $K$ and $L$ induced by height filtration along direction $v$.
\end{definition}

\begin{theorem}[Proposition 2 of \cite{dlotko2023euler}]
    For geometric simplicial complexes $K$ and $L$ in $\R^d$, we have:
         $$ \|\ect(K,v) - \ect(L,v)\|_1 \leq \sum_{k=0}^{d-1} 2 n_k W_{\infty}(\text{Dgm}_k(K,v), \text{Dgm}_k(L,v)),$$
    where $\text{Dgm}_k(K,v)$ and $\text{Dgm}_k(L,v)$ are the $k^{\text{th}}$ persistence diagrams of $K$ and $L$ induced by height filtration along direction $v$, and $n_k$ is the maximum number of bars in $\text{Dgm}_k(K,v)$ or $\text{Dgm}_k(L,v)$.
\end{theorem}

\begin{definition}
    Let $K, L\subset \R^d$ be geometric simplicial complexes and let $\phi: K\rightarrow L$ and $\psi: L \rightarrow K$ be a homotopy equivalence of $K$ and $L$.
    Let $H_K: K\times I \rightarrow K$ and $H_L: L\times I \rightarrow L$ be homotopies from $\psi\circ \phi$ to $\id_K$ and from $\phi\circ \psi$ to $\id_L$ respectively. We denote by $\|\cdot\|_{\R^d}^2$ the normal Euclidean 2-norm. If there exists $\epsilon$ strictly positive such that $\|x - \phi(x)\|^2_{\R^d} \leq \epsilon$ and $\|y - \psi(y)\|^2_{\R^d}\leq\epsilon$ for all $x\in K$ and $y\in L$,
    as well as $\|x - H_K(x,s)\|^2_{\R_d}\leq 2\epsilon$ and $\|y - H_L(y,s)\|^2_{\R_d}\leq 2\epsilon$ for all $x\in K, y\in L, s\in I$,
    then we say $K$ and $L$ are $\epsilon-$controlled homotopic.
\end{definition}

\begin{theorem}[Lemma 4.5 and Theorem 4.11 of \cite{arya2025sheaf}]
    Let $K, L\subset \R^d$ be geometric simplicial complexes that are $\epsilon-$controlled homotopic. Then we have: 
        $$ \max_i d_{\infty}^{\text{PHT}^i}(K,L)\leq \epsilon.$$
\end{theorem}

\begin{theorem}\label{theorem: ECT continuity}
    Let $K, L\subset \R^d$ be geometric simplicial complexes that are $\epsilon-$controlled homotopic, and let $n_k$ be the maximum number of bars in $\text{Dgm}_k(K,v)$ or $\text{Dgm}_k(L,v)$ for all $v\in S^{d-1}$, and let $n = \sum_{k=0}^{d-1} n_k$. Then we have:
        $$ d_{\ect}^1(K,L) \leq 2dn\epsilon.$$
\end{theorem}

\begin{proof}
    By the definition of the ECT metric we have: 
        \begin{align*}
            d_{\ect}^1(K,L) & = \int_{S^{d-1}} \|\ect(K,v) - \ect(L,v)\|_1 \d v \\
            & \leq \int_{S^{d-1}} \sum_{k=0}^{d-1} 2n_k W_{\infty}(\text{Dgm}_k(K,v), \text{Dgm}_k(L,v)) \d v \\
            & = \sum_{k=0}^{d-1} 2n_k \int_{S^{d-1}} W_{\infty}(\text{Dgm}_k(K,v), \text{Dgm}_k(L,v)) \d v \\
            & \leq \sum_{k=0}^{d-1} 2n_k \max_{v\in S^{d-1}} W_{\infty}(\text{Dgm}_k(K,v), \text{Dgm}_k(L,v)) \\
            & = 2\sum_{k=0}^{d-1} n_k d_{\infty}^{\text{PHT}^k}(K,L) \leq 2dn\epsilon.
        \end{align*}
\end{proof}

\begin{remark}
    For a geometric simplicial complex, the number of bars in the $k$-th persistent homology is bounded by the number of $k$-simplices, which is finite. In particular, $n_k$ is at most the larger of the number of $k$-simplices in $K$ and $L$, and $n = \sum_{k=0}^{d-1} n_k$ is bounded accordingly.
\end{remark}

\begin{theorem}\label{theorem: ect_pushforward_measure_stability}
    The ECT pushforward measure equipped with the Wasserstein distance $W_1$ is Lipschitz continuous for homotopic equivalent complexes.
\end{theorem}

\begin{proof}
    We know that ECT is Lipschitz continuous when equipped with $d_{\ect}^1$ metric from \Cref{theorem: ECT continuity}. It is sufficient to show that for two complexes $K, L\subset \R^d$, the Wasserstein distance between two pushforward measures is bounded by the ECT metric. Let $\|f-g\|_1 = \int |f(t)-g(t)| \d t$ be the $L_1$ function metric. 
    We have: 
    
    \begin{align*}
        W_1(\ect(K)_*\mu&, \ect(L)_*\mu) = \inf_{\gamma\in \Gamma(\ect(K)_*\mu,\ect(L)_*\mu)}\left(\int_{\cf(\R)\times\cf(\R)} \|f-g\|_1 \d\gamma(f,g)\right) \\
        & = \inf_{\gamma\in \Gamma(\ect(K)_*\mu,\ect(L)_*\mu)}\left(\int_{\im(\ect(K))\times \im(\ect(L))} \|f-g\|_1 \d\gamma(f,g)\right). \\
    \end{align*}
    
    Let $\nu_{r}(K,L)$ be the pushforward of the uniform measure on the sphere by the map $\ect(K,-)\times\ect(L,r(-)): S^{d-1}\rightarrow \CF\times\CF$.
    It is clear that $\nu_{r}(K,L)\in \Gamma(\ect(K)_*\mu,\ect(L)_*\mu)$ for all $r\in O(d)$. Then we have the following inequality:
    
    \begin{align*}
        W_1(\ect(K)_*\mu&, \ect(L)_*\mu)  \leq \inf_{r\in O(d)}\left(\int_{\im(\ect(K))\times \im(\ect(L))} \|f-g\|_1 \d \nu_{r}(K,L)(f,g)\right) \\
        & = \inf_{r\in O(d)}\left(\int_{S^{d-1}} \|\ect(K,v) - \ect(L,r(v))\|_1 \d\mu(v)\right) \\
        &\leq \inf_{r\in O(d)}d_{\ect}^1(K, r(L)).
    \end{align*}
\end{proof}

This theorem not only shows that the pushforward measure is stable under Wasserstein distance but also more stable than normal ECT outputs, as it is always less than or equal to the minimal distance between two shapes up to $O(d)$ actions.

\subsubsection*{SampEuler}

The ECT pushforward measure captures all the information from shape classes of strongly generic complexes and distinguishes different shape classes. However, being a measure on a space of curves, it is hardly computable. Based on this idea, we propose the following discretization of it, called \textbf{SampEuler}, as a computable descriptor of shape classes.

\begin{definition}
Let $K\subset\R^d$ be a geometric simplicial complex and let $\ect(K)_*\mu$ be the ECT pushforward measure of $K$.  Randomly sample $(X_1,\dots,X_n)\stackrel{\mathrm{i.i.d.}}{\sim} \ect(K)_*\mu.$
The \emph{SampEuler} of $K$ of order $n$, denoted $\samp_n$, is defined by
$$
  \samp_n(K) \;=\; \frac{1}{n}\sum_{i=1}^n \delta_{X_i},
$$
where $\delta_x$ is the Dirac measure at $x$.
\end{definition}

In practice, the ECT is discretized by fixing a set of directions, sampled at a fixed rate along a chosen orientation of the sphere, and using the \emph{same} ordered set of directions for every complex; the resulting descriptor is therefore an ordered, direction-indexed object. SampEuler differs by drawing the directions uniformly at random and independently for each complex, and by retaining the resulting ECCs only as an \emph{unordered} collection, the empirical measure $\samp_n(K)$ of the ECT pushforward measure $\ect(K)_*\mu$, keeping no record of which direction produced which curve. This construction makes SampEuler a discretization of the full shape descriptor $\ect(K)_*\mu$ that is invariant under rotations and reflections. For shape classification tasks, this invariance is desirable, as the shape class of a complex is defined up to isometry and the information about the particular choice of representation of the shape class acts as noise affecting the performance of ECT. 

For a geometric simplicial complex $K$, both $\ect(K)_*\mu$ and $\samp_n(K)$ are measures on $\CF$. We show in \Cref{measure-convergence} that $\samp_n$ converges to the ECT pushforward measure under the Wasserstein metrics as the order $n$ goes to infinity. Therefore, SampEuler captures sufficient information about the shape class $[K]$ and would be more stable than normal ECT discretizations. 

We will make use of the following lemma in the proof, proven in \cite{curry2022many} along the proof of \Cref{pushforward-measure}.
\begin{lemma}\label{lemma: support=im}
    For any simplicial complex $K$ embedded in $\R^d$, the support of $\ect(K)_*\mu$ is the set $\text{Im}(\ect(K))$.
\end{lemma}
The proof of our convergence theorem is a consequence of the following classical result.

\begin{definition}
    Let $(M,d)$ be a separable metric space and $r\geq 1$. We write $\mathcal{P}_r(M)$ for the set of Borel probability measures $\mu$ on $M$ with finite $r$-th moment, that is, $\int_M d(x,x_0)^r\,\d\mu(x) <\infty$ for some (equivalently, any) $x_0\in M$. 
    
    For $\mu,\nu\in\mathcal{P}_r(M)$, the \deff{$r$-Wasserstein distance} is
    \[
        W_r(\mu,\nu) = \left(\inf_{\gamma\in\Gamma(\mu,\nu)} \int_{M\times M} d(x,y)^r\,\d\gamma(x,y)\right)^{1/r},
    \]
    where $\Gamma(\mu,\nu)$ is the set of couplings of $\mu$ and $\nu$, i.e.\ Borel probability measures on $M\times M$ with marginals $\mu$ and $\nu$. Equipped with $W_r$, the space $\mathcal{P}_r(M)$ is a metric space.
\end{definition}

\begin{theorem}[Convergence of empirical measure \cite{varadarajan1958convergence, bobkov2019one}]\label{measure-convergence}
    Suppose $(M,d)$ is a separable metric space and $X_1,\dots, X_n\stackrel{\text{i.i.d.}}{\sim} \mu\in\mathcal{P}_r(M)$ with $r\geq 1$. Then, 
        $$W_r(\mu_n, \mu)\rightarrow_{a.s.} 0 \text{ as $n\rightarrow \infty$ }.$$
\end{theorem}

\begin{theorem}\label{theorem: SampEuler convergence}
    Given a geometric simplicial complex $K\subset \R^d$, denote $\mu_n = \samp_n(K)$ and $\mu_K =\ect(K)_*\mu$. Then, $W_r(\mu_n, \mu_K)\rightarrow_{a.s.} 0 \text{ as $n\rightarrow \infty$ }$.
\end{theorem}

\begin{proof}
    It is sufficient to check that the conditions of \Cref{measure-convergence} are satisfied. We first check that $\mu_K \in \mathcal{P}_r(\CF)$. The Euler curves of $K$ are integer-valued, uniformly bounded, and differ only on a fixed compact interval, so there exists some constant $C_r$ such that $\|\ect(K,v_1)- \ect(K, v_2)\|_r\leq C_r$ for all $v_1, v_2\in S^{d-1}$. From \Cref{lemma: support=im}, we know that the support of $\mu_K$ is the set of curves in the image of $\ect(K)$. For any $f_0\in\im(\ect(K))$, we have $$\int_{\cf(\R)} \|f - f_0\|_r^r\,\text{d}\mu_K(f) = \int_{\im(\ect(K))}\|f - f_0\|_r^r\,\text{d}\mu_K(f)\leq \int_{\im(\ect(K))}C_r^r\,\text{d}\mu_K(f)= C_r^r,$$ as $\mu_K$ is a probability measure. Hence, $\mu_K\in \mathcal{P}_r(\CF)$.

    Next, we show that the space of possible Euler Characteristic Curves (ECCs) is separable.  We know from \Cref{continuity-and-t2} that this space is a Hausdorff metric space using $L_p$ metrics with $1\leq p < \infty$. Take curves that change values at n rational numbers; this set $F_n$ is in one-to-one correspondence with the set $(\Q\times\Z)^n$, which is countable.  Take the countable union $F = \cup_{i=0}^{\infty} F_n$. Since each constructible function $f\in \cf(\R)$ only has finitely many value changes, we can find $f'\in F$ arbitrarily close to $f$. Since $F$ is countable, the space $\cf(\R)$ is separable. 
\end{proof}

\subsubsection*{Vectorization of SampEuler}
As shown above, computing ECT along random directions using SampEuler has nice theoretical guarantees. However, for small biological datasets, this would lead to a high-dimensional feature, and the classification models would suffer from the curse of dimensionality. Interpreting the pushforward measure result also requires some visualization tools to help. To tackle the two problems, we propose the following vectorization algorithm, which can be used as a low-dimensional vectorization tool as well as a visualization tool of the SampEuler. It evaluates the empirical ECT pushforward measure, computed by SampEuler, on the sets of curves having constant value across certain intervals.

\begin{definition}
Let $K$ be an embedded simplicial complex in $\R^d$ and $l >0$. Denote by $\Int(\R)_l$ the set of closed intervals of length $l$ and $K_{(a,b),c}: = \{\phi\in \text{supp}(\samp_n(K)): \phi|_{(a,b)} = c\1_{(a,b)}\}$ for $a,b\in\R$ and $c\in\Z$. We define the \deff{vectorization of SampEuler}(or vectorized SampEuler) of $K$ as follows:
    \begin{align*}
        \mathcal{V}(K, l): \Int(\R)_l\times\Z &\rightarrow [0,1],\\
        ([a,b],k) &\mapsto \samp_n(K)(K_{[a,b],k}),
    \end{align*}
    We will often abbreviate the notation to be $\mathcal{V}(K)([a,b],k)$ when $l$ is clear from the context. 
\end{definition}

Note that the above definition is still well-defined when we use the ECT pushforward measure instead of the SampEuler of $K$. For theoretical purposes, we work with the definition using ECT pushforward measures. It is clear that if two complexes are equal up to isometry, they have the same ECT pushforward measure. Hence, the evaluations of this vectorization for both complexes are equal when evaluated using the ECT pushforward measure. This shows that such vectorization is isometry invariant.

\begin{remark}
    We can recover the average of ECT over all directions at given filtration values using the vectorization of SampEuler by: 
        $$\int_{S^{d-1}}\ect(K)(v,t) \d v=\int_{\R} \mathcal{V}(K,0)({t}, y) \d y.$$
    Using this observation, we can write down the explicit formula for computing DETECT \cite{marsh2022detecting} from the vectorization of SampEuler:
        \begin{align*}
            \text{DETECT}(K)(x) & = \int_{S^{d-1}}\int_{-a}^{x}\ect(K)(v,t)-\overline{\ect(K)(v,-)} \d t \d v\\
            & = \int_{-a}^{x}\int_{S^{d-1}}\ect(K)(v,t) \d v \d t-\int_{-a}^{x}\int_{S^{d-1}}\overline{\ect(K)(v,-)} \d v \d t\\
            & = \int_{-a}^{x}\int_{-\infty}^{\infty} \mathcal{V}(K,0)(t, y) \d y \d t - \left(\frac{x+a}{2a}\right)\int_{S^{d-1}}\int_{-a}^a \ect(K)(v,t) \d t \d v \\
            & = \int_{-a}^{x}\int_{-\infty}^{\infty} \mathcal{V}(K,0)(t, y) \d y \d t - \left(\frac{x+a}{2a}\right) \int_{-a}^a \int_{-\infty}^{\infty} \mathcal{V}(K,0)(t, y) \d y \d t.
        \end{align*}
    This shows that vectorization of SampEuler contains all the information from DETECT while still being isometry invariant.
\end{remark}

\subsubsection*{Computation Implementations}
In this section, we discuss how ECT, SampEuler, and the vectorization of SampEuler are computed in applications for geometric simplicial complexes and grayscale images. We adopt \Cref{lemma: alternativeECT} from \cite{curry2022many}.

\begin{definition}
    Let $K$ be a geometric simplicial complex. For each vertex $x\in K$, the \deff{lower star} of $X$ along direction $v$, denoted by $K^{(x,v)}$, is defined by: 
        $$K^{(x,v)} : = \{\sigma\subset K: \max_{y\in\sigma} y\cdot v = x\cdot v\}.$$
\end{definition}

\begin{lemma}\label{lemma: alternativeECT}
    Let $K$ be a geometric simplicial complex. Denote by $\text{ver}(K)$ the set of vertices of $K$. Then we have:
        $$\ect(K)(v, -) = \sum_{x\in \text{ver}(K)} \chi(K^{(x,v)})\1_{[x\cdot v, \infty)}.$$
\end{lemma}

Using \Cref{lemma: alternativeECT}, we compute the ECT of $K$ as follows. We first fix a set of directions along a chosen orientation. For each chosen direction $v$, we calculate the vertex of maximum filtration value of each $k$-simplex $[x_0,\dots,x_k]$, i.e.\ we find $x_* = \text{argmax}\{x_i\cdot v: i=0,\dots, k\}$. We assign this $k$-simplex to the lower star of $x_*$. We compute the Euler characteristic of the lower star of all vertices after assigning all simplices. Without loss of generality, we assume the filtration values of all vertices are bounded in the interval $[-a, a]$.  We densely sample $t$-values in the interval $[-a, a]$. For each $t$, we compute the sum $\sum_{\substack{x\in \text{ver}(K) \\ x\cdot v\leq t}\chi(K^{(x,v)})}$ as the value of $\ect(K)(v,t)$. We record all the outputs in order, and end up with an $m\times n$ matrix representing ECT, where $n$ is the number of directions computed and $m$ is the number of $t$-values evaluated along each direction. See \Cref{figure: ect illustration} for illustration. An illustration video can be found in \url{https://github.com/reddevil0623/Shape-transform-descriptor-for-thymus-structures}. The computation of ECT for grayscale images is done similarly, and see \cite{lebovici2024efficient} for more details.

The computation of SampEuler is similar, but with a random set of directions generated instead of the fixed set of directions. We now explain the computation of the vectorization based on the results of SampEuler. Without loss of generality, we assume the filtration values of all vertices along all directions are bounded within the interval $[-a, a]$ and the Euler characteristics of all possible sublevel sets are bounded within the interval $[-b,b]$. Let $\Z_b = \Z\cap[-b, b]$. We uniformly partition the space $[-a, a]\times \Z_b$ into constant intervals of the same length. We consider the matrix representation of SampEuler. For each curve represented as an array of numbers $[c_1,\dots, c_m]$, let $c_i,\dots, c_j$ be the entries with corresponding filtration value in the interval $l$. If $c_i,\dots, c_j$ are all equal to a constant $c$, we add 1 to the corresponding entry in the vectorization result for interval $l\times c$. Otherwise, nothing changes. We repeat the process for all curves and divide the resulting matrix by $n$, the number of directions, to normalize. 

When applying SampEuler and its vectorization, we generate the $n$ directions using \texttt{numpy.\allowbreak random.\allowbreak uniform} \cite{harris2020array}, so that directions are drawn independently and uniformly from the unit circle, satisfying the assumption of \Cref{theorem: SampEuler convergence}. The resulting descriptor therefore approximates the ECT pushforward measure and, once sufficiently many directions are sampled, is insensitive to the particular random seed (\Cref{section: ablation}). All parameters are held constant across all samples in a given task, which ensures consistency between the training and test sets when the method is applied to classification tasks.

\begin{figure}
\centering
\includegraphics[width=\textwidth]{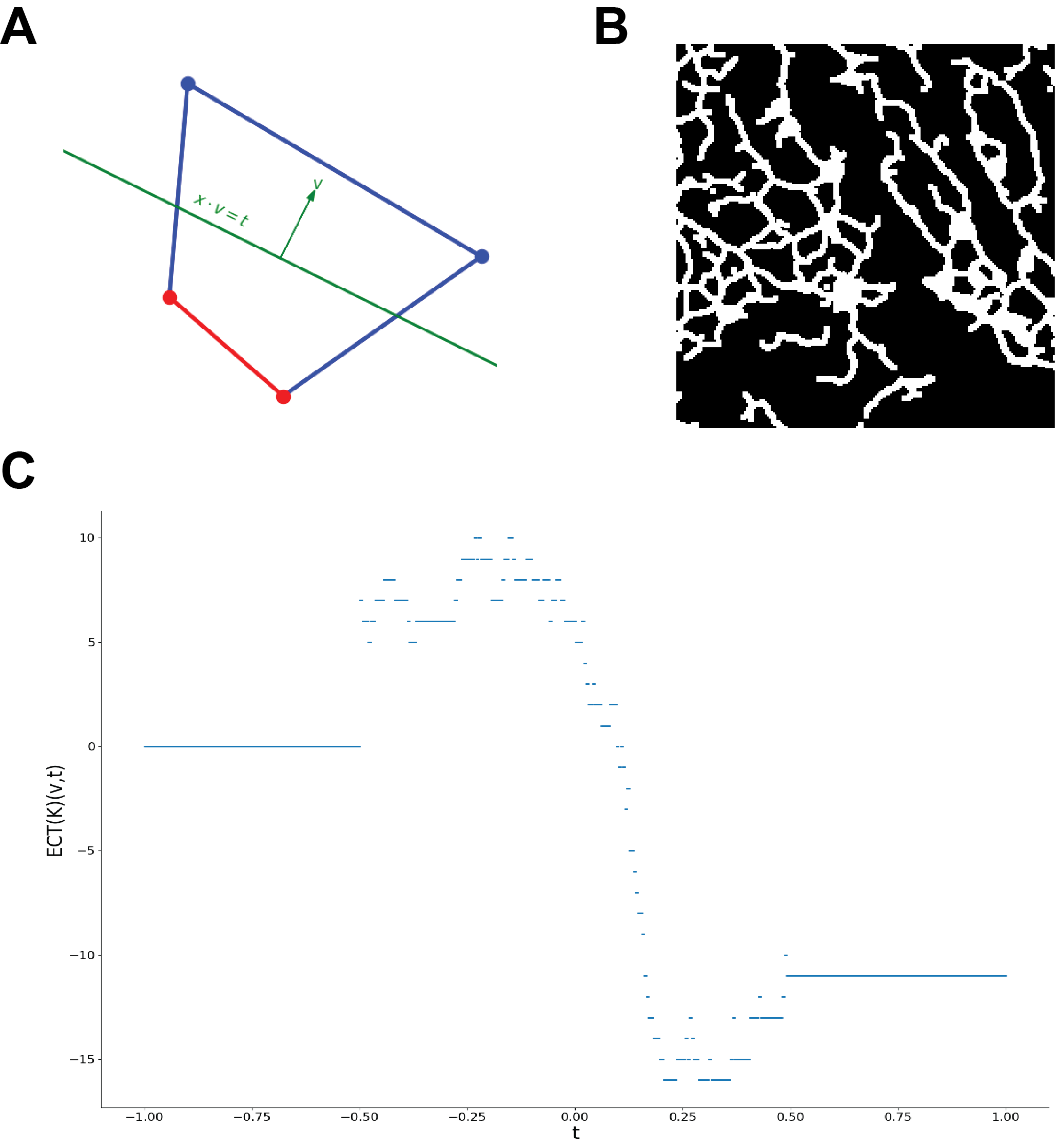}
\caption{\textit{A}: Schematic showing how the Euler characteristic transform (ECT) is calculated for complex $K$ along direction $v$ at filtration value t. The hyperplane $x\cdot v = t$ is drawn to determine the half-space $x\cdot v<t$. The vertices of $K$ lying in this half space, together with their lower stars, are considered as plotted in red. The Euler characteristic of the red part of the shape is then computed as the output $\ect(K)(v,t)$. \textit{B}: A segmented binary quadrant of mouse thymi. The image is of size 200 pixels $\times$ 200 pixels. We use \cite{lebovici2024efficient} to create a square complex of diagonal length 1 based on the image for computing ECT. \textit{C}: Euler characteristic curve along direction $(1,0)$ for the image in B. The initial increase in the curve value accounts for the emergence of disjoint white pixels on the left edge of the image. As we include more of the image towards the right, white pixels connect and form loops, aligning with the fast decrease starting around filtration value $-0.25$. The final increasing trend at around filtration value $0.25$ is again due to the increase in connected components on the right end of the image, due to the cutting of the quadrant.}
\label{figure: ect illustration}
\end{figure}

\subsection{Other Methods}
\subsubsection*{Persistence Homology with Signed Distance Function}
We use \cite{scikit-fmm} to compute the signed distance function filtration for each thymus image by taking the boundary as the boundary contour of the white pixels. We apply \cite{adams2017persistence} to compute degree 0 and degree 1 Persistence Image for each image using signed distance function filtration with a Gaussian kernel $\sigma = 1.0$ and uniform weight. 

\subsubsection*{Enrichment ratio}
To characterize changes in cell-type composition within each thymic region, a naive approach is to examine the raw proportion of each cell type in that region. However, with more than 40 cell types, Presel DP comprises around $60\%$ in both age groups. This makes changes in the remaining types be obscured when using proportions alone. To better capture meaningful differences across all cell types, we instead use a cell enrichment ratio, defined as follows.

\begin{definition}
    Consider quadrant $q\in Q$, and cell type $j\in J$. Let $C_{q,j}$ be the number of type $j$ cells in quadrant $q$. Denote 
        $$r_{q,j} = \frac{C_{q,j}+\alpha}{\sum_{k\in J}C_{q,k} +\alpha J},$$
    where $\alpha$ is the Dirichlet smooth constant set to avoid quadrants with few cells, and $r_{q,j}$ is the smoothed ratio of type $j$ cell in quadrant $q$. Similarly, we denote 
        $$Y_j = \frac{(\sum_{q'\in Y}C_{q',j})+\alpha}{(\sum_{k\in J}\sum_{q'\in Y}C_{q',j})+\alpha J}$$
    as the ratio of type $j$ cells in all young quadrants to be used as the baseline. We define the \textbf{enrichment ratio} for quadrant $q$ and cell type $j$ as 
        $$E_{q,j} = \frac{r_{q,j}}{Y_j}.$$
\end{definition}

We focus on cell types: DN3, Postsel DP(CD69+), Postsel DP(CD69-), Presel DP. For each quadrant $q$, we have a cell feature vector $(E_{q,1},\dots, E_{q,n})$. We equip the cell feature vector space with the $L_2$ Euclidean norm. 

\section{Data Analysis}
\subsection{SampEuler Examples}
In this section, we show how vectorization of SampEuler provides interpretable topological features. The vectorization creates an intensity map of ECC curves across all directions, revealing typical topological patterns for a given shape. We demonstrate how these patterns can be interpreted in terms of the underlying geometric structure.

We generate sets of simplicial complex representations of disjoint ellipses of major axis length 2 and minor axis length 1. We sample 80 points around the boundary of each ellipse and apply Gaussian noise with standard deviation $0.05$. We create $n$ such ellipses by randomly sampling the position of the centre and the angle of rotation. We fill all ellipses.  See \Cref{figure:eulerimage_example}. 

We first show that the vectorized SampEuler can detect the connectedness of the input complex. In \Cref{figure:eulerimage_example}\textit{A-B}, the ellipse centres are uniformly distributed in a square of side length 50. \Cref{figure:eulerimage_example}\textit{A} contains 50 such ellipses, and \Cref{figure:eulerimage_example}\textit{B} contains 40 such ellipses. By definition, the final converging value of the vectorized SampEuler is equal to the total Euler characteristic of the input complex. In this case, the Euler characteristic is equal to the number of connected components. The converging value is 50 in \Cref{figure:eulerimage_example}\textit{A} and 40 in \Cref{figure:eulerimage_example}\textit{B}, indicating that \Cref{figure:eulerimage_example}\textit{A} has more connected components. 

In \Cref{figure:eulerimage_example}\textit{C-D}, we show that vectorized SampEuler captures different sampling distributions used for sampling ellipse centres. The centres are restricted to quadrants 2, 3, and 4 for \Cref{figure:eulerimage_example}\textit{C}. Two general trends appear in the Euler characteristic curves. The first trend begins near filtration value $-30$ and rises to a peak near  $+10$. This corresponds to directions pointing into quadrant 1: the growing sublevel set first encounters the ellipses at filtration values comparable to those in \Cref{figure:eulerimage_example}\textit{A}. However, the final ellipse has a lower filtration value than in \Cref{figure:eulerimage_example}\textit{A} due to the missing corner in the first quadrant, and the ellipses are more densely packed in \Cref{figure:eulerimage_example}\textit{C}, explaining the steeper trend of the curve. The second trend starts around $-20$ and peaks at about $+20$. This reflects curves along directions in the third quadrant: in the absence of ellipses in the first quadrant, the first ellipse is included later than in \Cref{figure:eulerimage_example}\textit{A}, shifting both the start and peak to higher filtration values. The ending filtration value matches because both \Cref{figure:eulerimage_example}\textit{C} and \Cref{figure:eulerimage_example}\textit{A} include their final ellipses, the corner ellipse in the third quadrant, at similar heights. For all other directional curves, the increasing section lies between these two extremes, yielding the ``quadrilateral'' region in the vectorized SampEuler plot.

The vectorized SampEuler in \Cref{figure:eulerimage_example}\textit{D} displays a funnel-shaped region bounded approximately by lines connecting $(-40,0)$ to $(40,50)$ and $(-10,0)$ to $(10,50)$. This pattern reflects the geometry of the elliptical sampling region. For directions near the minor axis, ellipses first appear at higher filtration values but are captured within a short interval, producing steep ECC curves. For directions near the major axis, ellipses appear earlier but span a longer filtration range, producing shallower curves. The funnel shape emerges as ECCs transition continuously between these extremes across all sampling directions.

\begin{figure}
  \centering
  \includegraphics[width=\textwidth, 
                   height=0.9\textheight, 
                   keepaspectratio]{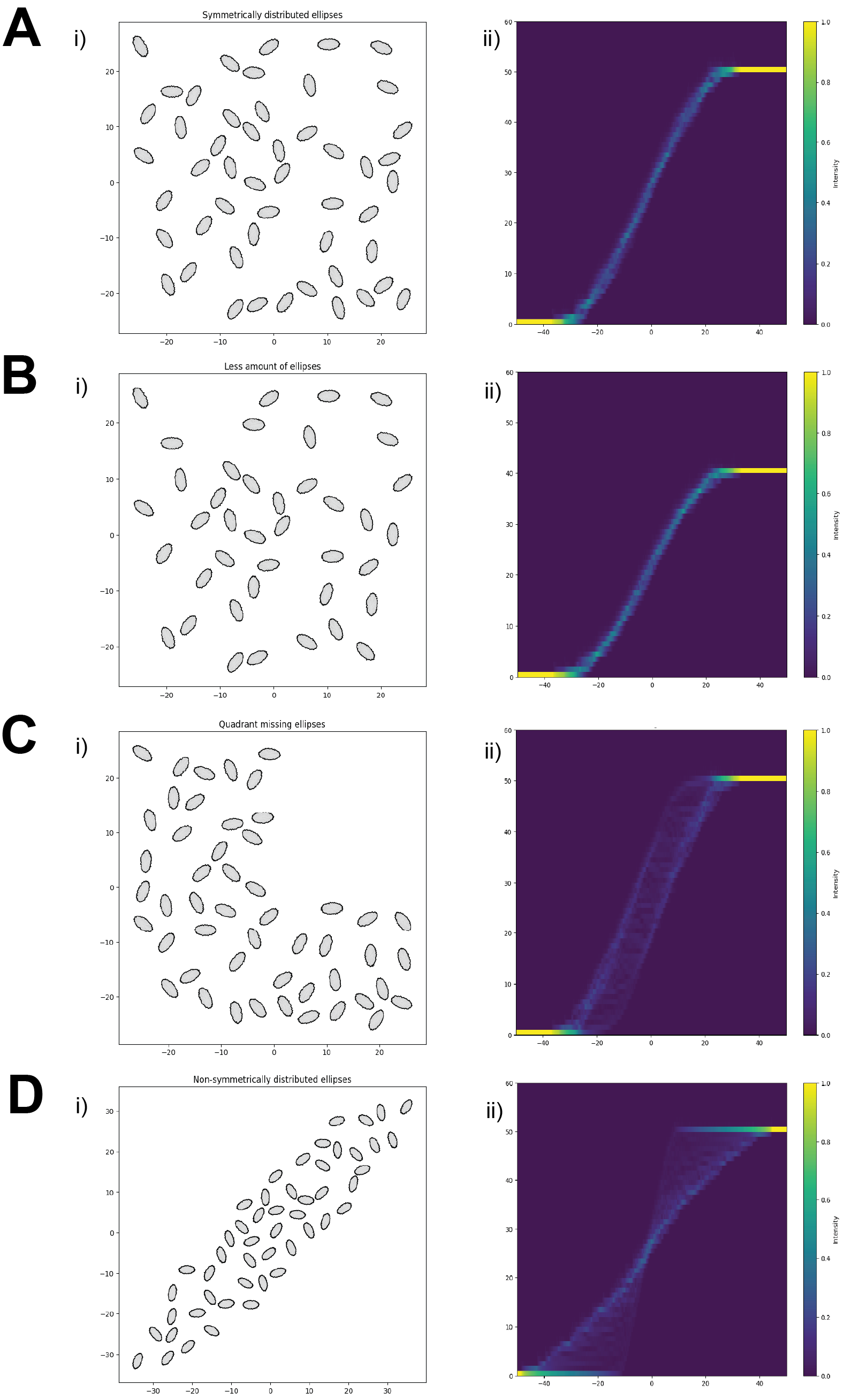}
  \caption{\textit{A}: i) 50 geometric simplicial complex representations of ellipses with centres uniformly sampled from the centred square of side length 50. ii) The corresponding vectorized SampEuler result.
  \textit{B}: i) 40 geometric simplicial complex representations of ellipses with centres uniformly sampled from the centred square of side length 50. ii) The corresponding vectorized SampEuler result.
  \textit{C}: i) 50 geometric simplicial complex representations of ellipses with centres uniformly sampled from the restriction of quadrant 2, 3, 4 of the centred square of side length 50. ii) The corresponding vectorized SampEuler result.
  \textit{D}: i) 50 geometric simplicial complex representations of ellipses with centres uniformly sampled from a centred ellipse with major axis length 100 and minor axis length 20. ii) The corresponding vectorized SampEuler result.}
  \label{figure:eulerimage_example}
\end{figure}

\subsection{Comparison to ECT and DETECT}

In this section, we compare SampEuler and vectorized SampEuler to established topological shape descriptors. This example demonstrates that SampEuler and vectorized SampEuler capture sufficient information about the shape class of the input complex.

We use geometric simplicial complexes explained in \Cref{section: toy sample} to illustrate that SampEuler and its vectorization are more robust to perturbations than the ECT, and are more informative than DETECT. We run two experiments for each method. In the first experiment, we generate two families of 20 complexes without rotating them. In the second experiment, we generate complexes, with each complex randomly rotated about the origin. We use mdscale \cite{pedregosa2011scikit} to visualize the results. As shown in \Cref{figure:eulerimage_example_results}, ECT only distinguishes the two families of shapes correctly when all complexes are aligned. However, when we randomly rotate the complexes, the ECT fails to distinguish the two families. For DETECT, information loss during averaging means that DETECT fails to distinguish the two families in both experiments. However, with SampEuler and vectorized SampEuler, clear separation of the two classes can be seen in both experiments, showing that they are robust to isometries and contain sufficient information about the shape of input complexes.

\begin{figure}
  \centering
  \includegraphics[width=\textwidth, 
                   height=0.95\textheight, 
                   keepaspectratio]{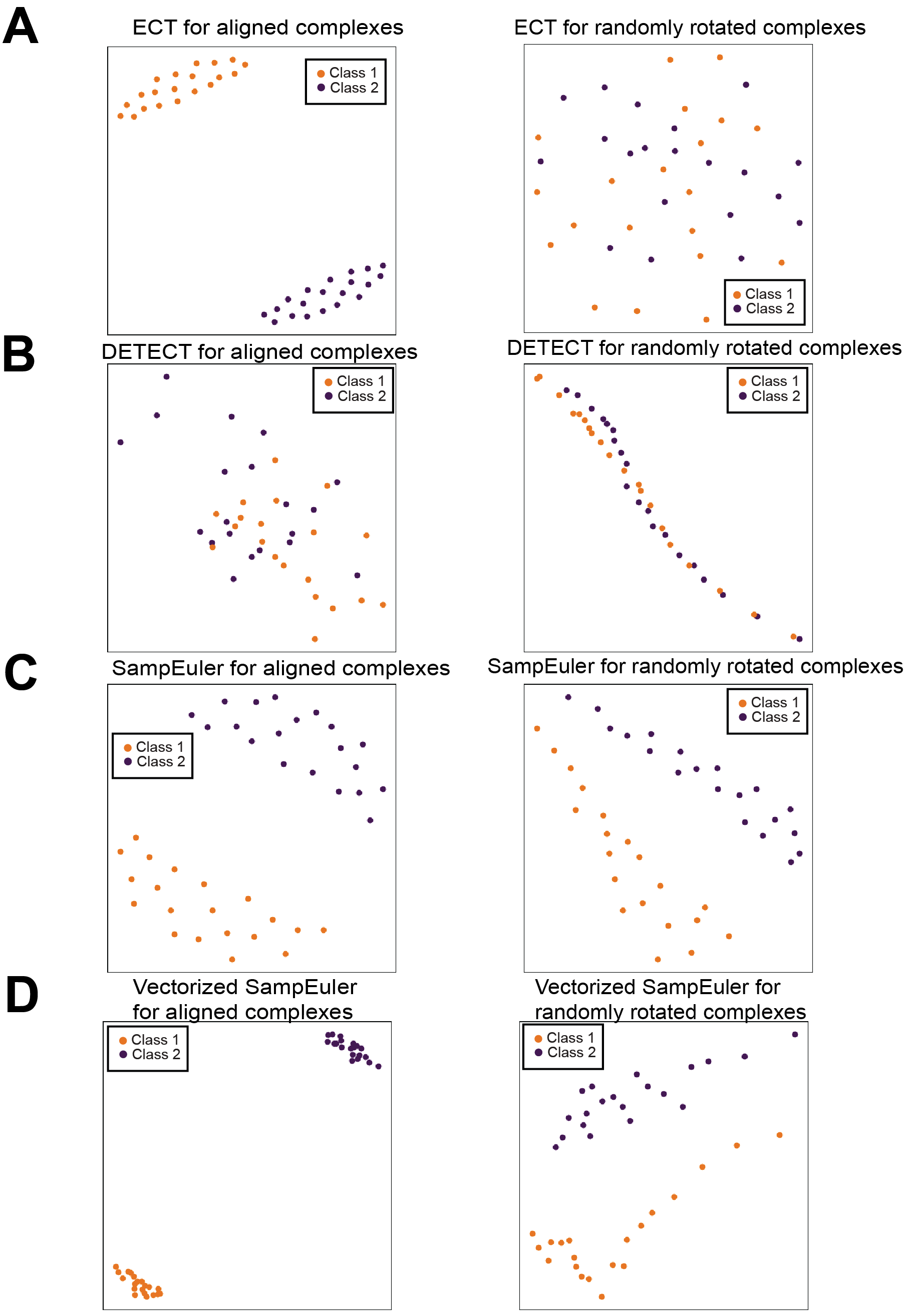}
  \caption{ MDS visualizations of pairwise similarities between analysis results of two classes of generated geometric simplicial complexes. The left plots are when all complexes are aligned, and the right plots are when all complexes are randomly rotated. \textit{A}: results of ECT. \textit{B}: results of DETECT. \textit{C}: results of SampEuler. \textit{D}: results of vectorized SampEuler.}
  \label{figure:eulerimage_example_results}
\end{figure}

\subsection{MPEG-7 Dataset for Shape Classification}
To test against existing methods and understand the impact of discretization on analysis results, we apply our methods to the MPEG-7 dataset with simple normalization steps, as described in \Cref{section: mpeg7}. We first compare with methods mentioned in \cite{le2018persistence, le2019tree}. We use the 70/30 train-test-split ratio and repeat the experiment 100 times to report the average accuracy. To be consistent with other topological pipelines, we use the Support Vector Machine (SVM) classification model \cite{cortes1995support}. We compute the pairwise Wasserstein distance for SampEuler and $L_2$-distances for vectorized SampEulers and construct the RBF-kernel as input to the SVMs. Classification accuracy is reported using the best-performing parameters from the ablation study (\Cref{section: ablation}).

\begin{longtable}{L{10cm} L{3cm}}
\caption{The average accuracy $(\%)$ of different methods on the 10-class subset of the MPEG7 dataset.}
\label{table: accuracy 1} \\

\toprule
\textbf{Method} & \textbf{Accuracy} \\
\midrule
\endfirsthead

\multicolumn{2}{c}{{\bfseries \tablename\ \thetable{} -- continued from previous page}} \\
\toprule
\textbf{Method} & \textbf{Accuracy} \\
\midrule
\endhead

\midrule
\multicolumn{2}{r}{{Continued on next page}} \\
\endfoot

\bottomrule
\endlastfoot

Persistence Scale Space Kernel \cite{reininghaus2015stable} & 73.33 \\
\hline
Persistence Weighted Gaussian Kernel \cite{kusano2016persistence} & 74.83 \\
\hline
Sliced Wasserstein Kernel \cite{carriere2017sliced} & 76.83 \\
\hline
Tangent Vector Representation with Gaussian Kernel \cite{anirudh2016riemannian} & 66.17 \\
\hline
Persistence Fisher Kernel \cite{le2018persistence} & 80.00 \\
\hline
Tree Sliced Wasserstein Kernel \cite{le2019tree} & (80.00, 85.00)\textsuperscript{*} \\
\hline
QUPID \cite{van2025discrete} & 68.80 \\
\hline
ECT \cite{turner2014persistent} & 86.23 \\
\hline
DETECT \cite{marsh2022detecting} & 51.02 \\
\hline
SampEuler & \textbf{91.67} \\
\hline
vectorized SampEuler & \textbf{94.39} \\
\hline
\multicolumn{2}{p{13cm}}{\footnotesize \textsuperscript{*} Exact value not reported; range estimated from the bar chart in the original text.} \\
\end{longtable}

To compare with other dedicated 2D pattern recognition methods and topological methods with deep learning, we also perform the classification task on the whole MPEG7 dataset with a 90/10 train-test-split ratio and repeat the experiments 10 times to report the average accuracy.

\begin{longtable}{L{9cm} L{3cm}}
\caption{The average accuracy (\%) of different methods on the MPEG7 dataset.}
\label{table: accuracy 2} \\

\toprule
\textbf{Method} & \textbf{Accuracy} \\
\midrule
\endfirsthead

\multicolumn{2}{c}{{\bfseries \tablename\ \thetable{} -- continued from previous page}} \\
\toprule
\textbf{Method} & \textbf{Accuracy} \\
\midrule
\endhead

\midrule
\multicolumn{2}{r}{{Continued on next page}} \\
\endfoot

\bottomrule
\endlastfoot

Skeleton Paths \cite{bai2009integrating} & 86.7 \\
\hline
Bag of Contours \cite{wang2014bag} & \textbf{97.2} \\
\hline
Persistence Images + NN \cite{adams2017persistence} & 92.3 \\
\hline
Learned Persistence Codebooks \cite{hofer2019learning} + NN & 92.7 \\
\hline
ECT \cite{turner2014persistent} & 71.3 \\
\hline
DETECT \cite{marsh2022detecting} & 36.4 \\
\hline
SampEuler & \cellcolor[HTML]{67FD9A}91.4 \\
\hline
vectorized SampEuler & \cellcolor[HTML]{67FD9A}86.8 \\
\hline
\multicolumn{2}{p{12cm}}{\footnotesize \colorbox[HTML]{67FD9A}{\phantom{X}} indicates our proposed methods.} \\

\end{longtable}

As shown in \Cref{table: accuracy 1}, SampEuler and vectorized SampEuler provide better classification accuracies than all other methods, with vectorized SampEuler attaining the highest accuracy, likely because its vectorization step reduces feature dimensionality while maintaining all shape information. In the larger dataset of \Cref{table: accuracy 2}, Bag-of-Contours and neural-network–-based methods attain higher accuracies than our approaches. The Bag-of-Contours method in \cite{wang2014bag} is tailored to 2D shape recognition and relies on precise boundary-contour segmentation, which can be computationally intensive and dataset-dependent. Likewise, neural networks often achieve state-of-the-art accuracy on sufficiently large training sets, though their learned representations are uninterpretable and demand substantial computational resources. Although SampEuler and vectorized SampEuler fall short of these top accuracies, they are segmentation-free, applicable to general shape analysis, and come with theoretical guarantees that support interpretability. In this larger dataset, the original SampEuler achieves better performance than its vectorized variant, presumably because its higher-resolution features can leverage the increased sample size to capture finer geometric details.

Across both datasets, SampEuler and vectorized SampEuler consistently outperform ECT- and DETECT-based baselines. DETECT suffers substantial information loss from the averaging step, leading to the worst performance overall. The high sensitivity of ECT to perturbations causes instability in classification tasks.

\subsection{Ablation Study}\label{section: ablation}
We use the MPEG7 dataset to perform the ablation study of the SampEuler method. The results are shown in \Cref{fig:xpoints-ablation} and \Cref{fig:k-ablation}. These results suggest that if we sample more than 100 directions and 1000 points along each direction, the additional information yields diminishing returns in classification accuracy. Moreover, leveraging the optimised implementation in Eucalc \cite{lebovici2024efficient} enables efficient computation of these descriptors, making the above parameter choices practical in routine use.

\newcommand{\xpointlist}{100,1000,3000,5000,7000,9000,11000,13000,15000}
\begin{figure}[p]
  \centering
  \begin{adjustbox}{max size={\textwidth}{0.95\textheight}}
    \begin{minipage}{\textwidth}
      \foreach \x [count=\i] in \xpointlist{%
        \begin{subfigure}{0.32\textwidth}
          \centering
          \includegraphics[width=\linewidth,keepaspectratio]%
            {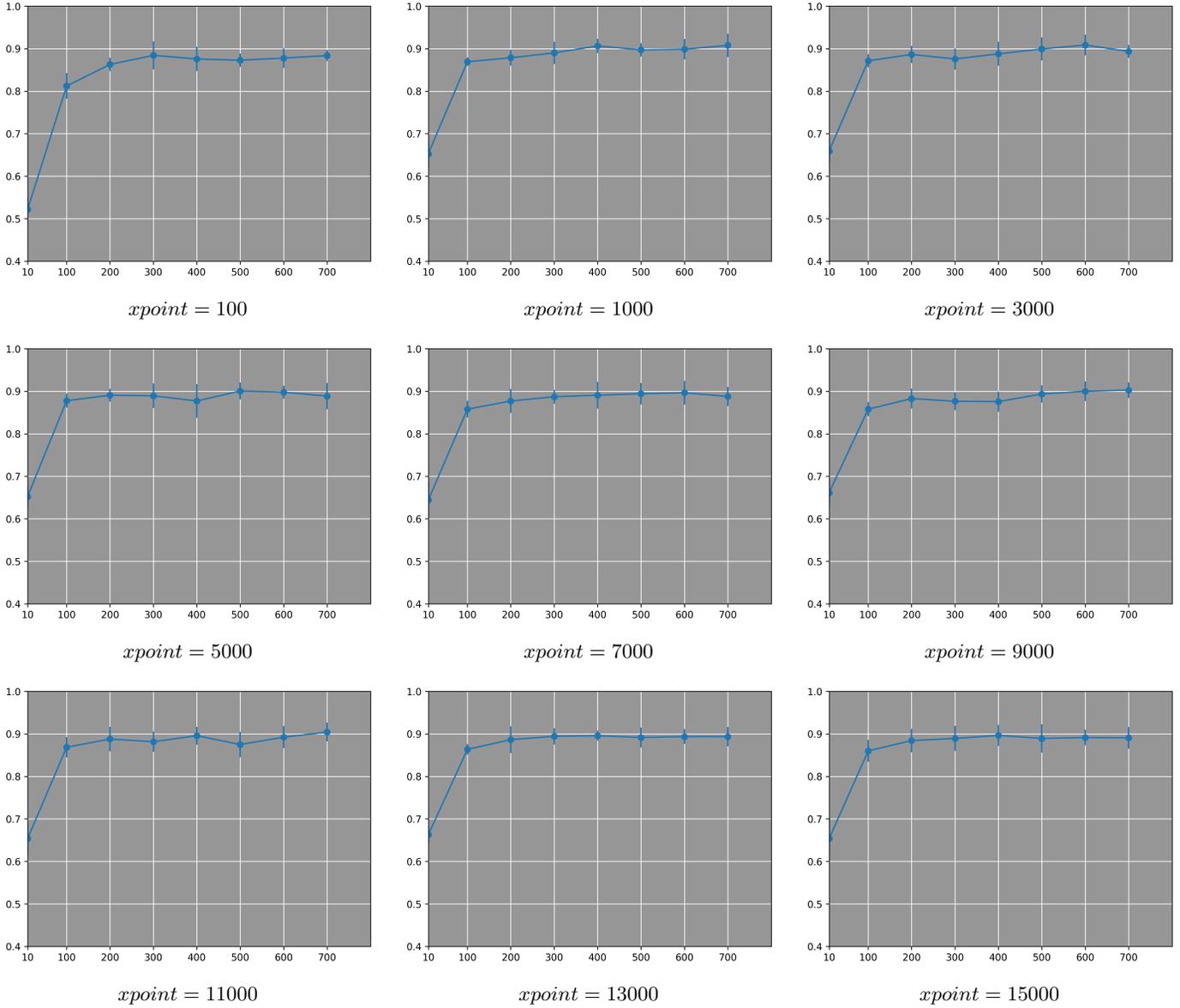}
          \caption*{$xpoint=\x$}
        \end{subfigure}\hfill
        \ifnum\intcalcMod{\i}{3}=0 \par\medskip\fi
      }
    \end{minipage}
  \end{adjustbox}
  \caption{Ablation study of SampEuler by fixing the number of directions sampled and varying the number of filtration values along each direction.}
  \label{fig:xpoints-ablation}
\end{figure}

\newcommand{\klist}{10,100,200,300,400,500,600,700}

\begin{figure}[p]
  \centering
  \begin{adjustbox}{max size={\textwidth}{0.95\textheight}}
    \begin{minipage}{\textwidth}
      \foreach \k [count=\i] in \klist{%
        \begin{subfigure}{0.32\textwidth}
          \centering
          \includegraphics[width=\linewidth,keepaspectratio]%
            {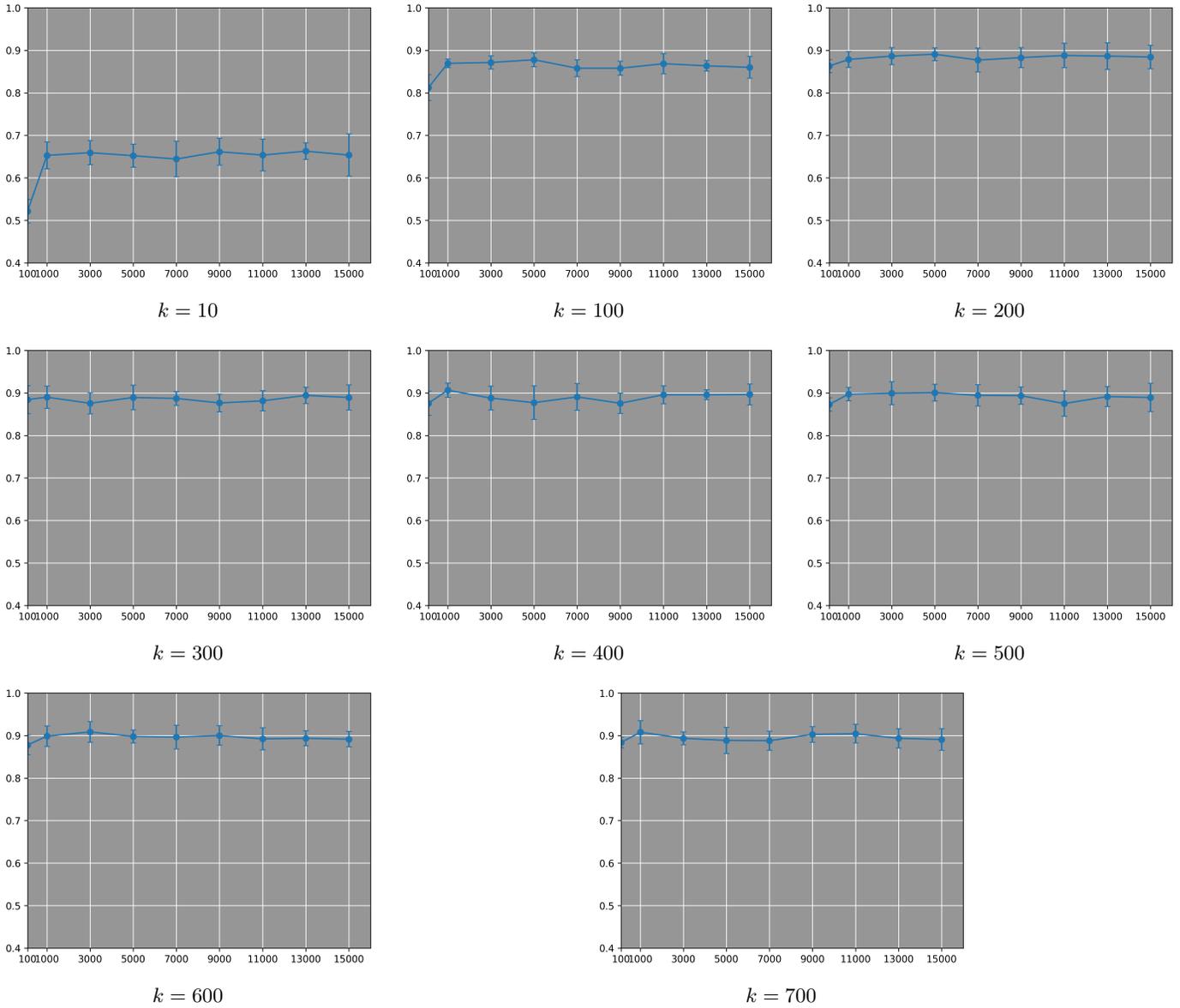}
          \caption*{$k=\k$}
        \end{subfigure}\hfill
        \ifnum\intcalcMod{\i}{3}=0 \par\medskip\fi
      }
    \end{minipage}
  \end{adjustbox}
  \caption{Ablation study of the SampEuler by fixing the number of filtration values along each direction and varying the number of directions sampled.}
  \label{fig:k-ablation}
\end{figure}

The grids in \Cref{fig:xpoints-ablation} and \Cref{fig:k-ablation} vary a single sampling parameter over a wide range for SampEuler alone and confirm that accuracy saturates well at low sampling rates. To examine this low-sampling regime more closely, and to better understand the comparisons with the ECT and DETECT baselines on the same axes, we ran a focused sweep on the full 70-class MPEG-7 dataset: we varied the number of directions $k$ (with $x = 1000$ fixed) and the number of filtration points $x$ (with $k = 100$ fixed) over the low range, and recorded the classification accuracy, the feature-generation time, and the total computation time for SampEuler, ECT, and DETECT (\Cref{fig:r23-ablation}).

As shown in \Cref{fig:r23-acc-k} and \Cref{fig:r23-acc-x}, all three methods quickly reach their peak accuracy, even at low sampling rates. SampEuler attains the highest accuracy, followed by ECT and then DETECT. The reason is that these classification tasks are concerned only with the shape class, so the orientation at which a shape is presented is pure noise; in particular, the MPEG-7 dataset contains the same objects rotated to different angles. The ECT records the full direction information and therefore carries this noise, so two complexes that differ only by a rotation yield very different ECT descriptors and the ECT caps well below SampEuler. SampEuler instead discards the direction labels and converges to the ECT pushforward measure, which is invariant up to isometry, so it is insensitive to this orientation noise and attains higher accuracy. DETECT is lowest throughout, owing to the information loss in its averaging step. The variance of the SampEuler accuracy is also markedly lower than that of ECT and DETECT, consistent with SampEuler being a more stable descriptor of the shape class, in keeping with its theoretical convergence properties. \Cref{fig:r23-time-k} and \Cref{fig:r23-time-x} show the feature-generation time for the three methods. Since all computations are based on the fast ECT implementation of~\cite{lebovici2024efficient}, the feature-generation time is essentially method-independent and scales linearly with both $k$ and $x$; the only minor differences come from the random-direction generation in SampEuler and the smoothing and averaging in DETECT, which are negligible compared with the overall cost. \Cref{fig:r23-total-k} and \Cref{fig:r23-total-x} show the total computation time, including feature generation, distance computation, and SVM training and testing. Here the three methods diverge sharply: SampEuler's total time is higher and grows with both $k$ and $x$, because computing its Wasserstein distance is more expensive than the simple curve distances used by ECT and DETECT, whose total time stays low and nearly flat. This additional cost is nonetheless well justified: SampEuler already attains its peak accuracy in the inexpensive low-sampling regime, and it is, to our knowledge, the only construction that achieves isometry invariance while retaining rich geometric and topological information about the input. It is, moreover, avoidable: used directly as input to the classifier (as in our thymus analysis) or through its vectorization, which is compatible with the $L^2$ metric, SampEuler incurs essentially the same total cost as ECT and DETECT while still outperforming them in accuracy.

\begin{figure}[htbp]
  \centering
  \begin{subfigure}{0.45\textwidth}\centering
    \includegraphics[width=\linewidth]{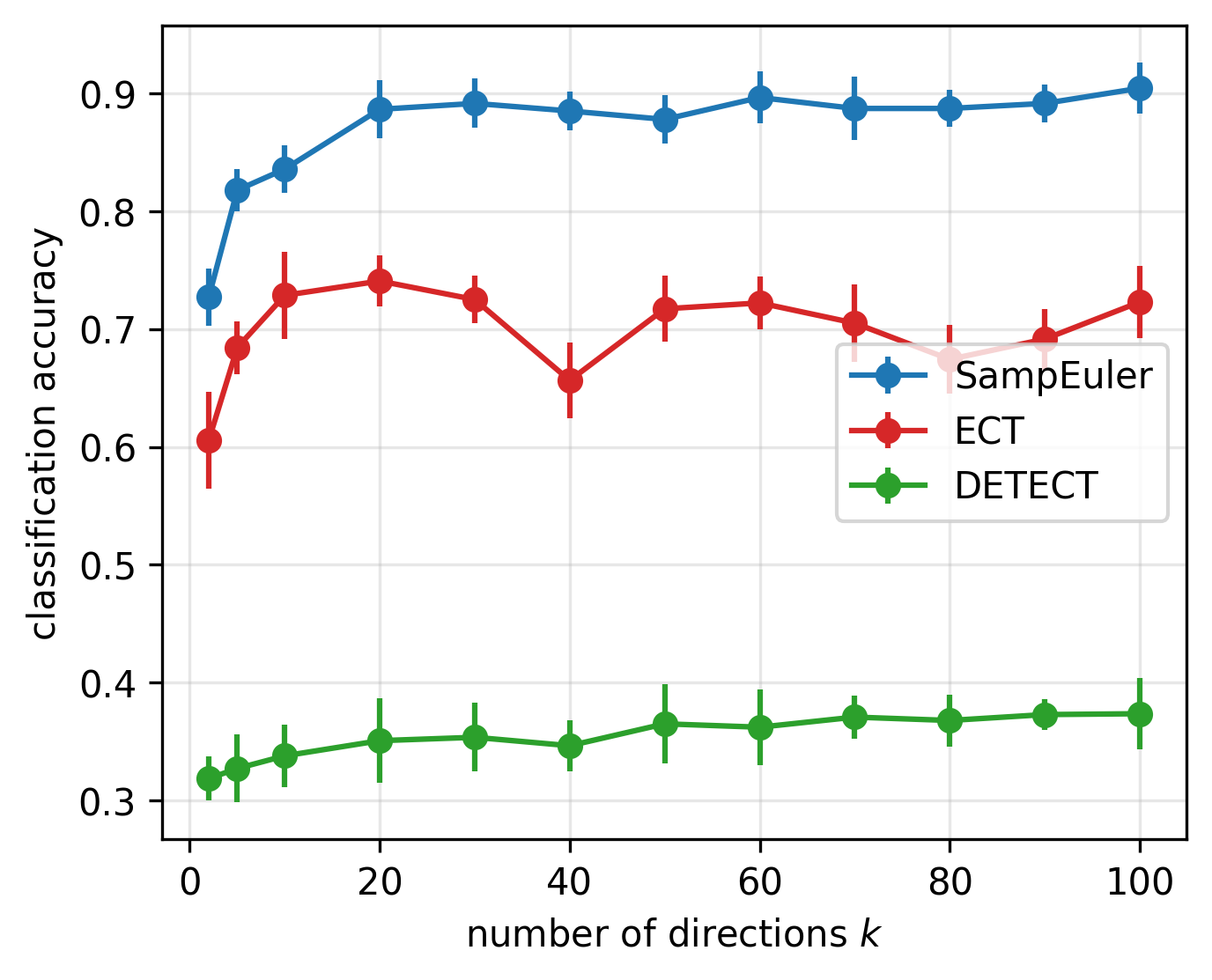}
    \caption{}\label{fig:r23-acc-k}
  \end{subfigure}\hfill
  \begin{subfigure}{0.45\textwidth}\centering
    \includegraphics[width=\linewidth]{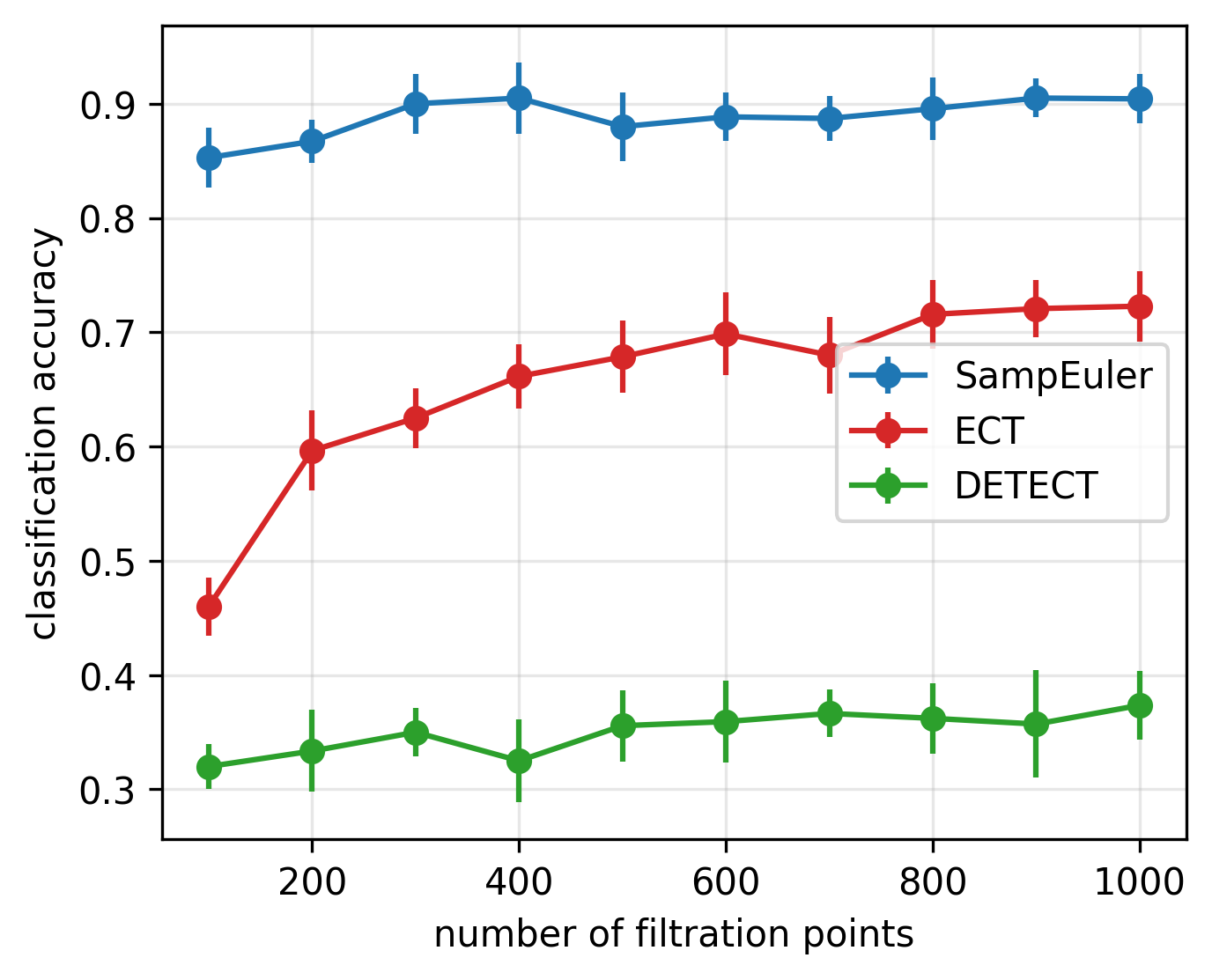}
    \caption{}\label{fig:r23-acc-x}
  \end{subfigure}

  \begin{subfigure}{0.45\textwidth}\centering
    \includegraphics[width=\linewidth]{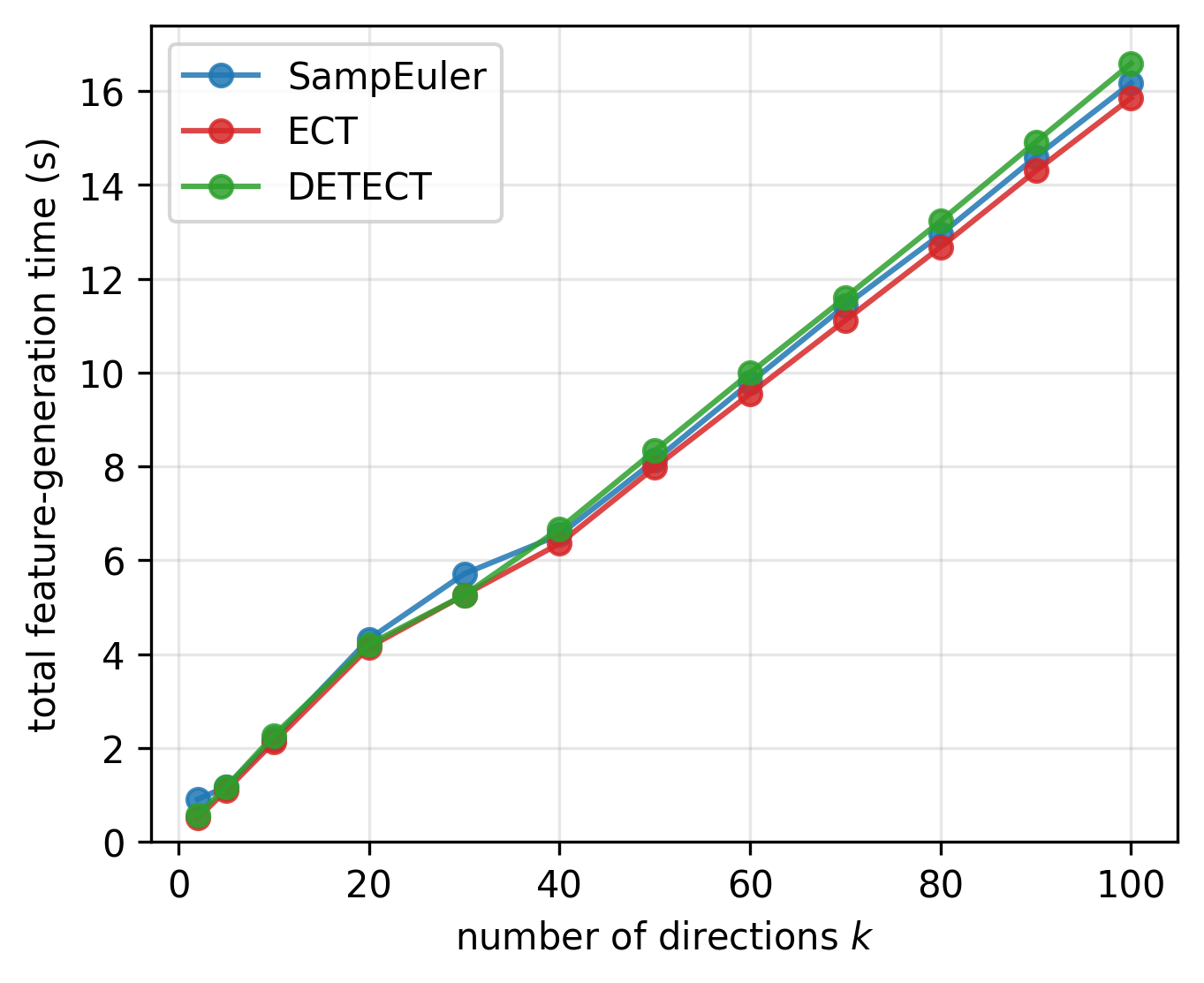}
    \caption{}\label{fig:r23-time-k}
  \end{subfigure}\hfill
  \begin{subfigure}{0.45\textwidth}\centering
    \includegraphics[width=\linewidth]{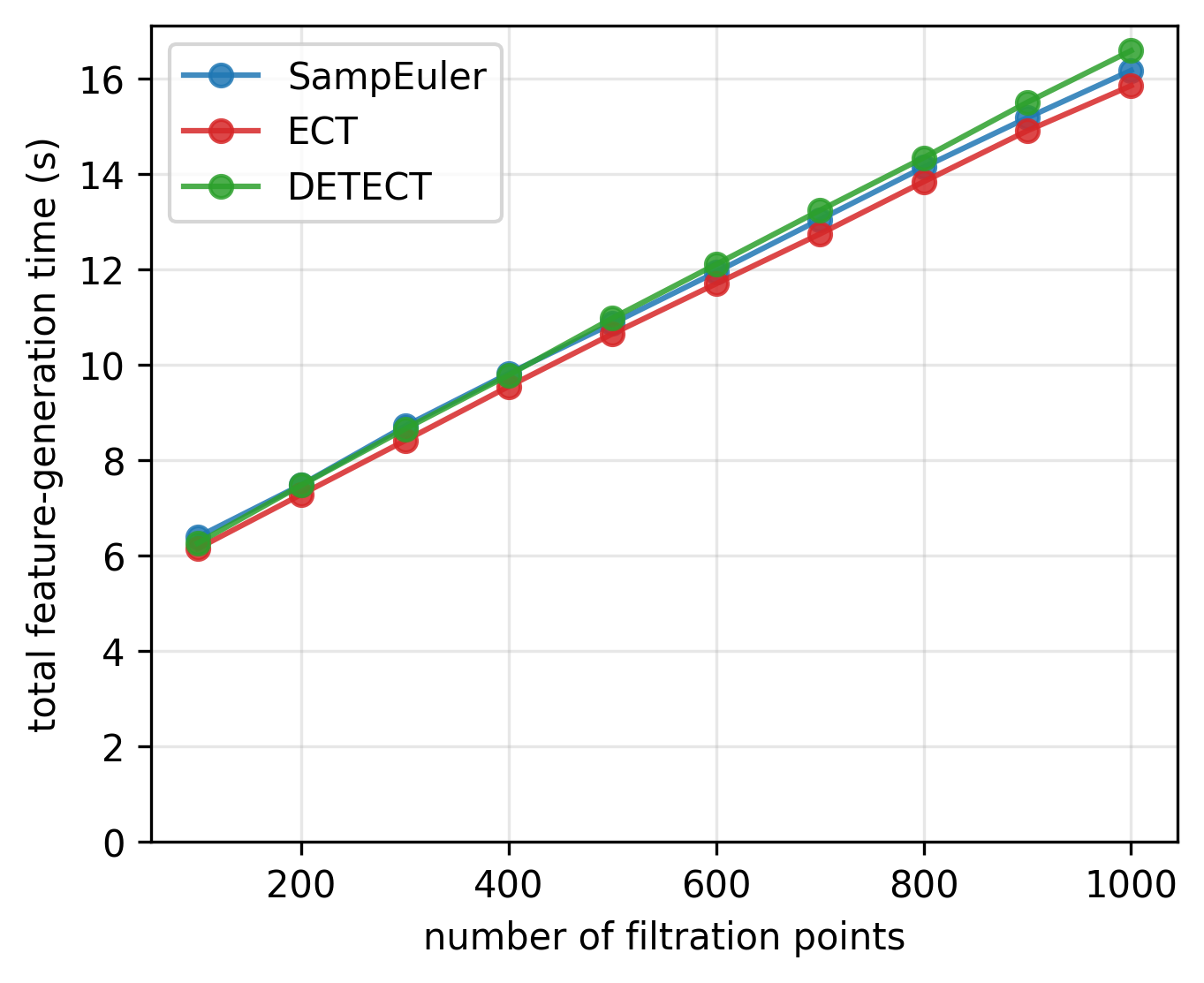}
    \caption{}\label{fig:r23-time-x}
  \end{subfigure}

  \begin{subfigure}{0.45\textwidth}\centering
    \includegraphics[width=\linewidth]{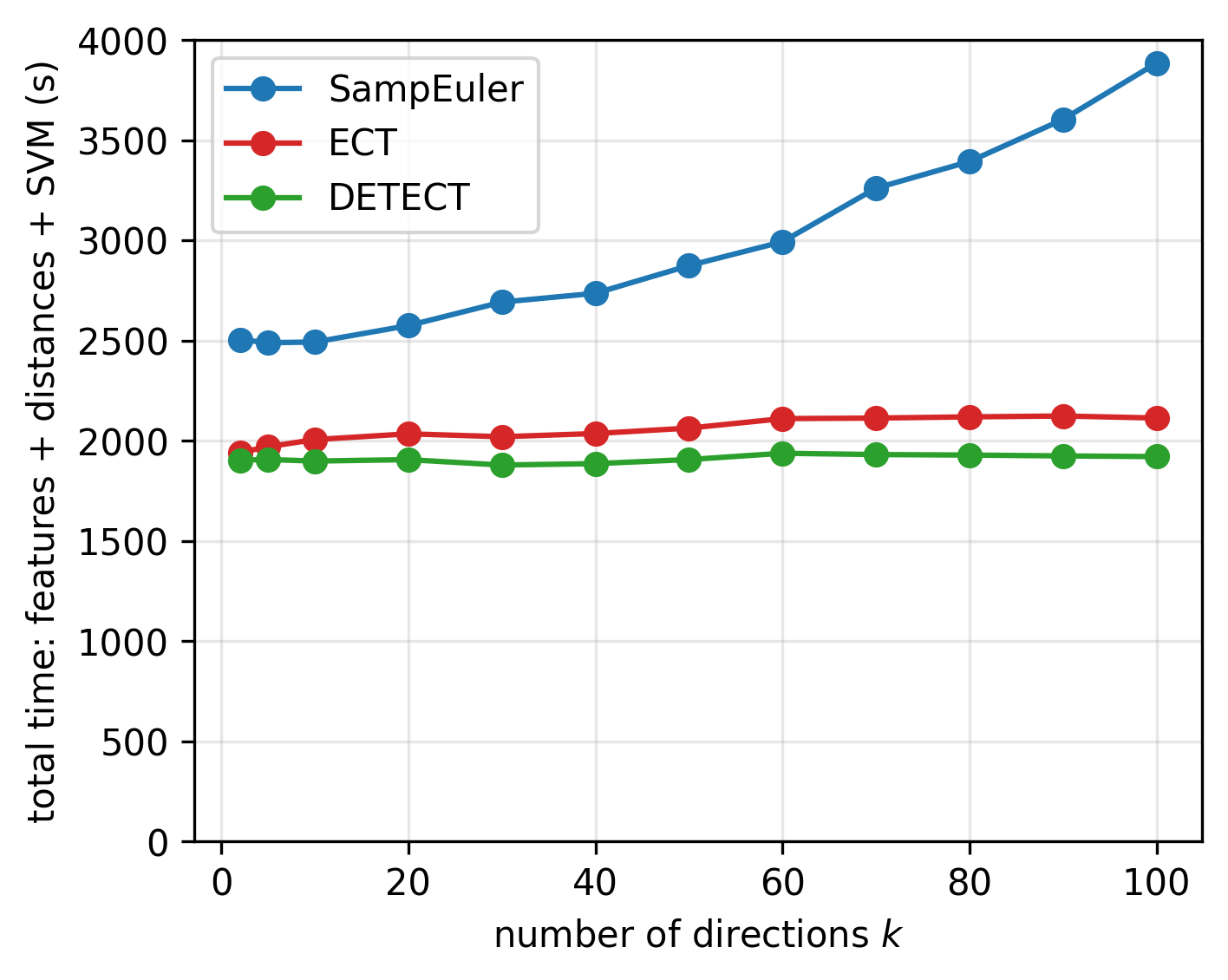}
    \caption{}\label{fig:r23-total-k}
  \end{subfigure}\hfill
  \begin{subfigure}{0.45\textwidth}\centering
    \includegraphics[width=\linewidth]{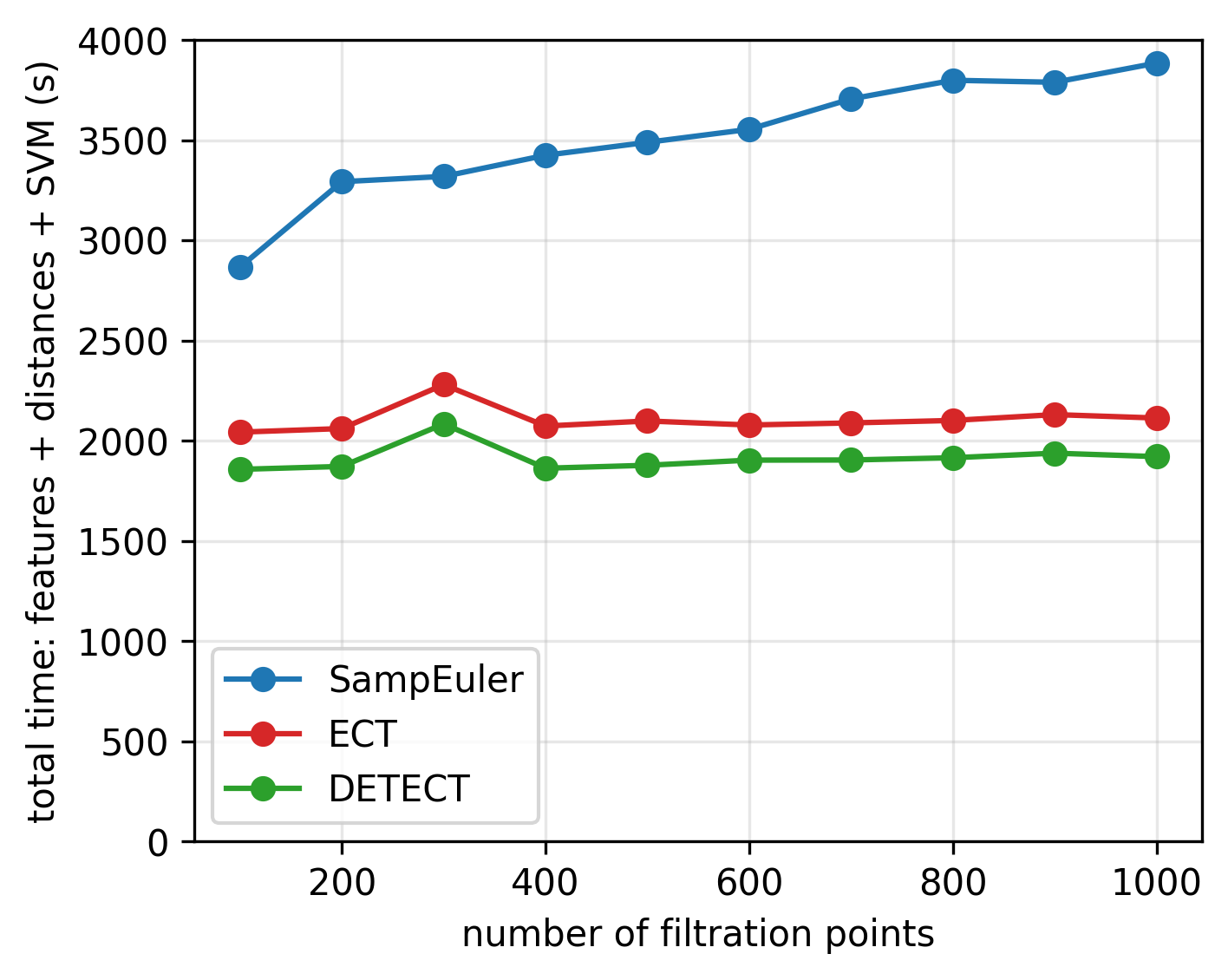}
    \caption{}\label{fig:r23-total-x}
  \end{subfigure}
  \caption{Low-range sampling ablation on the full 70-class MPEG-7 dataset, comparing SampEuler, ECT, and DETECT. Panel (\textit{a}): classification accuracy versus the number of directions $k$ (with $x = 1000$ fixed). Error bars show the standard deviation over $10$ train/test splits. Panel (\textit{b}): classification accuracy versus the number of filtration points $x$ (with $k = 100$ fixed). Error bars show the standard deviation over $10$ train/test splits. Panel (\textit{c}): the feature-generation time over the full dataset versus the number of directions $k$ (with $x = 1000$ fixed). Panel (\textit{d}): the feature-generation time versus the number of filtration points $x$ (with $k = 100$ fixed). Panel (\textit{e}): the total computation time (feature generation, distance computation, and SVM training and testing) versus the number of directions $k$ (with $x = 1000$ fixed). Panel (\textit{f}): the same total computation time versus the number of filtration points $x$ (with $k = 100$ fixed).}
  \label{fig:r23-ablation}
\end{figure}

\subsection{Quantifying Thymus Morphological Difference}
We report mean classification accuracy across 50 trials and the wall-clock runtime of the ECT-based pipelines. For ECT-based methods, we compute 360 directions with 300 points along each direction. SampEulers are computed using the same parameters but with random set of directions. As a baseline, we benchmark against persistence-diagram (PD) features~\cite{adams2017persistence} computed from signed-distance-function (SDF) filtrations~\cite{scikit-fmm}. Concretely, for each binary image we compute its SDF, construct the sublevel-set filtration, and extract PDs in homology dimensions~0 and~1, thereby capturing births and deaths of connected components and loops while encoding geometry via distance to the boundary.
We determine global birth and persistence ranges to set the bounding boxes, discard points with infinite death, and discretize Persistence Images (PIs) with a pixel size of \(\tfrac{\text{global persistence range}}{500}\).
PIs use a Gaussian kernel with \(\sigma = 1.0\) and uniform weights; the images are then flattened into feature vectors. For both the ECT-based features and PIs, we train a random-forest classifier: we select hyperparameters via a fixed grid, refit with the best setting, and evaluate across trials.

We also compare against state-of-the-art deep learning models for image classification~\cite{krizhevsky2012imagenet,dosovitskiy2020image,sandler2018mobilenetv2,tan2019efficientnet}.
For AlexNet, ViT, MobileNetV2, and EfficientNet-B0, we initialize from ImageNet-pretrained weights and replace the final layer with a two-way linear head.
Inputs are resized to $224 \times 224$ RGB and normalized with ImageNet statistics. In each trial, we train using stochastic gradient descent (learning rate \(10^{-3}\), momentum \(0.9\)), batch size \(16\), and cross-entropy loss for \(5\) epochs.

\begin{longtable}{L{4.2cm} L{2.3cm} L{2.3cm} L{1.8cm} L{1.8cm}}
\caption{The average accuracy $(\%)$ and run time of different methods on thymus K8 quadrants and K14 quadrants.}
\label{table: accuracy_runtime} \\

\toprule
\textbf{Method} & \textbf{K8 Acc.} & \textbf{K14 Acc.} & \textbf{K8 Time} & \textbf{K14 Time} \\
\midrule
\endfirsthead

\multicolumn{5}{c}{{\bfseries \tablename\ \thetable{} -- continued from previous page}} \\
\toprule
\textbf{Method} & \textbf{K8 Acc.} & \textbf{K14 Acc.} & \textbf{K8 Time} & \textbf{K14 Time} \\
\midrule
\endhead

\midrule
\multicolumn{5}{r}{{Continued on next page}} \\
\endfoot

\bottomrule
\endlastfoot

PH Image \cite{adams2017persistence} & 0.659$\pm$0.119 & 0.698$\pm$0.120 & 2m 46.5s & 3m 14.2s \\
\hline
AlexNet \cite{krizhevsky2012imagenet} & 0.697$\pm$0.135 & 0.690$\pm$0.112 & 4m 18.6s & 4m 17.7s \\
\hline
Vision Transformer \cite{dosovitskiy2020image} & 0.680$\pm$0.143 & 0.641$\pm$0.118 & 32m 0.2s & 31m 57.5s \\
\hline
MobileNet V2 \cite{sandler2018mobilenetv2} & 0.636$\pm$0.119 & 0.687$\pm$0.113 & 24m 54.5s & 24m 41.4s \\
\hline
EfficientNetB0 \cite{tan2019efficientnet} & 0.687$\pm$0.150 & 0.681$\pm$0.124 & 34m 23.8s & 34m 15.1s \\
\hline
ECT & 0.685$\pm$0.100 & 0.684$\pm$0.089 & 58.2s & 1m 9.7s \\
\hline
DETECT & 0.700$\pm$0.129 & 0.684$\pm$0.106 & 31.3s & 29.8s \\
\hline
SampEuler & \textbf{0.713$\pm$0.108} & 0.723$\pm$0.106 & 46.5s & 42.4s \\
\hline
vectorized SampEuler & 0.689$\pm$0.116 & \textbf{0.727$\pm$0.102} & 33.6s & 28.6s \\

\end{longtable}

We use established methods to confirm our interpretations of the thymus structures. We compute the zeroth and first Betti numbers of images for the number of connected components and loops. The results are plotted as violin plots in \Cref{fig:combined_violin}. For K8, the violin plots for the zeroth Betti number show that young quadrants have a slightly higher number of connected components.  Similarly, Betti-1 plots show that old quadrants have a higher rank for $H_1$, meaning that old quadrants contain more loops. For K14, the violin plots for Betti-0 and Betti-1 both show that the old quadrants have a much larger range of variation compared with the young quadrants. These observations are consistent with our interpretations of the SHAP scores of vectorized SampEuler features. To illustrate the connection of the SHAP scores to the underlying geometry, we show two old K8 quadrants at the extremes of the SHAP-based score in \Cref{fig:k8_shap_examples}. The high-scoring quadrant is a well-connected network that forms many loops, following the decreasing trend emphasised in the SHAP map. The low-scoring quadrant is sparse and disconnected.

To assess the density of holes in the cortex component, we apply the signed distance function filtration to each image so that the centres of the holes have negative distances. Now, the persistence of zeroth-degree features in persistent homology represents how dense the holes are. The further away the holes are, the longer these zeroth-degree features will persist. As shown in \Cref{fig:combined_pers_image}, the young quadrants have more persistent characteristics than the old quadrants. These observations are consistent with our interpretation from the gradient of vectorized SampEuler computation.

\begin{figure}[htbp]
  \centering
  \begin{subfigure}[b]{0.60\linewidth}
    \centering
    \includegraphics[width=\linewidth]{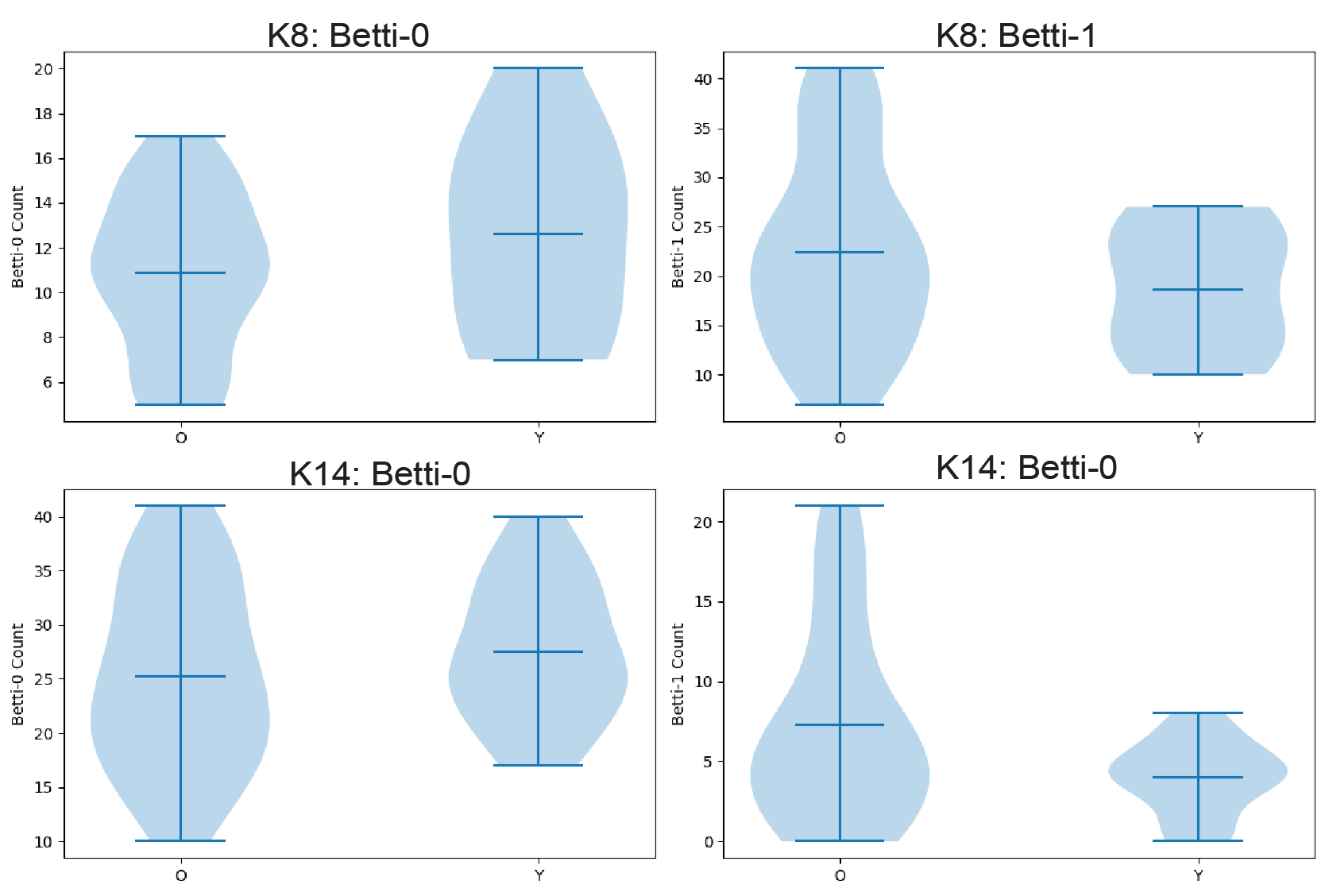}
    \caption{Violin plots of Betti numbers 0 \& 1.}
    \label{fig:combined_violin}
  \end{subfigure}%
  \hfill
  \begin{subfigure}[b]{0.35\linewidth}
    \centering
    \includegraphics[width=\linewidth]{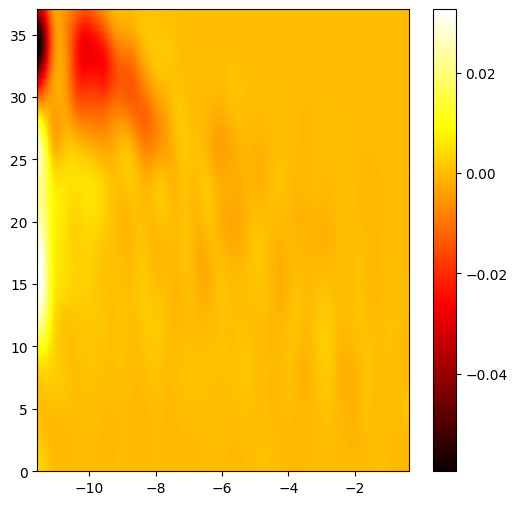}
    \caption{Difference of persistence images: old -- young.}
    \label{fig:combined_pers_image}
  \end{subfigure}
  \caption{%
    (a) Violin plots of Betti number 0 and Betti number 1 from top 5\% to 95\% of samples. 
    (b) Difference of persistence images: \(\mathrm{P}_{\mathrm{old}} - \mathrm{P}_{\mathrm{young}}\).
  }
  \label{fig:combined_persistence}
\end{figure}

\begin{figure}[htbp]
  \centering
  \includegraphics[width=0.6\linewidth]{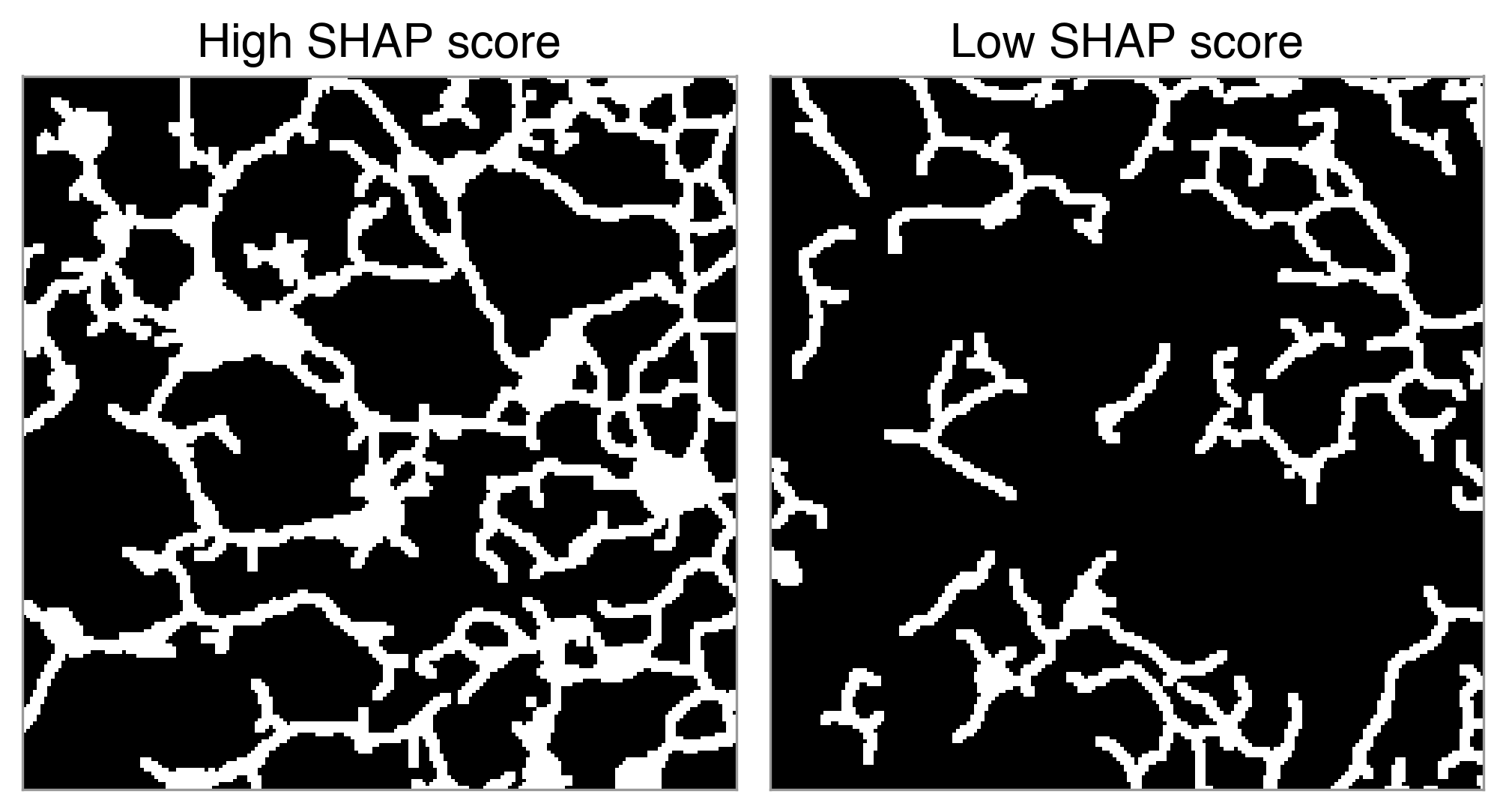}
  \caption{Two old K8 quadrants at the extremes of a SHAP-based score. The high-scoring quadrant (left) is a well-connected network that forms many loops, following the decreasing trend emphasised in the SHAP map; the low-scoring quadrant (right) is sparse and disconnected.}
  \label{fig:k8_shap_examples}
\end{figure}

\subsection{Depth Analysis}
We apply depth analysis of the cortical region to preserve the spatial heterogeneity. Since the thymus and medulla boundaries are not homotopic, we choose the following definition as the notion of depth.

\begin{definition}
    Given a pixel $p$ in the cortical region of the thymus, we define: 
        $$d_p : = \min\{d(p, x), \text{ for }x \text{ in medulla}\}.$$
     where $d$ is the usual Euclidean distance. For a quadrant $Q$ in thymus $T$, we define: 
        $$d_Q := \frac{1}{|Q|}\sum_{p\in Q} d_p.$$
    We define the normalized depth of $Q$ as follows:
        $$D(Q) := 1 - \frac{d_Q}{\max_{Q'\in T} d_{Q'}}.$$
\end{definition}

\begin{remark}
    Since the medulla region consists of more than one connected component in all thymic samples, there is no homotopy between the thymic capsule and the cortex-medulla junction (CMJ). The above notion of depth is used when the CMJ is of interest to the study. 
    
    One could also define the distance $d_p$ as the minimal distance to any thymic boundary point and obtain another notion of normalized depth with 0 at the capsule and 1 at the centre of the thymus. This notion of depth is more suitable when the whole thymus is of interest. We use it later for the DN cell analysis.
\end{remark}

Based on the normalized depth of a quadrant, we construct the following depth kernel for each quadrant: 
    $$w_Q(t) = \exp\left(-\frac{1}{2}\left(\frac{D(Q)-t}{h}\right)^2\right).$$
We use $h=0.1$ and normalize the weight by each age group $A$ and get the normalized kernel $\tilde{w}_Q(t) = \frac{w_Q(t)}{\sum_{Q'\in A}w_{Q'}(t)}$. To check that the depth notion is consistent for both age groups, we plot the distribution of quadrants at various depth intervals in \Cref{fig: distance_distribution}. The result suggests that the shrinkage in size of the thymus due to ageing does not affect our notion of depth, and so it is suitable to compare quadrants at the same depth from different age groups to preserve the spatial heterogeneity of the cortical region. 

\begin{figure}[!htbp]
    \centering
    \includegraphics[width=0.4\linewidth]{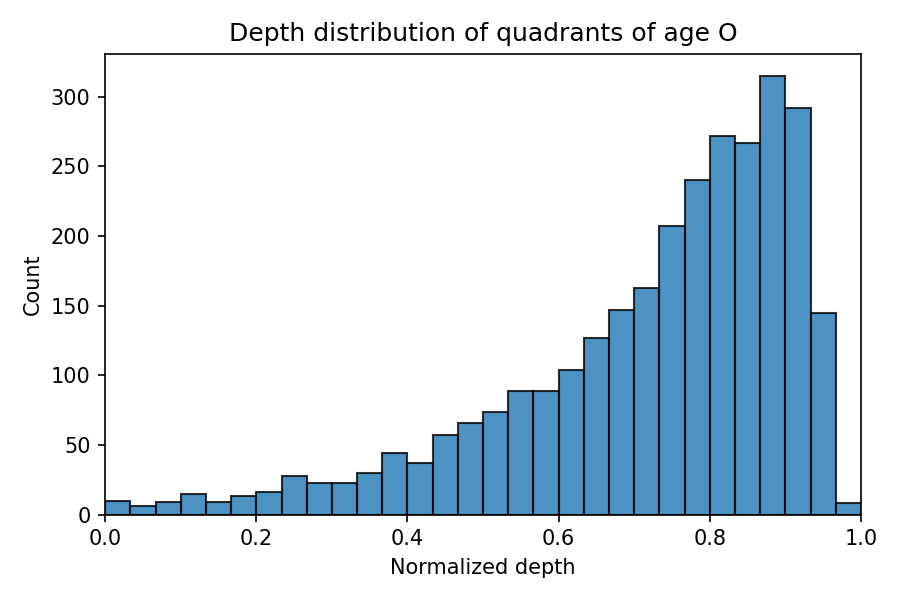}
    \includegraphics[width=0.4\linewidth]{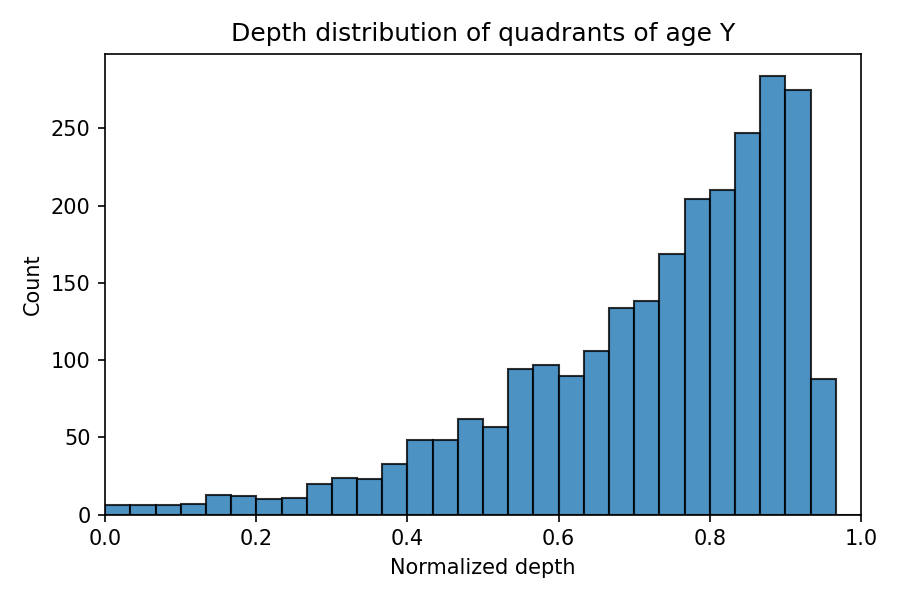}
    \caption{Depth distributions of quadrants of each age group. The x-axis is the normalized depth $D(Q)$ of each quadrant $Q$, and the y-axis is the count of quadrants having depth in a given depth interval.}
    \label{fig: distance_distribution}
\end{figure}

For depth $t$, the normalized kernels define probability measures on quadrants for each age group,
\begin{align}
P_Y^t &= \sum_{Q \in Y} \tilde w_Q^Y(t)\,\delta_Q, \\
P_O^t &= \sum_{Q \in O} \tilde w_Q^O(t)\,\delta_Q,
\end{align}
where $\delta_Q$ is the Dirac measure of the corresponding feature of $Q$. We can then compute the energy distance \cite{szekely2013energy} $E(P_Y^t, P_O^t)$ between the two measures using the pairwise distance matrix $M$ of quadrant features by 
\begin{equation}
\begin{aligned}
E(P_Y^t, P_O^t) &= 2\left( \sum_{i\in Y}\sum_{j\in O} \tilde{w}_i^Y(t)\tilde{w}_j^O(t)M_{ij}\right) \\
&\quad - \sum_{i,j\in O} \tilde{w}_i^O(t)\tilde{w}_j^O(t)M_{ij}
    - \sum_{i,j\in Y} \tilde{w}_i^Y(t)\tilde{w}_j^Y(t)M_{ij}.
\end{aligned}
\end{equation}

\subsection{Cell Distribution vs Morphology}
For the same set of mouse thymi, we obtained cell‐type and spatial locations for all cells within each thymus. Our goal is to characterize and compare age‐related changes within the thymus and the epithelial architecture of the thymus. To this end, we perform k-medoids clustering \cite{kaufman2009finding} separately on cell-location features and on shape descriptors of thymic quadrants, and then compare the resulting cluster partitions.

To allow faster computation of the pairwise Wasserstein distance matrix while keeping sufficient information, we sample $100$ directions with $3000$ points along each direction, and we use the sliced Wasserstein distance algorithm \cite{bonneel2015sliced} with $50$ slices. 

Since we focus on the cortical compartment, and for simplicity, we restrict the compositional analysis of haematopoietic cells to \textbf{DN3, Presel DP, Postsel DP (CD69+)} and \textbf{Postsel DP (CD69-)} thymocytes. These cell types not only constitute the vast majority of cortical thymocytes but are also centered around the developmentally-relevant process of positive selection.

We discuss how to choose the number of clusters for both clusterings. We first applied silhouette analysis \cite{kaufman2009finding} to find the most suitable number of clusters from the set $\{2,3,\dots,10\}$ for both k-medoids clusterings. The analysis suggests $k=2$ for both the clustering based on enrichment ratios and the clustering based on SampEulers. The results are shown in \Cref{figure: 2-cluster} and \Cref{fig:2-cluster(2)}. Two clusters are not able to provide meaningful biological interpretations about the impact of the ageing process. Therefore, we manually force three clusters for both clusterings as the second-highest silhouette score choice. To further constrain the analysis to look at cells in direct contact with thymic epithelial cells (TECs), we filter to only count cells of selected cell types that are at a distance less than or equal to $5\text{ }\mu \text{m}$ from a TEC. The results are shown in \Cref{figure: 3vs3-cluster} and \Cref{fig:3vs3-cluster(2)}.
\begin{figure}
  \centering
  \includegraphics[width=\textwidth, 
                   height=0.95\textheight, 
                   keepaspectratio]{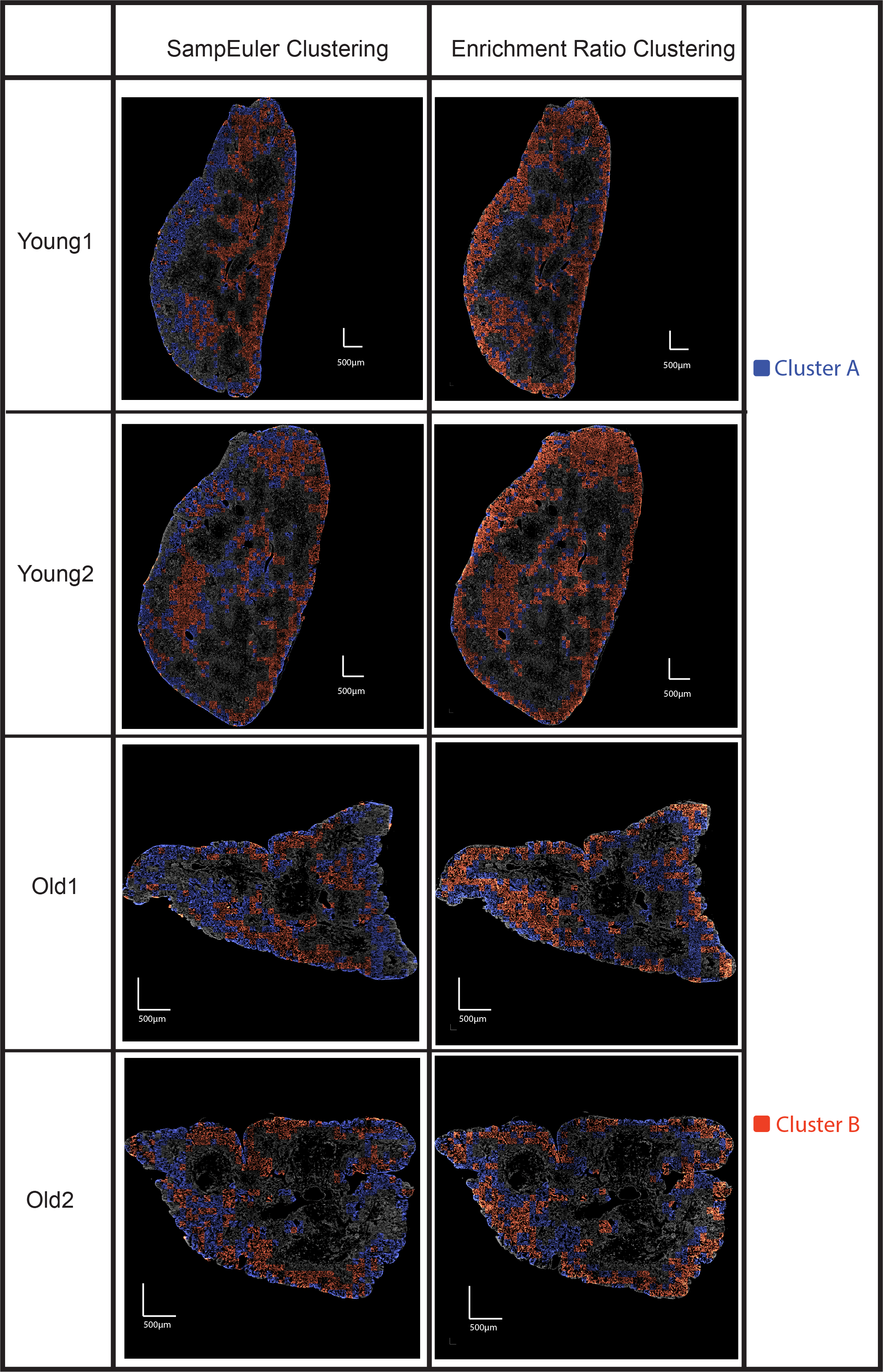}
  \caption{ K-medoids clustering results with 2 clusters for all thymic cortex quadrants. This plot contains result for the first four thymi, the remaining results are included in \Cref{fig:2-cluster(2)}. The first column on the left is the results using SampEulers, the second column is the results using cell enrichment ratios.}
  \label{figure: 2-cluster}
\end{figure}
\begin{figure}
    \centering
    \includegraphics[width=\textwidth, 
                   height=0.95\textheight, 
                   keepaspectratio]{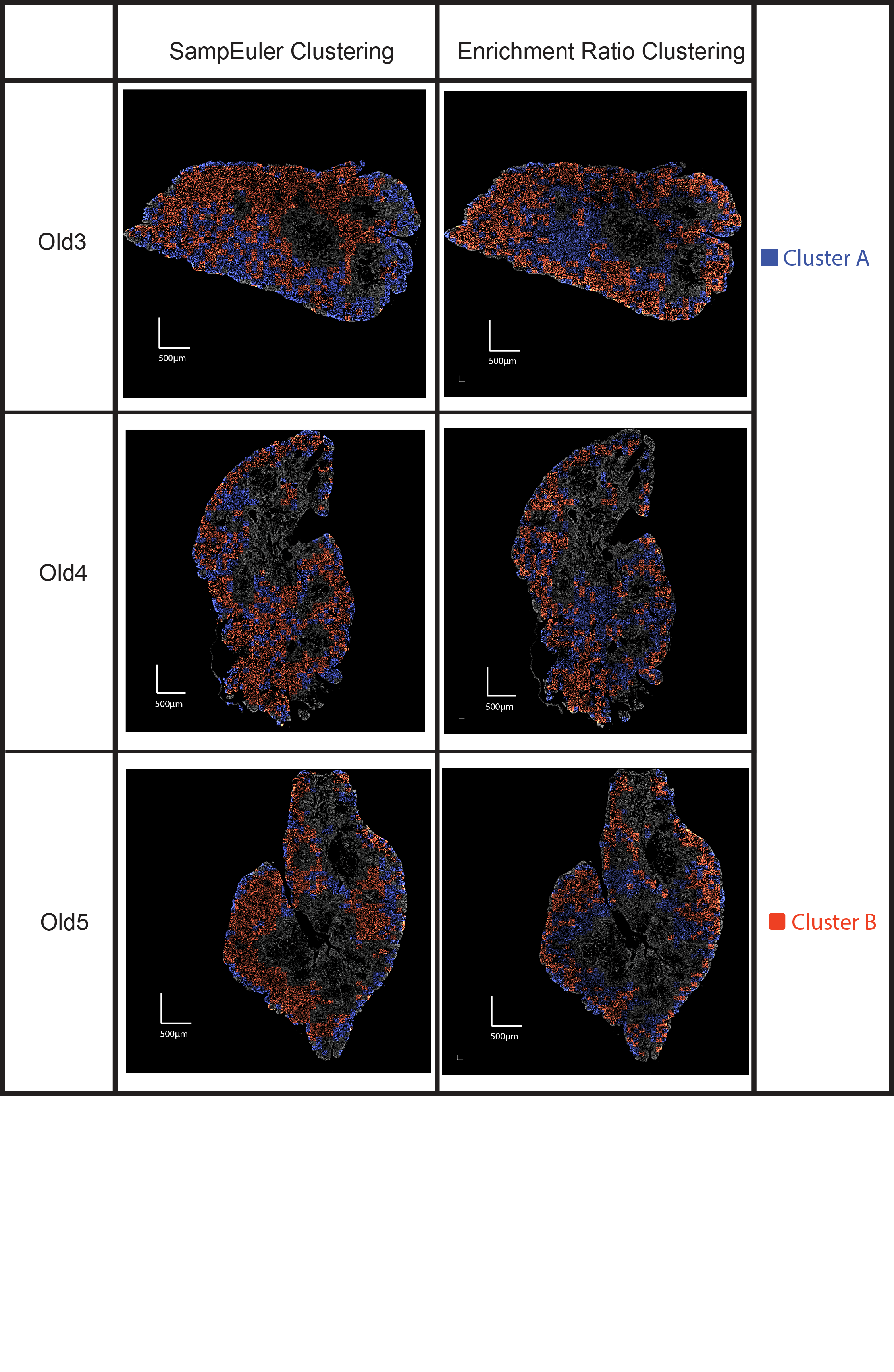}
    \caption{K-medoids clustering results with 2 clusters for all thymic cortex quadrants. This plot contains result for the remaining thymi continuing from \Cref{figure: 2-cluster}. The first column on the left is the results using SampEulers, the second column is the results using cell enrichment ratios.}
    \label{fig:2-cluster(2)}
\end{figure}
\begin{figure}
  \centering
  \includegraphics[width=\textwidth, 
                   height=0.95\textheight, 
                   keepaspectratio]{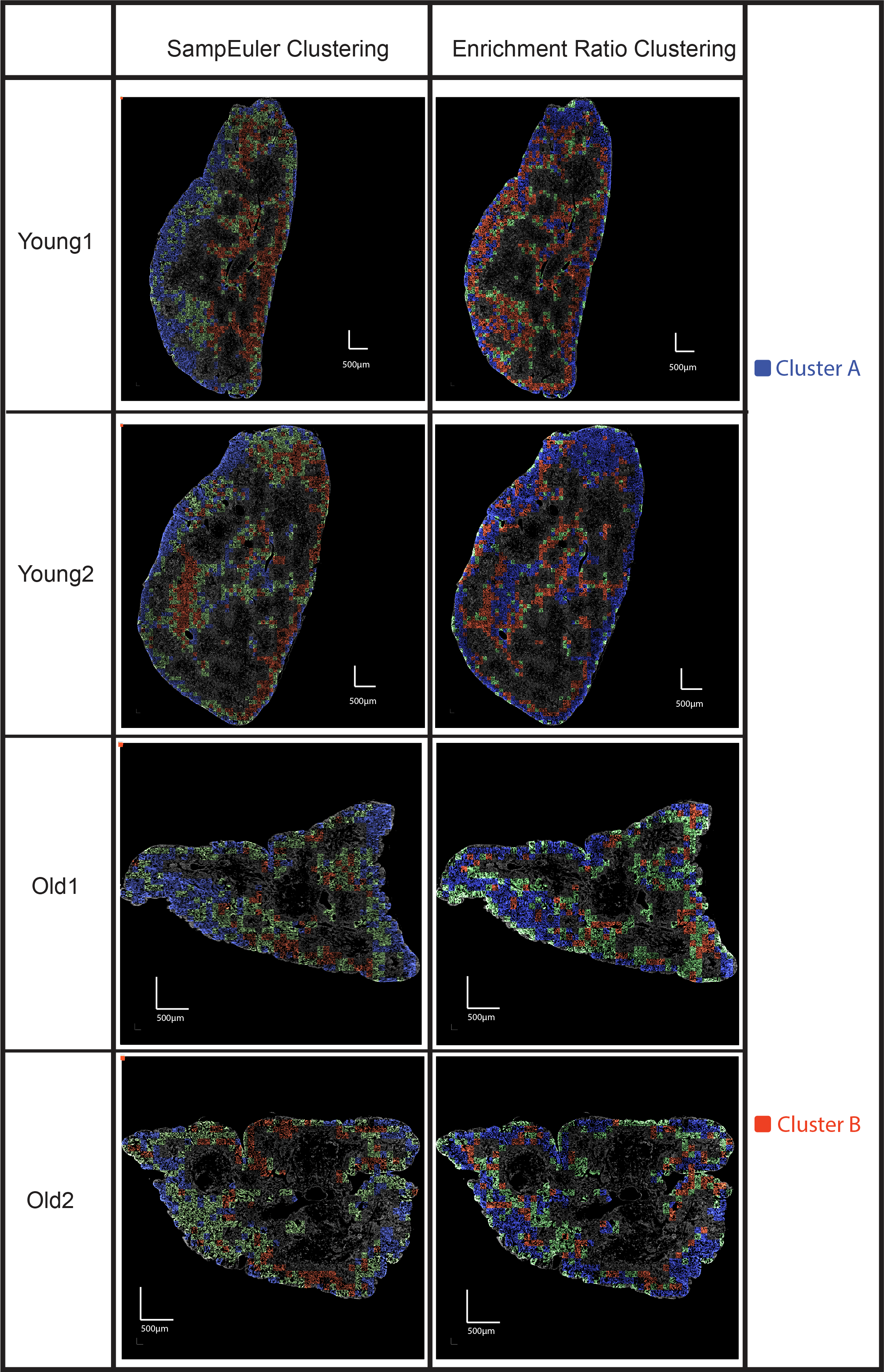}
  \caption{ K-medoids clustering results for all thymic cortex quadrants using 3 clusters. This plot contains results for the first four thymi, the remaining results are included in \Cref{fig:3vs3-cluster(2)}. The first column on the left is the results using SampEulers, the second column is the results using selected cell enrichment ratios.}
  \label{figure: 3vs3-cluster}
\end{figure}

\begin{figure}
    \centering
    \includegraphics[width=\textwidth, 
                   height=0.95\textheight, 
                   keepaspectratio]{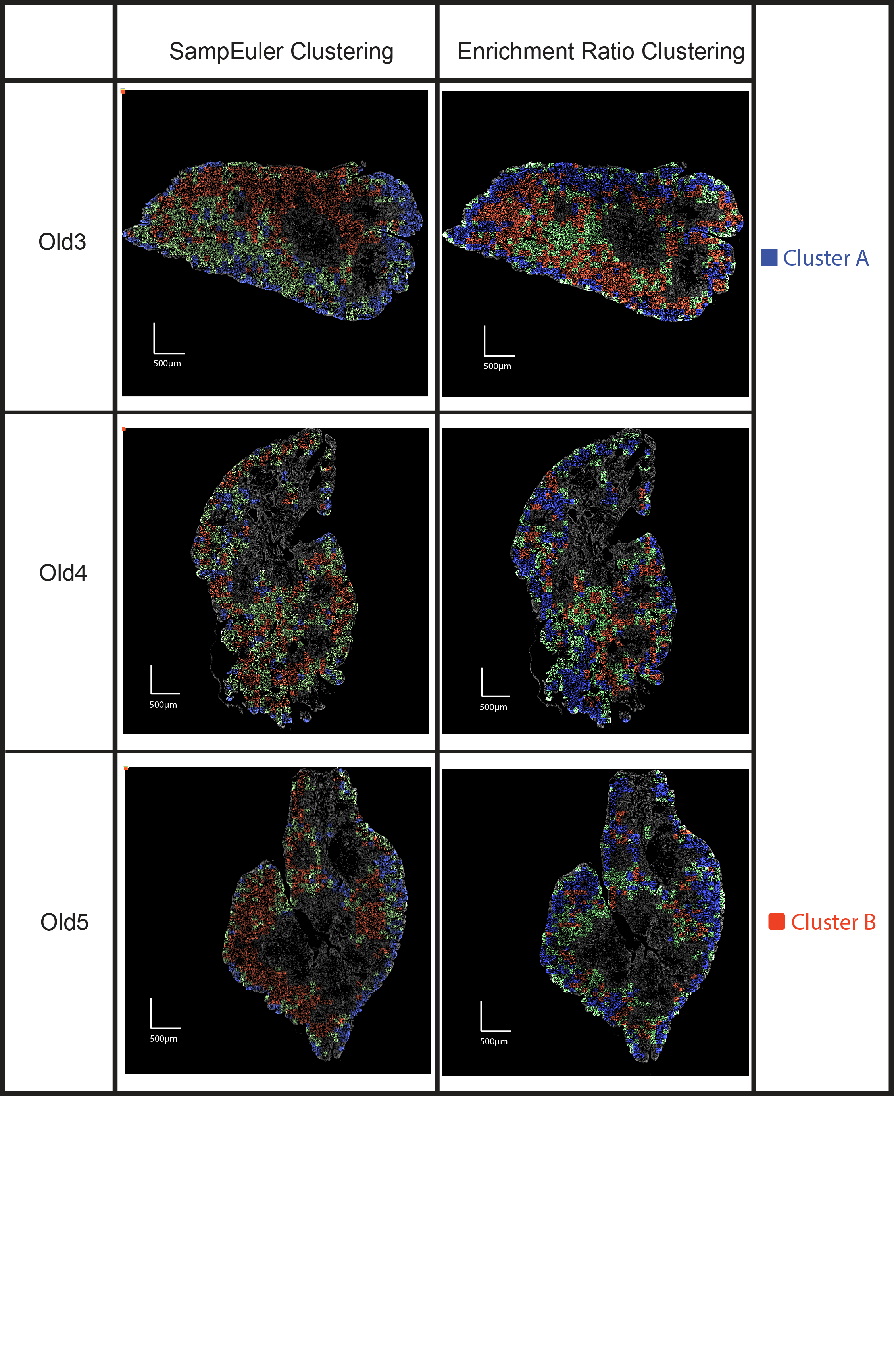}
    \caption{K-medoids clustering results for all thymic cortex quadrants using 3 clusters. This plot contains results for the remaining thymi continuing from \Cref{figure: 3vs3-cluster}. The first column on the left is the results using SampEulers, the second column is the results using cell enrichment ratios.}
    \label{fig:3vs3-cluster(2)}
\end{figure}

\begin{figure}
    \centering
    \includegraphics[width=\textwidth, 
                     height=0.95\textheight, 
                     keepaspectratio]{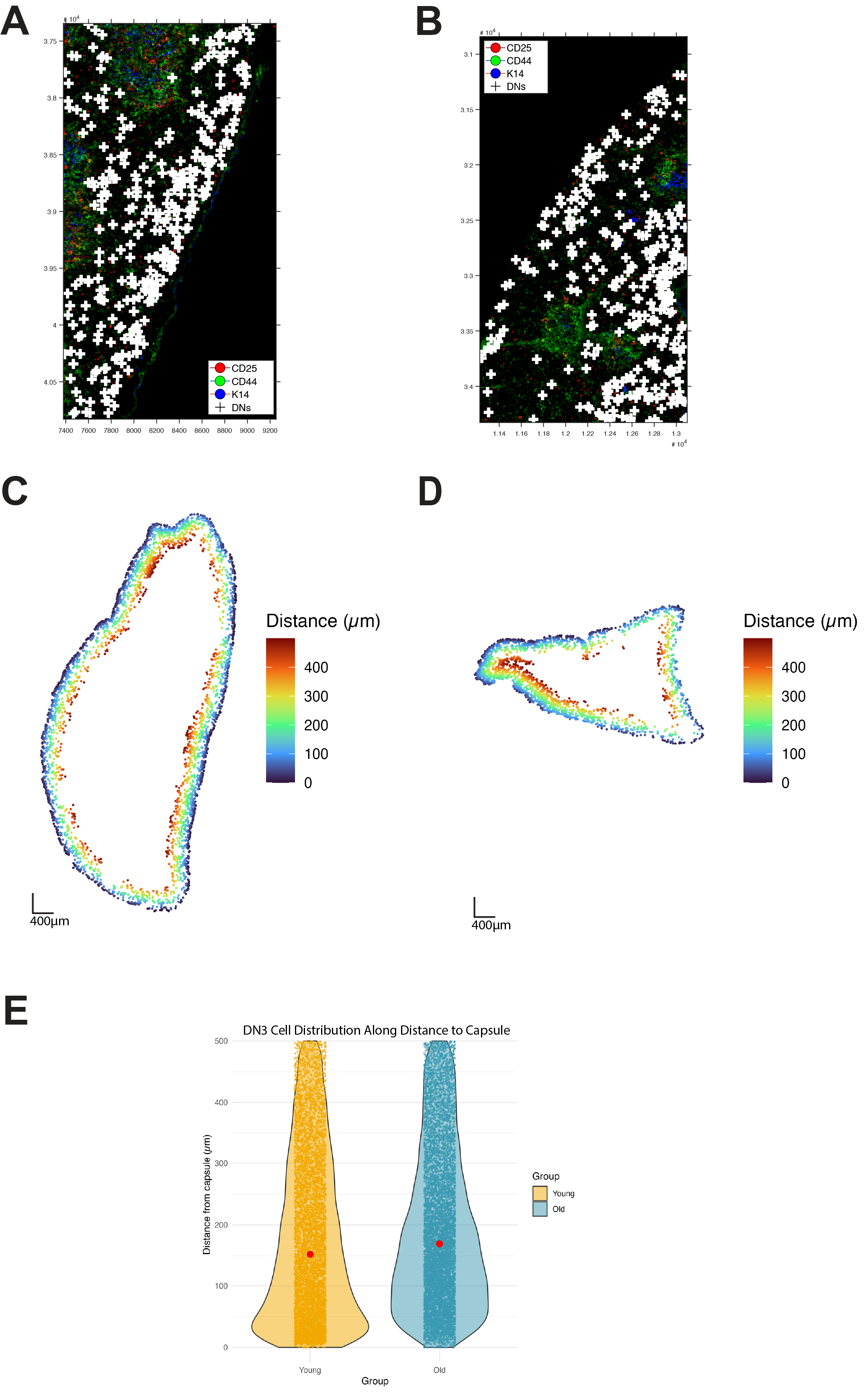}
    \caption{Qualitative assessment of DN3 localization characterized by CD4-, CD8a-, CD5-, CD44-, CD25+, CD45+ for a young \textit{A} and an old thymus tissue \textit{B}. A close-up of the sub-capsular/deep-cortical area is represented for each age. One x,y-unit corresponds to 0.5 $\mu$m. \textit{C} and \textit{D} Spatial scatter plots illustrating the spatial location of DN3s within a set distance of 500 $\mu$m from the capsule for a young and an old thymus, respectively. \textit{E} Violin plot of individual DN3 distances from the capsule split by age group. The quantification is based on the whole dataset (2 young and 5 old samples). These distributions were compared using a two-sample t-test (p = 6.3e-11) and a Wilcox test (p = 3.1e-18).}
    \label{figure: DN_plots}
\end{figure}

\begin{longtable}{L{1.2cm} L{1.5cm} L{1.8cm} L{2.0cm} L{2.3cm} L{1.6cm} L{1.2cm} L{1.5cm}}
\caption{PhenoCycler antibody panel with host species, clone, vendor, catalogue number, fluorophore, dilution, and exposure time. DAPI was imaged at every round. For the other markers, the order in which they appear in the table represents the order in which they were imaged in groups of three (see Materials and Methods).
Ar Hms = Armenian hamster, Ms = mouse, Rb = rabbit, Rt = rat, Sy Hms = Syrian hamster}
\label{CODEXantibodies}\\
\toprule
Marker & Host species & Clone & Vendor & Catalogue number & Fluorophore & Dilution & Exposure time [ms] \\
\midrule
\endfirsthead
\multicolumn{6}{c}{{\bfseries \tablename\ \thetable{} -- continued from previous page}} \\
\toprule
Marker & Host species & Clone & Vendor & Catalogue number & Fluorophore & Dilution & Exposure time [ms] \\ 
\midrule
\endhead
\midrule
\multicolumn{6}{r}{{Continued on next page}} \\
\endfoot

\bottomrule
\endlastfoot

DAPI  & -       & -         & -            & -            & -        & 1:1000 & 150 \\
\hline
AIRE  & Rt      & 5H12      & Invitrogen   & 14-5934-82   & Cy5      & 1:200 & 200 \\
\hline
B220  & Rt      & Ra3-6B2   & Akoya        & 4150006      & AF488    & 1:200 & 200 \\
\hline
CD8A  & Rt      & 53-6.7    & Akoya        & 4250017      & ATTO550  & 1:200 & 150 \\
\hline
FOXP3 & Rt      & FJK-16s   & Invitrogen   & 14-5773-82   & Cy5      & 1:200 & 200 \\
\hline
TCRGD & Ar Hms  & GL3       & BioLegend    & 118128       & AF488    & 1:100 & 300 \\
\hline
NP63  & Rb      & polyclonal& Proteintech  & 12143-1-AP   & ATTO550  & 1:200 & 300 \\
\hline
PD1   & Rt      & 29F.1A12  & Invitrogen   & P378-1MG     & Cy5      & 1:200 & 300 \\
\hline
CD86  & Rt      & GL-1      & Invitrogen   & MA1-10299    & AF488    & 1:200 & 300 \\
\hline
CD69  & Ar Hms  & H1.2F3    & Invitrogen   & 14-0691-82   & ATTO550  & 1:200 & 300 \\
\hline
CD11C & Ar Hms  & N418      & Akoya        & 4350013      & Cy5      & 1:200 & 300 \\
\hline
CD45  & Rt      & 30-F11    & Akoya        & 4150002      & AF488    & 1:200 & 150 \\
\hline
ICOS  & Ar Hms  & C398.4A   & Invitrogen   & 14-9949-82   & ATTO550  & 1:200 & 200 \\
\hline
XCR1  & Ms      & ZET       & BioLegend    & 148202       & Cy5      & 1:100 & 300 \\
\hline
PDCA1 & Rt      & eBio927   & Invitrogen   & 16-3172-81   & AF488    & 1:200 & 300 \\
\hline
SIRPA & Rt      & P84       & Invitrogen   & 16-1721-81   & ATTO550  & 1:100 & 300 \\
\hline
LY6G  & Rt      & 1A8       & Akoya        & 4350015      & Cy5      & 1:200 & 100 \\
\hline
CD71  & Rt      & R17217    & Invitrogen   & 17-0711-82   & AF488    & 1:200 & 300 \\
\hline
KI67  & Ms      & B56       & Akoya        & 4250019      & ATTO550  & 1:200 & 150 \\
\hline
SCA1  & Rt      & D7        & Invitrogen   & 14-5981-85   & Cy5      & 1:200 & 200 \\
\hline
CD25  & Rt      & PC61.5    & Invitrogen   & 14-0251-86   & AF488    & 1:200 & 300 \\
\hline
HELIOS& Ar Hms  & 22F6      & BioLegend    & 137202       & ATTO550  & 1:200 & 200 \\
\hline
K10   & Rb      & polyclonal& Invitrogen   & PA5-104456   & Cy5      & 1:400 & 80 \\
\hline
CD62L & Rt      & 95218     & R\&D         & MAB5761      & AF488    & 1:200 & 200 \\
\hline
CD5   & Rt      & 53-7.3    & Akoya        & 4250007      & ATTO550  & 1:200 & 250 \\
\hline
ERTR7 & Rt      & ER-TR7    & Bio-Rad      & MCA2402      & Cy5      & 1:400 & 80 \\
\hline
CD24  & Rt      & M1/69     & Akoya        & 4150014      & AF488    & 1:200 & 200 \\
\hline
CD31  & Rt      & MEC13.3   & Akoya        & 4250001      & ATTO550  & 1:200 & 200 \\
\hline
CD49F & Rt      & GoH3      & Akoya        & 4350007      & Cy5      & 1:200 & 200 \\
\hline
CD11B & Rt      & M1/70     & Akoya        & 4150015      & AF488    & 1:200 & 200 \\
\hline
CD4   & Rt      & RM4-5     & Akoya        & 4250016      & ATTO550  & 1:200 & 150 \\
\hline
LY51  & Rt      & 6C3       & BioLegend    & 108302       & Cy5      & 1:100 & 300 \\
\hline
PIGR  & Rb      & polyclonal& Invitrogen   & PA5-35340    & AF488    & 1:200 & 250 \\
\hline
MHCII & Rt      & M5/114.15.2 & Akoya      & 4250003      & ATTO550  & 1:400 & 150 \\
\hline
TCRB  & Ar Hms  & H57-597   & Akoya        & 4350006      & Cy5      & 1:200 & 150 \\
\hline
PDPN  & Sy Hms  & 8.1.1     & Invitrogen   & MA5-18054    & AF488    & 1:200 & 300 \\
\hline
CD206 & Rt      & C068C2    & BioLegend    & 141701       & ATTO550  & 1:400 & 250 \\
\hline
K8    & Rt      & TROMA-1   & EMD Millipore& MABT329      & Cy5      & 1:200 & 80 \\
\hline
K14   & Rb      & polyclonal& Invitrogen   & PA5-16722    & AF488    & 1:200 & 80 \\
\hline
CD44  & Rt      & IM7       & Akoya        & 4250002      & ATTO550  & 1:200 & 200 \\
\bottomrule
\end{longtable}

\begin{longtable}{L{0.7cm} L{2cm} L{4.5cm} L{5.8cm} L{1.6cm} L{1.4cm}}
\caption{List of PhenoCycler cell phenotypes with their defining marker thresholds, distance thresholds, and regional filters.}
\label{CODEXphenotypes} \\
\toprule
No. & Phenotype name & Corresponding cell type & Defining markers and their normalized expression thresholds & Distance thresholds [pixels] & Region filter \\
\midrule
\endfirsthead
\multicolumn{6}{c}{{\bfseries \tablename\ \thetable{} -- continued from previous page}} \\
\toprule
No. & Phenotype & Cell type & Markers and Expression Thresholds & Dist.\ [px] & Region \\
\midrule
\endhead
\midrule
\multicolumn{6}{r}{{Continued on next page}} \\
\endfoot
\bottomrule
\endlastfoot
1 & ICOS+ & T-cells expressing ICOS & ICOS \textgreater\ 0.25 & 7 & - \\
\hline
2 & T-reg & Regulatory T-cells & FOXP3 \textgreater\ 0.15 & 6 & - \\
\hline
3 & CD11b+ & Monocytes & CD11b \textgreater\ 0.15 & 8 & - \\
\hline
4 & T-gd (PD1+) & gd T-cells expressing PD1 & (TCRgd \textgreater\ 0.15) \& (PD1 \textgreater\ 0.25) & 6 & - \\
\hline
5 & T-gd (PD1-) & gd T-cells not expressing PD1 & (TCRgd \textgreater\ 0.15) \& (PD1 $\le$ 0.25) & 6 & - \\
\hline
6 & DN3 & DN T-cells & (CD44 $\le$ 0.2) \& (CD25 \textgreater\ 0.2) \& (CD5 $\le$ 0.2) \& (TCRb $\le$ 0.2) & 6 & Cortex \\
\hline
7 & ISP8 & Immature SP8 T-cells & (CD25 $\le$ 0.2) \& (CD8a \textgreater\ 0.3) \& (CD4 $\le$ 0.2) \& (CD5 $\le$ 0.2) \& (TCRb $\le$ 0.2) & - & Cortex \\
\hline
8 & SP8 (CD69+) & SP8 T-cells expressing CD69 & (CD8a \textgreater\ 0.3) \& (CD4 $\le$ 0.2) \& ((TCRb \textgreater\ 0.15) $\mid$ (CD5 \textgreater\ 0.15)) \& (CD69 \textgreater\ 0.25) & - & Medulla \\
\hline
9 & SP8 (CD69-) & SP8 T-cells not expressing CD69 & (CD8a \textgreater\ 0.3) \& (CD4 $\le$ 0.2) \& ((TCRb \textgreater\ 0.15) $\mid$ (CD5 \textgreater\ 0.15)) \& (CD69 $\le$ 0.25) & - & Medulla \\
\hline
10 & MP & Macrophages & CD206 \textgreater\ 0.075 & 8 & - \\
\hline
11 & BC & B cells & B220 \textgreater\ 0.1 & 4 for BC clusters, 6 for interspersed BC & - \\
\hline
12 & SP4 (Helios+) & SP4 T-cells expressing Helios & (CD8a $\le$ 0.2) \& (CD4 \textgreater\ 0.2) \& ((TCRb \textgreater\ 0.15) $\mid$ (CD5 \textgreater\ 0.15)) \& (Helios \textgreater\ 0.15) & - & Medulla \\
\hline
13 & pDC (PDCA1+) & Plasmacytoid dendritic cells expressing PDCA1 & (CD11c \textgreater\ 0.25) \& (PDCA1 \textgreater\ 0.25) & 11 & - \\
\hline
14 & cDC1 (XCR1+) & Conventional dendritic cells expressing XCR1 & (CD11c \textgreater\ 0.25) \& (XCR1 \textgreater\ 0.35) & 11 & - \\
\hline
15 & cDC2 (SIRPa+) & Conventional dendritic cells expressing SIRPa & (CD11c \textgreater\ 0.25) \& (PDCA1 \textgreater\ 0.25) & 11 & - \\
\hline
16 & DC (CD11c+) & Remaining dendritic cells expressing CD11c & CD11c \textgreater\ 0.25 & 11 & - \\
\hline
17 & SP4 (CD69+) & SP4 T-cells expressing CD69 & (CD8a $\le$ 0.2) \& (CD4 \textgreater\ 0.2) \& ((TCRb \textgreater\ 0.2) $\mid$ (CD5 \textgreater\ 0.15)) \& (CD69 \textgreater\ 0.25) & - & Medulla \\
\hline
18 & SP4 (CD69-) & SP4 T-cells not expressing CD69 & (CD8a $\le$ 0.2) \& (CD4 \textgreater\ 0.2) \& ((TCRb \textgreater\ 0.2) $\mid$ (CD5 \textgreater\ 0.15)) \& (CD69 $\le$ 0.25) & - & Medulla \\
\hline
19 & Postselection DP (CD69+) & DP T-cells that have undergone selection and express CD69 & (CD4 \textgreater\ 0.2) \& (CD8a \textgreater\ 0.2) \& ((CD5 \textgreater\ 0.3) $\mid$ (TCRb \textgreater\ 0.2)) \& (CD69 \textgreater\ 0.25) & - & Cortex \\
\hline
20 & Postselection DP (CD69-) & DP T-cells that have undergone selection and do not express CD69 & (CD4 \textgreater\ 0.2) \& (CD8a \textgreater\ 0.2) \& ((CD5 \textgreater\ 0.3) $\mid$ (TCRb \textgreater\ 0.3)) \& (CD69 $\le$ 0.25) & - & Cortex \\
\hline
21 & Preselection DP & DP T-cells that have not undergone selection & (CD4 \textgreater\ 0.05) \& (CD8a \textgreater\ 0.05) \& ((CD5 $\le$ 0.3) $\mid$ (TCRb $\le$ 0.3)) \& (CD69 $\le$ 0.3) & - & Cortex \\
\hline
22 & mTEC (Aire+) & mTEC expressing Aire & Aire \textgreater\ 0.09 & 9 & - \\
\hline
23 & mTEC (Ly6g+) & mTEC expressing Ly6g & Ly6g \textgreater\ 0.09 & 14 & - \\
\hline
24 & mTEC (PIGR+) & mTEC expressing PIGR & PIGR \textgreater\ 0.18 & 11 & Medulla \\
\hline
25 & mTEC (K10+) & mTEC expressing K10 & K10 \textgreater\ 0.1 & - & - \\
\hline
26 & mTEC (CD49f+) & mTEC expressing CD49f & (K14 \textgreater\ 0.1 $\mid$ K8 \textgreater\ 0.1) \& (CD49f \textgreater\ 0.3) \& (PDPN \textless\ 0.15) & - & Medulla \\
\hline
27 & mTEC hi (K14+, CD86+) & mTEC expressing K14, CD86, and MHCII & (K14 \textgreater\ 0.1) \& (MHCII \textgreater\ 0.35) \& (CD86 \textgreater\ 0.2) & - & - \\
\hline
28 & mTEC hi (K8+, CD86+) & mTEC expressing K8, CD86, and MHCII & (K8 \textgreater\ 0.1) \& (MHCII \textgreater\ 0.35) \& (CD86 \textgreater\ 0.2) & - & Medulla \\
\hline
29 & TEC (PDPN+) & TEC expressing PDPN & (K14 \textgreater\ 0.1 $\mid$ K8 \textgreater\ 0.1 $\mid$ Np63 \textgreater\ 0.1) \& (PDPN \textgreater\ 0.15) \& (ERTR7 \textless\ 0.15) & 5 & - \\
\hline
30 & mTEC hi (K14+, CD86-) & mTEC expressing K14, MHCII, but not CD86 & (K14 \textgreater\ 0.1) \& (MHCII \textgreater\ 0.35) \& (CD86 $\le$ 0.2) & - & - \\
\hline
31 & mTEC hi (K8+, CD86-) & mTEC expressing K8, MHCII, but not CD86 & (K8 \textgreater\ 0.1) \& (MHCII \textgreater\ 0.35) \& (CD86 $\le$ 0.2) & - & Medulla \\
\hline
32 & mTEC lo (K14+) & mTEC expressing K14 but not MHCII & (K14 \textgreater\ 0.1) \& (MHCII $\le$ 0.35) & - & - \\
\hline
33 & mTEC lo (K8+) & mTEC expressing K8 but not MHCII & (K8 \textgreater\ 0.1) \& (MHCII $\le$ 0.35) & - & Medulla \\
\hline
34 & cTEC hi (Ly51+) & cTEC expressing Ly51 and MHCII & (K8 \textgreater\ 0.1) \& (MHCII \textgreater\ 0.35) \& (Ly51 \textgreater\ 0.2) & - & Cortex \\
\hline
35 & cTEC lo (Ly51+) & cTEC expressing Ly51 but not MHCII & (K8 \textgreater\ 0.1) \& (MHCII $\le$ 0.35) \& (Ly51 \textgreater\ 0.2) & - & Cortex \\
\hline
36 & cTEC hi (Ly51-) & cTEC expressing MHCII but not Ly51 & (K8 \textgreater\ 0.1) \& (MHCII \textgreater\ 0.35) \& (Ly51 $\le$ 0.2) & - & Cortex \\
\hline
37 & cTEC lo (Ly51-) & cTEC expressing neither Ly51 nor MHCII & (K8 \textgreater\ 0.1) \& (MHCII $\le$ 0.35) \& (Ly51 $\le$ 0.2) & - & Cortex \\
\hline
38 & Endo & Endothelial cells & CD31 \textgreater\ 0.2 & - & - \\
\hline
39 & Fb (PDPN+, ERTR7-) & Fibroblasts expressing PDPN but not ERTR7 & (PDPN $\ge$ 0.15) \& (ERTR7 \textless\ 0.15) \& (Np63 \textless\ 0.1) & - & - \\
\hline
40 & Fb (PDPN-, ERTR7+) & Fibroblasts expressing ERTR7 but not PDPN & (PDPN \textless\ 0.15) \& (ERTR7 $\ge$ 0.15) \& (Np63 \textless\ 0.1) & - & - \\
\hline
41 & Fb (PDPN+, ERTR7+) & Fibroblasts expressing both PDPN and ERTR7 & (PDPN $\ge$ 0.15) \& (ERTR7 $\ge$ 0.15) \& (Np63 \textless\ 0.1) & - & - \\
\end{longtable}

\FloatBarrier

\bibliographystyle{unsrt}
\bibliography{bibliography}

\end{document}